\documentclass{amsart}

\usepackage{graphicx}
\usepackage{amsmath}
\usepackage{amscd}
\usepackage{amsthm}
\usepackage{amsfonts}
\usepackage{amssymb}
\usepackage{hyperref}
\usepackage{verbatim}
\usepackage[all]{xy}

\newtheorem{thm}{Theorem}[section]
\newtheorem*{nonum}{Theorem}

\newtheorem{claim}[thm]{Claim}

\newtheorem{corollary}[thm]{Corollary}

\newtheorem*{thmMainTrue}{Theorem \ref{thm.main}}

\newtheorem*{thmMainLag}{Theorem \ref{t.Lag.Perm}}
\newtheorem*{corClassify}{Corollary \ref{cor.classify}}

\theoremstyle{definition}
\newtheorem{exam}[thm]{Example}
\newtheorem{defn}[thm]{Definition}
\newtheorem{rem}[thm]{Remark}
\theoremstyle{plain}

\newtheorem{lem}[thm]{Lemma}

\newtheorem{prop}[thm]{Proposition}

\newcommand{\term}[1]{\textbf{#1}}
\newcommand{\mtrx}[1]{\begin{matrix} #1 \end{matrix}}
\newcommand{\pmtrx}[1]{\begin{pmatrix} #1 \end{pmatrix}}
\newcommand{\cmtrx}[1]{\left\{\begin{matrix}#1\end{matrix}\right\}}
\newcommand{\arry}[2]{\begin{array}{#1} #2 \end{array}}
\newcommand{\tbl}[2]{\begin{tabular}{#1} #2 \end{tabular}}
\newcommand{\RHScase}[1]{\left\{\begin{array}{ll} #1 \end{array}\right.}

\newcommand{\CC}{\mathbb{C}}
\newcommand{\NN}{\mathbb{N}}
\newcommand{\RR}{\mathbb{R}}

\newcommand{\ZZ}{\mathbb{Z}}

\newcommand{\AAA}{\mathcal{A}}
\newcommand{\BBB}{\mathcal{B}}
\newcommand{\CCC}{\mathcal{C}}
\newcommand{\EEE}{\mathcal{E}}
\newcommand{\HHH}{\mathcal{H}}
\newcommand{\LLL}{\mathcal{L}}
\newcommand{\MMM}{\mathcal{M}}
\newcommand{\OOO}{\mathcal{O}}
\newcommand{\PPP}{\mathcal{P}}

\newcommand{\RRR}{\mathcal{R}}
\newcommand{\TTT}{\mathcal{T}}
\newcommand{\VVV}{\mathcal{V}}
\newcommand{\ZZZ}{\mathcal{Z}}

\newcommand{\eps}{\varepsilon}
\newcommand{\Del}{\Delta}

\newcommand{\dset}[1]{\{1,\dots, #1\}}

\newcommand{\inv}[1]{#1^{-1}}

\renewcommand{\AA}{\mathbb{A}}
\newcommand{\sS}{\mathbb{S}}

\newcommand{\perm}{\mathfrak{S}}

\newcommand{\irr}{\perm^0}

\newcommand{\RClass}{\RRR}

\newcommand{\Heven}{\HHH^{even}}
\newcommand{\Hodd}{\HHH^{odd}}
\newcommand{\Hhyp}{\HHH^{hyp}}
\newcommand{\Hnonhyp}{\HHH^{nonhyp}}

\newcommand{\mA}{\mathbf{A}}
\newcommand{\mB}{\mathbf{B}}
\newcommand{\mC}{\mathbf{C}}
\newcommand{\mD}{\mathbf{D}}
\newcommand{\mI}{\mathbf{I}}

\newcommand{\mS}{\mathbf{S}}

\newcommand{\mU}{\mathbf{U}}
\newcommand{\mV}{\mathbf{V}}
\newcommand{\mW}{\mathbf{W}}

\newcommand{\mO}{\mathbf{0}}

\newcommand{\ma}{\mathbf{a}}
\newcommand{\mb}{\mathbf{b}}
\newcommand{\mc}{\mathbf{c}}
\newcommand{\md}{\mathbf{d}}
\newcommand{\me}{\mathbf{e}}
\newcommand{\mf}{\mathbf{f}}
\newcommand{\mg}{\mathbf{g}}
\newcommand{\mmu}{\mathbf{u}}
\newcommand{\mv}{\mathbf{v}}
\newcommand{\mx}{\mathbf{x}}
\newcommand{\my}{\mathbf{y}}

\newcommand{\ee}{\mathbf{e}}
\newcommand{\vv}{\mathbf{v}}
\newcommand{\ww}{\mathbf{w}}

\newcommand{\mone}{\mathbf{1}}

\newcommand{\bigO}{\vec{\mathbf{o}}}
\newcommand{\bigo}{\mathbf{o}}

\newcommand{\zeros}{ 0 & \dots & 0 \\ \vdots & \ddots & \vdots \\ 0 & \dots & 0 }

\newcommand{\msp}{\mbox{ }}

\newcommand{\IET}{IET}

\newcommand{\rtt}{1}
\newcommand{\rtb}{0}
\newcommand{\rt}{\rtt}
\newcommand{\rb}{\rtb}
\newcommand{\reps}{\eps}
\newcommand{\ropp}{\tilde{\eps}}
\newcommand{\rsub}[1]{#1}
\newcommand{\sig}{\sigma}

\newcommand{\inx}{\cdot}

\newcommand{\EVEN}{\mU}
\newcommand{\ODD}{\mV}
\newcommand{\PAIR}{\mW}
\newcommand{\SPACE}{\mS}
\newcommand{\BLAH}{\mD}

\newcommand{\lA}{A}
\newcommand{\lB}{Z}
\newcommand{\LL}[2]{\mtrx{#1 \\ #2}}

\newcommand{\AB}{\LL{\lA}{\lB}}
\newcommand{\BA}{\LL{\lB}{\lA}}
\newcommand{\stdperm}[1]{\cmtrx{\AB & #1 & \BA}}

\newcommand{\spaN}{\mathrm{span}}

\newcommand{\dsum}[1]{\sum_{#1}}
\newcommand{\dusum}[2]{\sum_{#1}^{#2}}
\newcommand{\duprod}[2]{\prod_{#1}^{#2}}
\newcommand{\duS}[2]{\mathcal{S}_{#1}^{#2}}

\numberwithin{equation}{section}

\begin{document}


\title[Self-Inverses]{Self-Inverses, Lagrangian Permutations and Minimal Interval Exchange Transformations with Many Ergodic Measures}

\author[J. Fickenscher]{Jonathan Fickenscher}

\address{Department of Mathematics, Princeton University, Princeton, NJ~08542, USA}

\email{jonfick@princeton.edu}


\date{\today}

\begin{abstract}
\noindent Thanks to works by M. Kontsevich and A. Zorich followed by C. Boissy, we have a classification of all Rauzy Classes of any given genus. It follows from these works that Rauzy Classes are closed under the operation of inverting the permutation. In this paper, we shall prove the existence of self-inverse permutations in every Rauzy Class by giving an explicit construction of such an element satisfying the sufficient conditions. We will also show that self-inverse permutations are Lagrangian, meaning any suspension has its vertical cycles span a Lagrangian subspace in homology. This will simplify the proof of a lemma in a work by G. Forni. W. A. Veech proved a bound on the number of distinct ergodic probability measures for a given minimal interval exchange transformation. We verify that this bound is sharp by construcing examples in each Rauzy Class.
\end{abstract}

\maketitle

\pagestyle{headings}

\setcounter{tocdepth}{2}
\tableofcontents

\section{Introduction}
\label{chap.intro}
Interval exchange transformations (\IET s) are encoded by a permutation $\pi$ and length vector $\lambda$. In \cite{c.R79}, Rauzy introduces Rauzy induction, a first return map of an \IET\ on a specific subinterval. This induction takes one of two forms on the space of \IET s and therefore descends to two different maps on the set of permutations. Therefore permutations are divided into \term{Rauzy Classes}, minimal sets closed under the two types of induction maps. We dedicate Sections \ref{sec.IET} and \ref{sec.RC} to providing some basic background and well known results concerning both \IET s and Rauzy Classes.

From another direction, we consider the \term{moduli space of Abelian differentials}. By the zippered rectangle construction in \cite{c.Ve82}, Veech shows that a generic \IET\ is uniquely ergodic (a result independently proved by Masur in \cite{c.Mas82}). This construction establishes a relationship between an \IET\ and flat surfaces with oriented measured foliations. We present an equivalent construction, called a \term{suspension}, in Section \ref{sec.surface}. Using suspensions, we assign properties to a permutation $\pi$: its \term{signature} (see Definition \ref{def.signature}), which is related to the singularities of these suspensions, and its \term{type} (see Section \ref{sec.classify}), which represents any other necessary data from its suspensions. The crucial result in this section is the following:
\begin{corClassify}
    Every Rauzy class is uniquely determined by signature and type. So given Rauzy Class $\RClass$, if $\pi\in\irr$ has the same signature and type as $\RClass$, then necessarily $\pi\in\RClass$.
\end{corClassify}
\noindent This immediately follows from \cite{c.KZ} and \cite{c.B09}. In Sections \ref{sec.hyperelliptic} and \ref{sec.spin}, we discuss \term{hyperelliptic} surfaces and the \term{parity} of a surface's \term{spin structure}. These discussions give us the necessary information to determine a permutation's type.

In Equations (2.2) from \cite{c.Ve84_I}, Veech shows a definition of Rauzy Induction on permutations. It is clear from this definition that the map $\pi\mapsto\inv\pi$ conjugates one type of induction with the other. This relationship conjures two natural questions:
\begin{enumerate}
    \item Are Rauzy Classes closed under taking inverses?
    \item Do all Rauzy classes contain self-inverse permutations?
\end{enumerate}
\noindent The work leading up to Corollary \ref{cor.classify} in Section \ref{sec.classify} provides an affirmative to the first question: any suspensions of $\pi$ and $\inv\pi$ have the same signature and type and therefore $\pi$ and $\inv\pi$ belong to the same class. However, proving a positive result for the second question would naturally imply one for the first also. This work answers the second question.
\begin{thmMainTrue}
	Every (true) Rauzy Class contains a permutation $\pi$ such that $\pi=\inv{\pi}$.
\end{thmMainTrue}
\noindent In Section $\ref{sec.blocks}$, we form patterns of letters, or \term{blocks}, that we may use to construct a self-inverse $\pi$ such that $\pi\in\RClass$ by Corollary \ref{cor.classify}. This method follows in the spirit of \cite{c.Z08}. In that paper, Zorich constructs permutations with desired properties. He then shows that these permutations belong to the desired Rauzy Class in a fashion similar to Corollary \ref{cor.classify}.

We consider the topic of Lagrangian subspaces of suspensions in Section \ref{chapLag}. We call a permutation $\pi$ \term{Lagrangian} if the vertical trajectories of any suspension of an \IET\ $T=(\pi,\mone)$, where $\mone=(1,\dots,1)$, span a $g$-dimensional subspace in homology, where $g$ is the genus of $\pi$. We prove the following:
\begin{thmMainLag}
        Suppose $\pi\in\irr_\AAA$ is self-inverse. Then $\pi$ is Lagrangian.
\end{thmMainLag}
\noindent This theorem provides an alternative proof of Forni's Lemma 4.4 in \cite{c.For02}. We present this proof as Corollary \ref{cor.Lag.For}. In this lemma, Forni shows that the set of $q\in\HHH_g$ (the moduli space of Abelian differentials of genus $g$) such that
\begin{enumerate}
    \item The vertical trajectories of $q$ are (almost all) periodic,
    \item These trajectories span a $g$-dimensional subspace in homology,
\end{enumerate}
is a dense set in $\HHH_g$. Corollary \ref{cor.Lag.For} uses Theorem \ref{t.Lag.Perm} and the fact that the Teichm\"uller geodesic flow is generically dense in each connected component of $\HHH_g$. We further show that the permutations we construct in Section \ref{chapTrue} need only consider the transposition pairs (letters interchanged by the permutation) to form such a basis.

While we know that almost every \IET\ is uniquely ergodic (see again \cite{c.Ve82} and \cite{c.Mas82}), there do exist minimal \IET\ $T$ that admit more than one distinct ergodic probability measure. Before the result of unique ergodicity, Keane gave such an example in \cite{c.Kea77}. We discuss the necessary tools in Sections \ref{sec.theta} and \ref{sec.cone_of_measures} to produce results similar to Keane's example. In the latter section, we also give the upper bound on the number of such ergodic measures (a well known bound derived in \cite{c.Ve78}). In Section \ref{chap.measures}, construct \IET\ 's in every Rauzy Class that must have the maximum number of such probability measures, giving an explicit proof that this bound is indeed sharp. We use the self-inverse permutations constructed in Section \ref{chapTrue} to create our examples.

\subsection{Interval Exchange Transformations}\label{sec.IET}

    Let $\perm_d$ be the set of permutations on $\dset{d}$. $\pi\in\perm_d$ is \term{irreducible} if $\pi(\dset{k})=\dset{k}$ only when $k=d$. The set of all irreducible permutations on $\dset{d}$ is $\irr_d$. $\pi\in\perm_d$ is \term{standard} if $\pi(d)=1$ and $\pi(1)=d$. Note that a standard permutation is necessarily irreducible. If $\AAA$ is an alphabet of $d$ letters, then $(\pi_0,\pi_1)\in\perm_\AAA$ is a pair of bijections, $\pi_\eps:\AAA\to\dset{d}$. We say that $\pi = (\pi_0,\pi_1)$ if $\pi = \pi_1\circ \pi_0^{-1}$. A pair $(\pi_0,\pi_1)$ uniquely determines a $\pi\in\perm_d$, but it follows that for any alphabet $\AAA'$ of $d$ letters and bijection $\tau:\AAA'\to\AAA$,
	$$ \pi = (\pi_0,\pi_1)\in\perm_\AAA \iff \pi = (\pi_0\circ\tau,\pi_1\circ\tau) \in\perm_{\AAA'}.$$
    We will use this as a natural equivalence between pairs, and in the case above we say freely that $\pi=(\pi_0,\pi_1) = (\pi_0\circ\tau,\pi_1\circ\tau)$. Let $\irr_\AAA$ be the set of irreducible permutations on $\AAA$, or $(\pi_0,\pi_1)\in\perm_\AAA$ such that $\pi=(\pi_0,\pi_1)$ is irreducible.
	
    When we refer to a (sub)interval of $\RR$, we mean open on the right and closed on the left (i.e. of the form $[a,b)$, for some $a<b$). Let $\RR_+^\AAA$ be the cone of positive length vectors in $\RR^\AAA$. For $\lambda\in\RR_+^\AAA$, let $|\lambda|:=\sum_{\alpha\in\AAA}\lambda_\alpha$, $I:=[0,|\lambda|)$, and define subintervals $I_\alpha^\eps \subseteq I$, $\alpha\in\AAA$ and $\eps\in\{0,1\}$ as
	$$I_\alpha^\eps:= \left[\sum_{\{\beta\in\AAA: \pi_\eps(\beta)<\pi_\eps(\alpha)\}}\lambda_\beta,\sum_{\{\beta\in\AAA: \pi_\eps(\beta)\leq\pi_\eps(\alpha)\}}\lambda_\beta\right).$$
    \begin{defn}
	An Interval Exchange Transformation (\IET) $T=(\pi,\lambda)$, $\pi=(\pi_0,\pi_1)\in\irr_\AAA$ and $\lambda\in\RR_+^\AAA$, is the unique map $T: I \to I$ such that for each $\alpha\in\AAA$,
	    \begin{itemize}
		\item $T$ restricted to $I_\alpha^0$ is a translation.
		\item $T(I^0_\alpha)=I^1_\alpha$.
	    \end{itemize}
    \end{defn}
    \begin{rem}
	The use of the pair $(\pi_0,\pi_1)$ on an alphbet rather than simply $\pi\in\irr_d$ was developed in many papers, including for instance \cite{c.Ker85}, \cite{c.MMY05} and \cite{c.Bu06}.
    \end{rem}

	By convention, we label $\pi\in\perm_d$ as $\pi=(\pi^{-1}(1), \dots, \pi^{-1}(d))$ to indicate the ordering of the original subintervals after the \IET. We shall likewise denote $\pi=(\pi_0,\pi_1)$ by
        $$\pi=(\pi_0,\pi_1)=\cmtrx{\pi_0^{-1}(1)&\dots & \pi_0^{-1}(d)\\ \pi_1^{-1}(1)&\dots & \pi_1^{-1}(d)}$$
    indicating the orders of the subintervals before and after the application of $T$. Figure \ref{fig.IETyoccoz} shows an example of an \IET\ with permutation
		$$\pi = \cmtrx{a&b&c&d\\d&a&c&b}.$$
    We may associate a translation vector $\omega\in\RR^\AAA$ to $T=(\pi,\lambda)$ by
		$$ T(x)=x+\omega_\alpha,~~x\in I^0_\alpha.$$
    In this case $\omega$ can be described by a matrix $\Omega_{\pi}$ by $\omega=\Omega_{\pi}\lambda,$ where
    \begin{equation}\label{eq.omega_pi2}
        (\Omega_{\pi})_{\alpha,\beta}=\left\{\begin{array}{rl} 1,& \mbox{if }\pi_0(\alpha)<\pi_0(\beta)~\&~\pi_1(\alpha)>\pi_1(\beta),\\
								-1,& \mbox{if }\pi_0(\alpha)>\pi_0(\beta)~\&~\pi_1(\alpha)<\pi_1(\beta),\\
								0,& \mathrm{otherwise.}\end{array}\right.
    \end{equation}
    \begin{rem}
        The matrix $\Omega_\pi$, $\pi=(\pi_0,\pi_1)$ is the same as the matrix $L^\pi$ seen in \cite{c.Ve78} and $M$ in \cite{c.R79}. For an example of the notation in Equation \eqref{eq.omega_pi2}, see \cite{c.Vi06}.
    \end{rem}
    \begin{figure}[h]
        \begin{center}
           \setlength{\unitlength}{250pt}
            \begin{picture}(1,.36)
                \put(0,0){\includegraphics[width=\unitlength]{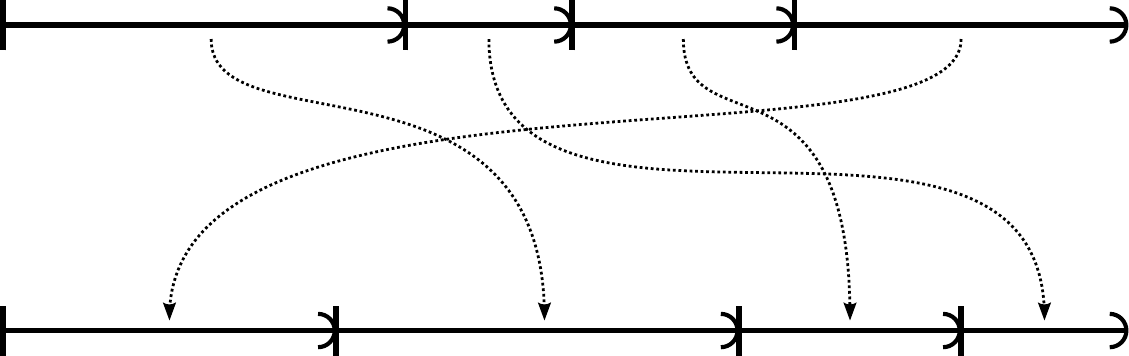}}

                \put(0.18,.33){$I^0_a$}
                \put(0.41,.33){$I^0_b$}
                \put(0.59,.33){$I^0_c$}
                \put(0.83,.33){$I^0_d$}

                \put(.13,-.04){$I^1_d$}
                \put(.46,-.04){$I^1_a$}
                \put(.73,-.04){$I^1_c$}
                \put(.90,-.04){$I^1_b$}
            \end{picture}
       \end{center}
       \caption{An \IET\ on $\AAA=\{a,b,c,d\}$.}\label{fig.IETyoccoz}
    \end{figure}

\subsection{Rauzy Classes}\label{sec.RC}

    In this section, we define a family of maps on irreducible permutations, known as Rauzy induction. Introduced in \cite{c.R79}, this is realized as a first return map of an \IET\ on appropriate subintervals. These moves partition each set $\irr_\AAA$ into equivalence classes under induction. We then state a relationship between induction and the map $\pi\mapsto\inv\pi$ in Claim \ref{cor.2}, observed by Veech.

    \begin{defn}\label{def.RV}
         Given an \IET\ $T=(\pi,\lambda)$, $\pi=(\pi_0,\pi_1)\in\perm_\AAA$, let $\alpha_\eps=\pi_\eps^{-1}(d)$, $\eps\in\{0,1\}$, denote the last letter on each row of $\pi$. Define  $I':=[0,|\lambda|-\min\{\lambda_{\alpha_0},\lambda_{\alpha_1}\})$. Then the first return map $T'$ of $T$ on $I'$ is an \IET\ with $T=(\pi',\lambda')$, $\pi'=(\pi_0',\pi_1')$, defined by the following rules:
        \begin{itemize}
            \item Assume $\lambda_{\alpha_0}>\lambda_{\alpha_1}$. We call this \term{Rauzy induction of type $\rtb$}. Then $\pi'=(\pi_0',\pi_1')$ is defined by the following rules:
                $$ \pi_0'=\pi_0,\mbox{ and }\pi_1'(\alpha)=\RHScase{
							\pi_1(\alpha), & \mbox{if }\pi_1(\alpha)\leq\pi_1(\pi_0^{-1}(d)),\\
							\pi_1(\alpha)+1, & \mbox{if }\pi_1(\pi_0^{-1}(d))<\pi_1(\alpha)<d,\\
							\pi_1(\pi_0^{-1}(d))+1, &\mbox{if }\pi_1(\alpha)=d,}$$
                or by the following diagram
                    $$ \xymatrix{{\pi=\cmtrx{\dots& & & \dots& \alpha_0\\
    			     \dots& \alpha_0& \beta&\dots&\alpha_1}}
    		          \ar[r]^\rtb &
    		      {\cmtrx{\dots& & & \dots& \alpha_0\\
    			             \dots& \alpha_0&\alpha_1& \beta&\dots}=\pi',}}$$
    	       and $\lambda$ is related to $\lambda'$ by
        	       $$ \lambda_\alpha'=\RHScase{
    					\lambda_{\alpha_0}-\lambda_{\alpha_1},&\mbox{if }\alpha=\alpha_0,\\
    					\lambda_{\alpha}, & \mathrm{otherwise.}}$$
            \item Now assume $\lambda_{\alpha_0}<\lambda_{\alpha_1}$. This is \term{Rauzy induction of type $\rtt$}. Then $\pi'=(\pi_0',\pi_1')$ is defined by the following rules:
            	$$ \pi_1'=\pi_1,\mbox{ and }\pi_0'(\alpha)=\RHScase{
							\pi_0(\alpha), & \mbox{if }\pi_0(\alpha)\leq\pi_0(\alpha_1),\\
							\pi_0(\alpha)+1, & \mbox{if }\pi_0(\alpha_1)<\pi_0(\alpha)<d,\\
							\pi_0(\alpha_1)+1, &\mbox{if }\alpha=\alpha_0,}$$
            	or by the following diagram
            	   $$ \xymatrix{{\pi=\cmtrx{\dots& \alpha_1& \beta&\dots&\alpha_0\\
            			\dots& & & \dots& \alpha_1}} \ar[r]^\rtt &
            	       {\cmtrx{\dots& \alpha_1&\alpha_0& \beta&\dots\\
            			\dots& & & \dots& \alpha_1}=\pi',}}$$
            	and $\lambda$ is related to $\lambda'$ by
            	   $$ \lambda_\alpha'=\RHScase{
			     		\lambda_{\alpha_1}-\lambda_{\alpha_0},&\mbox{if }\alpha=\alpha_1,\\
    					\lambda_{\alpha}, & \mathrm{otherwise.}}$$
        \end{itemize}
        We shall denote $\pi'$ as $0\pi$ or $1\pi$ if the induction was of type $0$ or $1$, respectively.
    \end{defn}

    \begin{rem}
        The case $\lambda_{\alpha_0}=\lambda_{\alpha_1}$ does not have a valid definition, as the resulting induced transformation is over $(d-1)$ symbols. However, such $\lambda$'s form a codimension one (therefore Lebesgue measure zero) set in $\RR_+^\AAA$.
    \end{rem}

    \begin{defn}\label{def.Keane}
        Assume $\pi\in\irr_\AAA$. Let $T=(\pi,\lambda)$ and $\partial I_\alpha$ denote the left endpoint of subinterval $I_\alpha^0$ for $\alpha\in\AAA$. Then $T$ satisfies the \term{Keane Condition} if
            \begin{equation}\label{eq.connexion}
                \underbrace{T\circ\cdots\circ T}_{m}(\partial I_\alpha) = T^m(\partial I_\alpha)\neq \partial I_\beta
            \end{equation}
        for all $m\geq 1$ and $\alpha,\beta\in\AAA$ such that $\pi_0(\beta)>1$.
    \end{defn}

    \begin{rem}
        Each violation of the Keane Condition satisfies an equality of Equation \eqref{eq.connexion} for a certain triple $(\alpha,\beta,m)$. However, each of these conditions is a codimension one hyperplane in $\RR_+^\AAA$. So given $\pi\in\irr_\AAA$, we see that the Keane property is satisfied for Lebesgue almost every $\lambda\in\RR_+^\AAA$.
    \end{rem}

    \begin{prop}\label{prop.Keane_is_inducible}
        Let $T^{(n)}$ denote the $n^{th}$ iteration of induction on \IET\ $T$. Then the following are equivalent:
        \begin{itemize}
            \item $T$ satisfies the Keane condition.
            \item $T^{(n)}$ is defined for all $n\geq 0$.
        \end{itemize}
    \end{prop}

	\begin{claim}
        Let $\eps\in\{0,1\}$. Then
            $$\pi\in\irr_\AAA \iff \eps\pi\in\irr_\AAA.$$
	\end{claim}
	
	\begin{proof}
        Suppose $\pi=(\pi_0,\pi_1)\in\perm_\AAA\setminus\irr_\AAA$, and fix type $\eps$. Then there is a proper subset $\AAA'\subset\AAA$, $\#\AAA=k<d$, such that $\pi_i(\AAA')=\{1,\dots,k\}$, $i\in\{0,1\}$. Most importantly, $\alpha_i=\pi_i^{-1}(d)\notin\AAA'$. So our induction must only move elements of $\AAA\setminus\AAA'$ as every element of $\AAA'$ appears before every element of $\AAA\setminus\AAA'$ in both rows. Namely $\eps\pi_i(\AAA')=\{1,\dots,k\}$ as well, or $\eps\pi\in\perm_\AAA\setminus\irr_\AAA$. The argument above applies if we first assume $\eps\pi\in\perm_\AAA\setminus\irr_\AAA$ and evaluate $\pi$, as $\eps^{m+1}\pi = \eps^{m}(\eps\pi)$ for some $m\geq 0$.
	\end{proof}
	So Rauzy induction is a closed operation in the set $\irr_\AAA$.
	\begin{defn}\label{def.rauzyclass}
        Given $\pi\in\irr_\AAA$, the \term{Rauzy Class} of $\pi$, $\RClass(\pi)\subseteq\irr_\AAA$, is the orbit of type $0$ and $1$ moves on $\pi$. The \term{Rauzy Graph} of $\pi$ is the graph with vertices in $\RClass(\pi)$ and directed edges corresponding to the inductive moves.
	\end{defn}
	
	\begin{exam}
        Consider permutation
	    $$\pi=\cmtrx{1&2&3\\3&2&1}.$$
	We have the following two other elements
            $$ 0\pi=\cmtrx{1 & 2 & 3 \\ 3 & 1 & 2}\mbox{, } 1\pi=\cmtrx{1 & 3 & 2 \\ 3 & 2 & 1}$$
	   The Rauzy Graph for $\RClass(\pi)$ is listed in Figure \ref{fig.RC321}.
    \end{exam}
    \begin{figure}[b]
	$$ \xymatrix{ {\cmtrx{1~3~2 \\ 3~2~1}} \ar@(dr,dl)^\rtb \ar@<1ex>[r]^\rtt& {\cmtrx{1~2~3 \\ 3~2~1}} \ar@<1ex>[l]^\rtt \ar@<1ex>[r]^\rtb & {\cmtrx{1~2~3 \\ 3~1~2}}
		\ar@(dr,dl)^\rtt \ar@<1ex>[l]^\rtb}$$
    \caption{The Rauzy Graph on $3$ symbols.}\label{fig.RC321}
    \end{figure}
	The given definition of a Rauzy Class is dependent on the choice of $\pi$, but the next claim shows that being in the
	same Rauzy Class is an equivalence condition and not dependent on our choice of representative.
	\begin{claim}\label{cor.1}
		For any $\pi^{(1)},\pi^{(2)}\in\RClass(\pi)$, there exists a directed path from $\pi^{(1)}$ to $\pi^{(2)}$ in
		the Rauzy Graph.
	\end{claim}
	\begin{proof}
        It suffices to show that for a permutation $\tilde{\pi}\in\RClass(\pi)$, there exists a path from $\reps\tilde{\pi}$ to $\tilde{\pi}$ for $\eps\in\{0,1\}$. By Definition \ref{def.RV}, there exists $n>0$ such that $\reps^n\tilde{\pi}=\tilde{\pi}$. So $n-1$ moves of type $\eps$ form a path from $\reps\tilde{\pi}$ to $\tilde{\pi}$.
	\end{proof}
    So if $\tilde{\pi}\in\RClass(\pi)$, then $\RClass(\pi)=\RClass(\tilde{\pi})$. The next result is used in Sections \ref{sec.hyperelliptic} and \ref{sec.hyperelliptic}.
    \begin{claim}\label{cor.std_in_RC}
        Every Rauzy Class $\RClass\subset\irr_\AAA$ contains a standard permutation (i.e. $\pi$ such that $\pi_0(\alpha) = \pi_1(\beta) =d$ and $\pi_0(\beta)=\pi_1(\alpha)=1$ for some $\alpha,\beta\in\AAA$).
    \end{claim}
    \begin{proof}
        Consider any $\pi\in\RClass$. Denote by $\alpha_\eps=\inv\pi_\eps(d)$ and $\beta_\eps=\inv\pi_\eps(1)$, for $\eps\in\{0,1\}$, or
        $$ \pi = \cmtrx{\LL{\beta_0}{\beta_1}~\LL{\dots}{\dots}~\LL{\alpha_0}{\alpha_1}}.$$ Let $n = \min\{\pi_0(\alpha_1), \pi_1(\alpha_0)\}$. Suppose $n=1$ and choose $\eps$ such that $\pi_{1-\eps}(\alpha_\eps)=1$. In this case, $\alpha_\eps = \beta_{1-\eps}$, and if we perform $m$ inductive moves of type $\eps$, the resulting permutation is standard, where $m=d-\pi_{1-\eps}(\beta_\eps)$. Consider the following diagram for $\eps=\rtb$:
        $$ \xymatrix{{\cmtrx{\LL{\beta_0}{\alpha_0} ~\LL{\dots}{\dots}~\LL{\dots}{\beta_0~\delta} ~\LL{\dots}{\dots}~\LL{\alpha_0}{\alpha_1}}} \ar^{\rtb^m}[r] &
            {\cmtrx{\LL{\beta_0}{\alpha_0} ~\LL{\dots}{\delta\dots}~\LL{\dots}{\alpha_1} ~\LL{\dots}{\dots}~\LL{\alpha_0}{\beta_0}}}}.$$
        If $n>1$ then we may fix $\eps\in\{0,1\}$ and find $\gamma\in\AAA$ such that $\pi_\eps(\gamma) < n < \pi_{1-\eps}(\gamma)$. If no such $\gamma$ exists, then $\inv{\pi_0}(\{n,\dots,d\})=\inv{\pi_1}(\{n,\dots,d\})$ and $\pi$ is not irreducible. So let $m=d-\pi_{1-\eps}(\gamma)$, and perform $m$ iterations of type $\eps$. Call this new permutation $\pi'$ and note that $\alpha'_{1-\eps}=\inv{\pi'}_{1-\eps}(d) = \gamma$ and $\alpha'_\eps = \inv{\pi'}_\eps(d)=\alpha_\eps$. Therefore $n' = \min\{\pi'_0(\alpha'_1), \pi'_1(\alpha'_0)\}<n$. Consider the following diagram for $\eps=\rtt$:
        $$ \xymatrix{\cmtrx{\LL{\dots}{\dots\gamma} ~\LL{\alpha_1\dots}{\dots\alpha_0} ~\LL{\gamma\dots}{\dots} ~\LL{\alpha_0}{\alpha_1}} \ar^{\rtt^m}[r] &
            \cmtrx{\LL{\dots}{\dots\gamma}~\LL{\alpha_1\dots\alpha_0}{\dots\alpha_0}~\LL{\dots}{\dots}~\LL{\gamma}{\alpha_1}}}.$$
        Repeat the above argument for $\pi'$ until $n'=1$ and we may derive a standard permutation.
    \end{proof}
    Consider one more observation that is used in Corollary \ref{cor.main}. This result is evident from Equations (2.2) in \cite{c.Ve84_I}.
	\begin{claim}\label{cor.2}
		For $\{\eps,\tilde\eps\}=\{0,1\}$ and $\pi\in\irr_\AAA$, $\reps\inv{\pi}=\inv{(\ropp\pi)}$.
	\end{claim}
	
	\begin{proof}
        We will show that $\inv{(\rb\inv\pi)}=\rt\pi$ as it will prove the claim for all cases. Let $\pi =(\pi_0,\pi_1)$.
        Then $\inv\pi = (\pi_1,\pi_0)$. Now let $\rb\inv\pi=(\pi_0^{\bullet},\pi_1^{\bullet})$. By Definition \ref{def.RV},
        $$ \pi_0^{\bullet} = \pi_1 \mbox{ and } \pi_1^{\bullet}(\alpha) =
            \RHScase{\pi_0(\alpha), & \mbox{if }\pi_0(\alpha)\leq\pi_0(\inv\pi_1(d)),\\
                \pi_0(\alpha)+1, & \mbox{if }\pi_0(\inv\pi_1(d)) < \pi_0(\alpha) < d, \\
                \pi_0(\inv\pi_1(d))+1, & \mbox{if }\pi_0(\alpha)=d.}$$
        Then $\inv{(\rb\inv\pi)} = (\pi_1^\bullet,\pi_0^\bullet)$. By checking Definition \ref{def.RV}, we conclude that $\rt\pi = (\pi_1^\bullet,\pi_0^\bullet) = \inv{(\rb\inv\pi)}$.
	\end{proof}

	So the action of taking the inverse permutation conjugates with the Ruazy moves on $\pi$ by sending them to
	the opposite move on $\pi^{-1}$.

    \begin{rem}\label{rem.RV_is_2to1}
         Let $T=(\pi,\lambda)$ for $\pi=(\pi_0,\pi_1)\in\irr_\AAA$ and $\lambda\in\RR_+^\AAA$. We also let $\alpha_\eps=\inv\pi_\eps(d)$ define the last letter of the rows of $\pi$. From the proof of Claim \ref{cor.1}, we can define $\pi^\eps$, for $\eps\in\{0,1\}$, such that $\reps\pi^\eps=\pi$. Also for $\eps\in\{0,1\}$, let $\lambda^\eps\in \RR_+^\AAA$ be defined by
             $$ \lambda^\eps_\alpha = \RHScase{ \lambda_{\alpha_0}+\lambda_{\alpha_1}& \mbox{if }\alpha=\alpha_\eps,\\ \lambda_\alpha& \mathrm{otherwise.}}$$
         It follows that if $T_\eps=(\pi^\eps,\lambda^\eps)$, then $T_\eps' = T$ and Rauzy Induction on $T_\eps$ is type $\eps$. So almost everywhere on the set $\irr_d\times\RR_+^\AAA$ (the set of all \IET's on $\AAA$), Rauzy induction is a $2$ to $1$ map.
    \end{rem}

\subsection{Suspended Surfaces for Interval Exchanges}\label{sec.surface}

    In \cite{c.Ve82}, Veech introduced the zippered rectangle construction, which allows us to associate to an \IET\ a flat surface with an Abelian differential. We present an equivalent construction, presented for example by Viana in \cite{c.Vi06}, of suspended surfaces over an \IET. We discuss Rauzy-Veech induction on these surfaces and introduce the moduli space of Abelian differentials.
	
	Fix $\pi=(\pi_0,\pi_1)\in\irr_\AAA$, $\pi_\eps:\AAA\rightarrow\dset{d}$, and $\lambda\in\RR_+^\AAA$. Let
    \begin{equation}\label{eq.Tpi}
        \TTT_\pi:=\left\{\tau\in\RR^\AAA: \sum_{\pi_0(\alpha)\leq k}\tau_\alpha>0,
            \sum_{\pi_1(\alpha)\leq k}\tau_\alpha<0\mbox{, for all }1\leq k < d\right\}.
    \end{equation}
    Define vectors $\vec{\zeta}_\alpha:=(\lambda_\alpha,\tau_\alpha)$ and segments $\zeta_\alpha^\eps$, $\eps\in\{0,1\}$, as the segment starting at $\sum_{\pi_\eps(\beta) <\pi_\eps(\alpha)}\vec{\zeta}_\beta$ and ending at $\sum_{\pi_\eps(\beta) \leq\pi_\eps(\alpha)}\vec{\zeta}_\beta$, noting that $\zeta_\alpha^0$ and $\zeta_\alpha^1$ are parallel (they are just translations of vector $\vec{\zeta}_\alpha$). Let $S:=S(\pi,\lambda,\tau)$ be the surface bounded by all $\zeta_\alpha^\eps$ with each $\zeta_\alpha^0$ and $\zeta_\alpha^1$ identified by translation. For example, if $\pi=(4,1,3,2)$, one suspension is given by Figure \ref{fig.suspend4132}.
    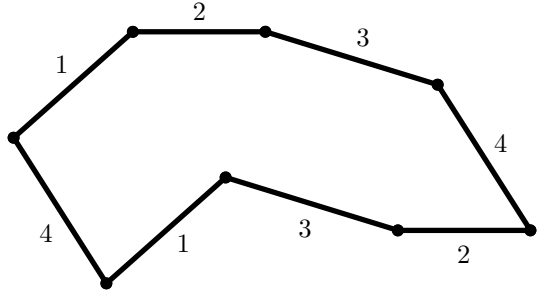
\begin{figure}[t]
        \begin{center}
           \setlength{\unitlength}{200pt}
            \begin{picture}(1,.57)
                \put(0,0){\includegraphics[width=\unitlength]{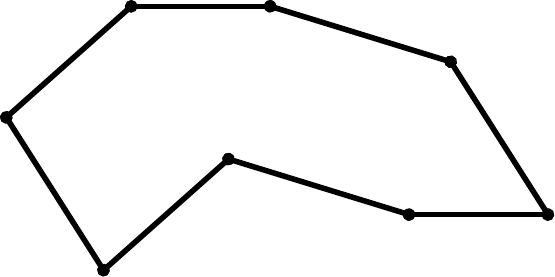}}

                \put(0.09,.41){$1$}
                \put(0.35,.51){$2$}
                \put(0.66,.46){$3$}
                \put(0.92,.26){$4$}

                \put(.06,.09){$4$}
                \put(.32,.07){$1$}
                \put(.55,.10){$3$}
                \put(.85,.05){$2$}
            \end{picture}
       \end{center}
        \caption{A suspension surface for $(4,1,3,2)$.}\label{fig.suspend4132}
    \end{figure}
    To avoid cumbersome notation, we denote the segments $\zeta_\alpha^\eps$ simply by $\alpha$ in a suspension. By definition, the leftmost endpoint is $(0,0)$. Define $I_S:=[0,|\lambda|)\times\{0\}$. With the exception of the points of discontinuity, the \IET\ $T=(\pi,\lambda)$ is realized by the first return of the positive vertical direction of $S$ on $I_S$, as is illustrated Figure \ref{fig.firstreturnisIET}.
    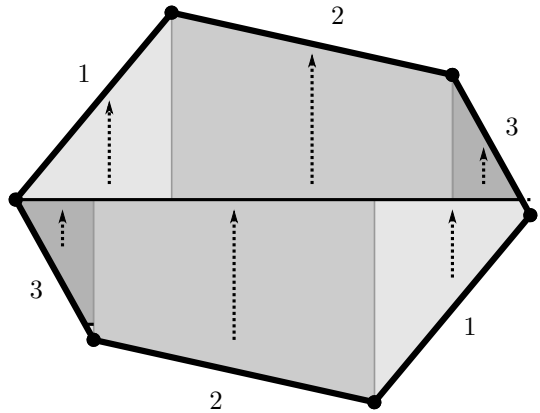
\begin{figure}[b]
        \begin{center}
           \setlength{\unitlength}{200pt}
            \begin{picture}(1,.80)
                \put(0,0){\includegraphics[width=\unitlength]{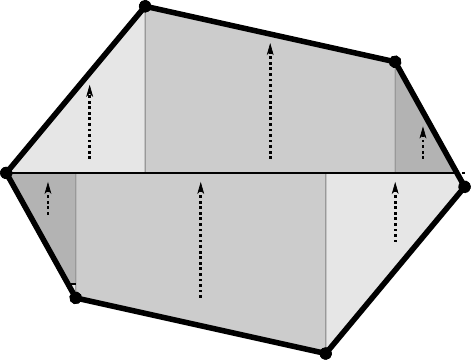}}

                \put(.13,.62){$1$}
                \put(.61,.73){$2$}
                \put(.94,.52){$3$}

                \put(.04,.21){$3$}
                \put(.38,.00){$2$}
                \put(.86,.14){$1$}
            \end{picture}
       \end{center}
        \caption{The first return of $I_S$ in the suspension is the original \IET.}\label{fig.firstreturnisIET}
    \end{figure}
			
    Each of the identifications on these surfaces is a translation. Therefore the standard form $dz$ in the polygon descends to a holomorphic $1$-form on the surface with zeroes, if any, at the vertex equivalence classes. Each vertex class is called a singularity of degree $k$, where $k$ is the degree of the corresponding zero of the differential and the total angle around the singularity is $2\pi(k+1)$.

    In order to give an explicit way to determine the degree of the singularities in a surface $S$, let us label the endpoints of our segments by $(\alpha,\eps,\imath)$, where $\alpha\in\AAA$, $\eps\in\{0,1\}$ and $\imath\in\{L,R\}$, to denote the left or right endpoint of segment $\zeta_\alpha^\eps$. We have the natural identification rules:
    \begin{enumerate}
        \item For $1\leq i < d$ and $\eps\in\{0,1\}$, $(\inv\pi_\eps(i),\eps,R)\sim(\inv\pi_\eps(i+1),\eps,L)$.
        \item $(\inv\pi_0(1),0,L)\sim(\inv\pi_1(1),1,L)$ and $(\inv\pi_0(d),0,R)\sim(\inv\pi_1(d),1,R)$.
        \item For $\alpha\in\AAA$ and $\imath\in\{L,R\}$, $(\alpha,0,\imath)\sim(\alpha,1,\imath)$.
    \end{enumerate}
    The equivalence sets determine the identified singularities in our surface $S$. The first rule lets us consider only the vertices of the form $(\alpha,\eps,L)$. With the exception of $(\inv\pi_0(1),0,L)$, every other vertex of this form has a downward direction in $S$. If a singularity is of degree $k$, it must have $k+1$ different vertices in its equivalence class, to ensure the total angle of $2\pi(k+1)$. Therefore, if we have $n$ vertices identified in our surface on the top (or bottom) row, it is a singularity of degree $n-1$.

    \begin{exam}
         A suspension of $(4,3,2,1)$ has one singularity in Figure \ref{fig.sing4321}, which has $3$ copies on the top row. So it is a singularity of degree $2$.
    \end{exam}
	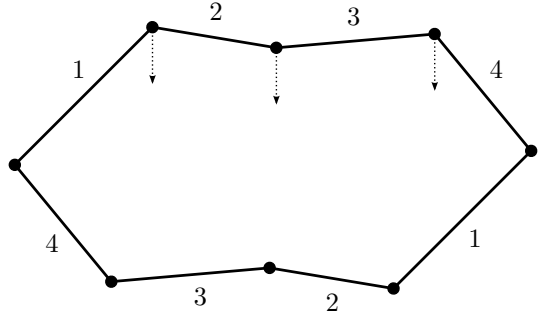
\begin{figure}[h]
        \begin{center}
           \setlength{\unitlength}{200pt}
            \begin{picture}(1,.58)
                \put(0,0){\includegraphics[width=\unitlength]{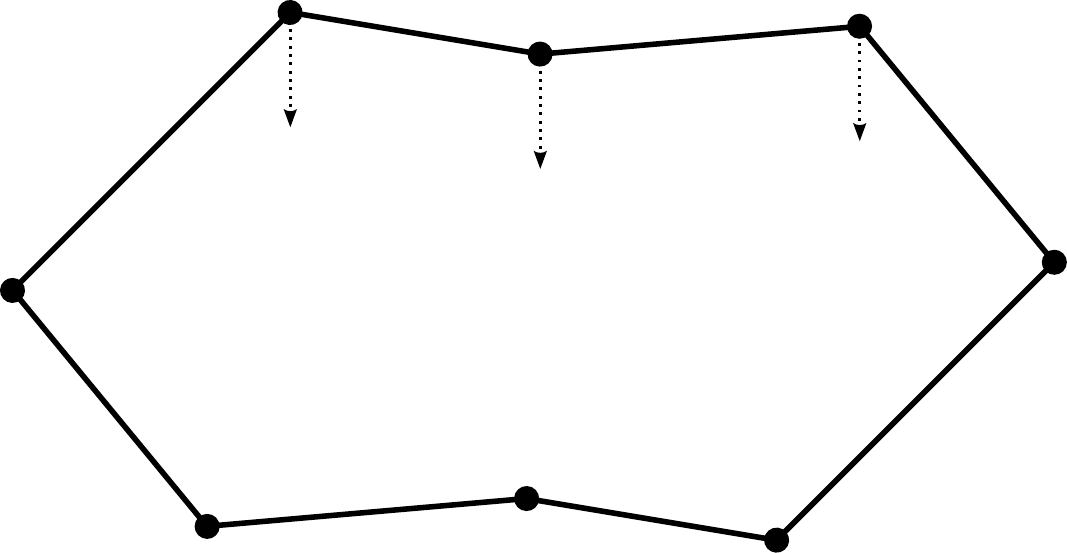}}

                \put(.12,.41){$1$}
                \put(.38,.52){$2$}
                \put(.64,.51){$3$}
                \put(.91,.41){$4$}

                \put(.07,.08){$4$}
                \put(.35,-.02){$3$}
                \put(.60,-.03){$2$}
                \put(.87,.09){$1$}
            \end{picture}
       \end{center}
    \caption{A suspension for $(4,3,2,1)$ has one singularity of degree $2$.}\label{fig.sing4321}
    \end{figure}
    Suppose our surface has $m$ singularities of degrees $\ell_1,\dots,\ell_m$. If $s=\dusum{i=1}{m} \ell_i$, the number of edges in our surface is $d=m+s+1$. So the Euler characteristic is $\chi(S)=m-(m+s+1)+1=-s$. The genus of the surface is then $g(S) = \frac{2-\chi(S)}{2} = 1+\frac{s}{2}$. The number and degrees of singularities do not depend on our choice of $\lambda$ or $\tau$, only on $\pi$. Therefore the genus of $\pi$, $g(\pi)$, is well defined.
    
    Rauzy induction may be extended to these surfaces as well and is called \term{Rauzy-Veech (R-V) induction}. Let $\eps\in\{0,1\}$ be such that $\lambda_{\alpha_\eps}>\lambda_{\alpha_{1-\eps}}$. We can define a new surface by $S'=S(\pi',\lambda',\tau')$ where $(\pi',\lambda')$ is defined as in Section \ref{sec.RC} and $\tau'$ is defined as
	\begin{equation}\label{eq.Tpi_RV}
         \tau_\alpha':=\RHScase{
				\tau_\alpha, & \mbox{if }\alpha\neq\alpha_\eps,\\
				\tau_\alpha-\tau_{\alpha_{1-\eps}}, & \mbox{if }\alpha=\alpha_\eps.}
    \end{equation}
    This procedure is a ``cut and paste" by translation from $S=S(\pi,\lambda,\tau)$ to $S':=S(\pi',\lambda',\tau')$, as shown for $\pi=(3,2,1)$ and induction type $\rtt$ in Figure \ref{fig.suspensionRV}.
    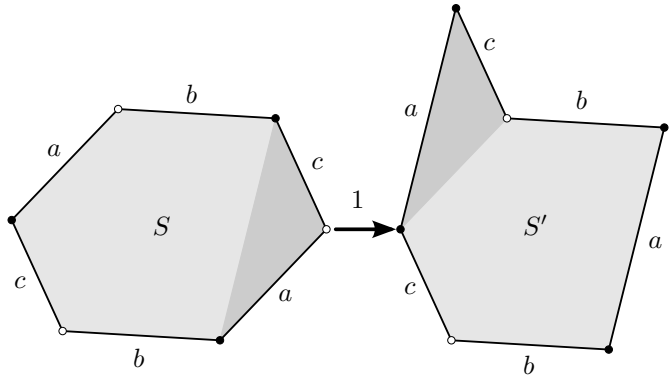
\begin{figure}
        \begin{center}
           \setlength{\unitlength}{250pt}
            \begin{picture}(1,.57)
                \put(0,0){\includegraphics[width=\unitlength]{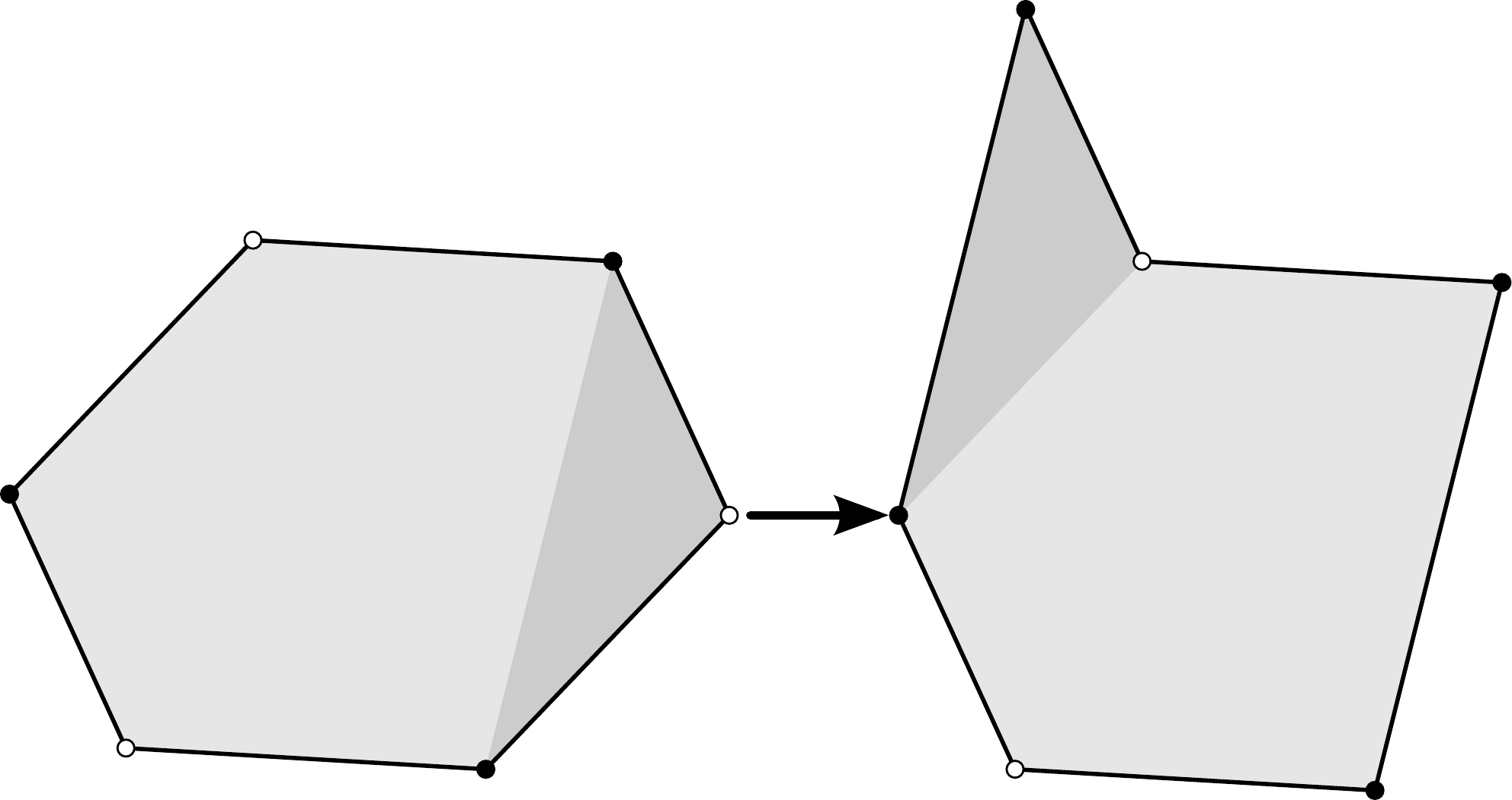}}

                \put(.22,.18){$S$}

                \put(.78,.18){$S'$}

                \put(.06,.30){$a$}
                \put(.27,.38){$b$}
                \put(.46,.28){$c$}

                \put(.01,.10){$c$}
                \put(.19,-.02){$b$}
                \put(.41,.08){$a$}

                \put(.60,.09){$c$}
                \put(.78,-.03){$b$}
                \put(.97,.16){$a$}

                \put(.60,.36){$a$}
                \put(.72,.46){$c$}
                \put(.86,.37){$b$}

                \put(.52,.22){$\rtt$}
            \end{picture}
       \end{center}
    \caption{A move of Rauzy-Veech induction on suspension $S$ for $\pi=(3,2,1)$.}\label{fig.suspensionRV}
    \end{figure}
    Note that, as opposed to the case of $(4,1,3,2)$, $\pi=(3,2,1)$ has two singularities of degree zero. The induced permutation, $\rt\pi$ (see Definition \ref{def.RV}), has the same number and degrees of singularities as $\pi$. This is a general fact.	
	\begin{prop}\label{prop.genus_is_inv}
		The number and degrees of singularities, and consequently the genus, are constant over a Rauzy Class.
	\end{prop}
	\begin{proof}
        This follows from counting before and after each type of inductive move to verify that the number and degrees of singularities do not change.
	\end{proof}
    While some singularities may be permuted by R-V induction, it is clear that the leftmost singularity remains fixed in the entire class. We shall call this singularity the \term{marked singularity}.
	\begin{defn}\label{def.signature}
        For a Rauzy Class $\RClass$, let $\ell_i$ denote the degrees of the $m$ singularities of $\pi$ with repetition. The $m$-tuple $(\ell_1,\dots,\ell_m)$, where $\ell_1$ is the degree of the marked singularity, is the \term{singularity signature} (or signature) of $\RClass$, denoted as $\sig=\sig(\RClass)$. If $\pi\in\RClass$, then $\sig(\pi)=\sig(\RClass)$.
	\end{defn}
    While the choice of $\ell_1$ is clear in Definition \ref{def.signature}, the other $\ell_i$'s may be in any order we wish. For example, the signature for $\pi=(8,3,2,4,7,6,5,1)$ can be written as $(1,1,2)$ or $(1,2,1)$ (see Figure \ref{fig._83247651}).
    
        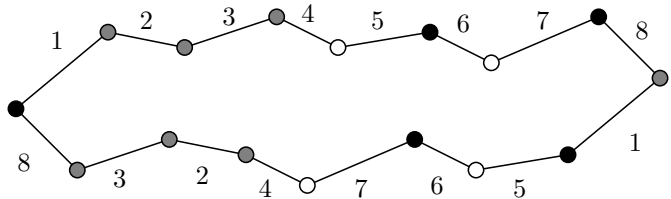
\begin{figure}[h]
        \begin{center}
           \setlength{\unitlength}{250pt}
            \begin{picture}(1,.3)
                \put(0,0){\includegraphics[width=\unitlength]{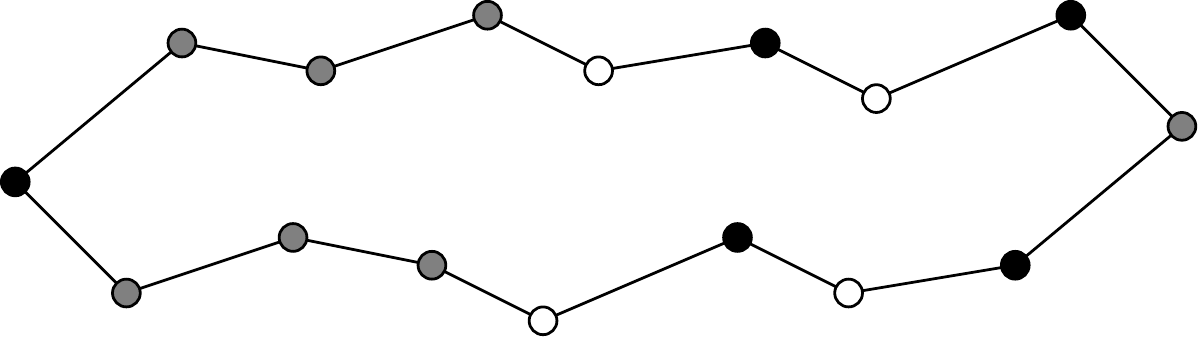}}

                \put(.065,.22){$1$}
                \put(.2,.25){$2$}
                \put(.325,.255){$3$}
                \put(.445,.26){$4$}
                \put(.55,.25){$5$}
                \put(.68,.245){$6$}
                \put(.8,.25){$7$}
                \put(.95,.235){$8$}

                \put(.015,.035){$8$}
                \put(.16,.01){$3$}
                \put(.285,.015){$2$}
                \put(.38,-.005){$4$}
                \put(.525,-.005){$7$}
                \put(.64,0){$6$}
                \put(.765,-.005){$5$}
                \put(.94,.065){$1$}
            \end{picture}
       \end{center}
       \caption{$\pi=(8,3,2,4,7,6,5,1)$ has signature $\sig(\pi)=(1,2,1)=(1,1,2)$.}\label{fig._83247651}
    \end{figure}

    Let $M_\AAA:=\{S(\pi,\lambda,\tau):\pi=(\pi_0,\pi_1)\in\irr_\AAA, \lambda\in\RR_+^\AAA, \tau\in\TTT_\pi, \mathrm{area}(S)=1\}$ minus the zero measure set where R-V induction is not well defined for all forward and backward iterates. Here the natural measure, $\mu$, is the product measure on $\irr_\AAA$, $\RR_+^\AAA$, and $\TTT_\pi$, where the first is a counting measure and the last two inherit Lebesgue measure from $\RR^\AAA$. Let $R$ denote the action of R-V induction on each $S\in M_\AAA$.
    \begin{rem}\label{rem.RV_is_now_1to1}
        Let $S=S(\pi,\lambda,\tau)$ and suppose $|\tau|=\dsum{\alpha\in\AAA}\tau_\alpha > 0$. Recall that $\alpha_\eps = \inv\pi_\eps(d)$. Consider $\pi^\rtt$ and $\lambda^\rtt$ from Remark \ref{rem.RV_is_2to1}. We see that $RS_\rtt=S$ for $S_\rtt=(\pi^\rtt,\lambda^\rtt,\tau^\rtt)$ where
            $$ \tau^\rtt_\alpha = \RHScase{\tau_{\alpha_0}+\tau_{\alpha_1}, & \mbox{if }\alpha = \alpha_\rtt,\\ \tau_\alpha, & \mathrm{otherwise.}}$$
        In this case the induction is type $\rtt$. Now let's attempt to construct $S_\rtb=S(\pi^\rtb,\lambda^\rtb,\tau^\rtb)$ such that $RS_\rtb=S$ by inductive move of type $\rtb$. We have $\pi^\rtb$ and $\lambda^\rtb$ as before. However Equation \eqref{eq.Tpi_RV} would require $\tau^\rtb$ to be defined by
            $$ \tau^\rtb_\alpha = \RHScase{\tau_{\alpha_0}+\tau_{\alpha_1}, & \mbox{if }\alpha = \alpha_\rtb,\\ \tau_\alpha, & \mathrm{otherwise.}}$$
        But then
            $$ \dsum{\alpha:\pi_\rtt(\alpha)\leq d-1}\tau^\rtb_\alpha = \dsum{\alpha:\alpha\neq \alpha_0,\pi_\rtt(\alpha)\leq d-1}\tau_\alpha+(\tau_{\alpha_0}+\tau_{\alpha_1})= |\tau| > 0.$$
        By Equation \eqref{eq.Tpi}, it follows that $\tau^\rtb \notin \TTT_{\pi^\rtb}$ (see Figure \ref{fig.RV_IS_1TO1}). $S_\rtt$ is therefore the unique suspension such that $RS_\rtt=S$. If instead $|\tau|<0$, we can similarly show that $S_\rtb$ exists while $S_\rtt$ does not. So we see that, as opposed to Rauzy induction on \IET's (see Remark \ref{rem.RV_is_2to1}), R-V induction is almost everywhere 1 to 1 on the space of suspensions.
    \end{rem}

    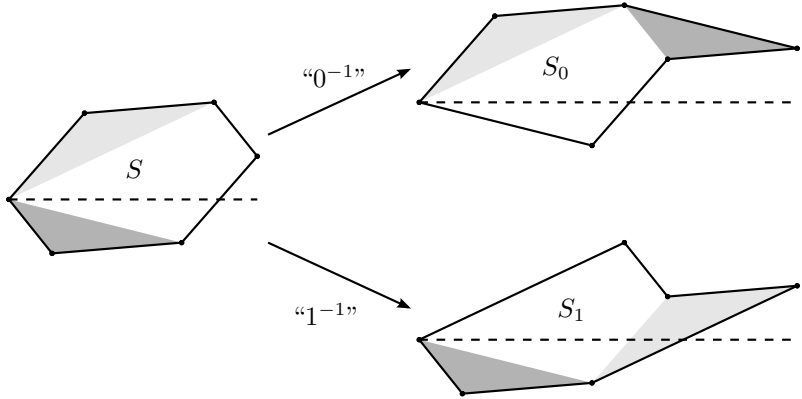
\begin{figure}[h]
        \begin{center}
           \setlength{\unitlength}{300pt}
            \begin{picture}(1,.5)
                \put(0,0){\includegraphics[width=\unitlength]{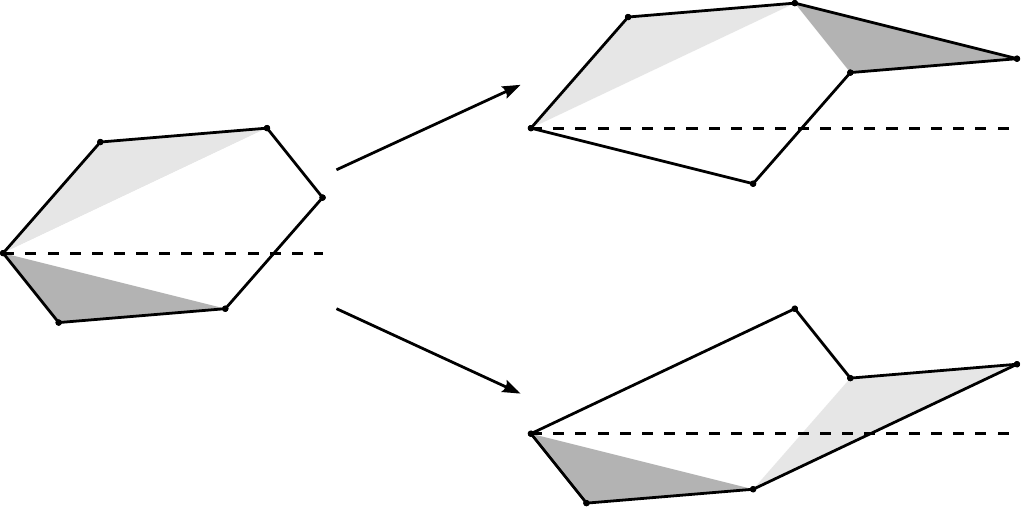}}
                \put(.15,.275){$S$}
                \put(.675,.4055){$S_\rtb$}
                \put(.695,.1){$S_\rtt$}
                \put(.37,.39){``$\rb^{-1}$"}
                \put(.36,.09){``$\rt^{-1}$"}
            \end{picture}
        \end{center}
        \caption{A suspension $S$ with $|\tau|>0$. $S_\rtb$ is not a valid suspension.}\label{fig.RV_IS_1TO1}
    \end{figure}

    Let the map $F_t:M_\AAA\rightarrow M_\AAA$ be the flow defined by
	   $$ F_t(\pi,\lambda,\tau)=(\pi,e^t\lambda,e^{-t}\tau)$$
    Denote by $M_\AAA^0$ the quotient space of $M_AAA$ under the equivalence $S\sim RS$ for $S\in M_\AAA$. Then a fundamental domain for $M_\AAA^0$ is
        $$\{ S(\pi,\lambda,\tau)\in M_\AAA: 1\leq|\lambda|\leq|\lambda'|^{-1}\}.$$
    The flow $F_t$ is well defined on $M_\AAA^0$ as $F_tR=RF_t$. There exists an $F_t$-invariant probability measure on $M_\AAA^0$ which is absolutely continuous with respect to $\mu$ (Veech \cite{c.Ve82}), and $F_t$ is ergodic on this space with respect to this measure. This action $F_t$ on $M_\AAA^0$ is called the \term{Teichm\"uller flow}. So each $M_\AAA^0$ has a natural mapping to the moduli space of Abelian differentials. Denote by $\HHH(\ell_1,\dots, \ell_m)$ the stratum of differentials with $m$ zeros of degrees $\ell_1,\dots,\ell_m$. As opposed to Rauzy classes, the ordering of the $\ell_i$'s is completely arbitrary in terms of the strata of Abelian differentials.

\subsection{The Matrix \texorpdfstring{$\Theta$}{Theta}}\label{sec.theta}

  Let $\AAA$ and $\pi=(\pi_0,\pi_1)\in\irr_\AAA$ be fixed with $d=\#\AAA$. Likewise, let $\RRR=\RRR(\pi)$ denote the Rauzy Class of $\pi$. Then we will consider the space $\RRR\times (\RR_+)^\AAA$ to be the space of all \IET's with permutations in $\RRR$. If we let
    $$ \Del_\AAA:= \{\lambda\in\RR_+^\AAA: |\lambda| = 1\},$$
  then $\RRR\times \Del_\AAA$ is the space of all such \IET's acting on the unit interval. Consider any $T=(\pi,\lambda)$ such that $T'=(\pi',\lambda')$ (its image under Rauzy induction) exists. In this case, let $\alpha$ be the ``winner'' and $\beta$ be the ``loser'' of this move (i.e. $\{\alpha,\beta\} = \{\pi_0^{-1}(d),\pi_1^{-1}(d)\}$ and $\lambda_\alpha>\lambda_\beta$, see Definition \ref{def.RV}). We then define a matrix $\Theta_{\pi,\lambda}$ in $\RR_+^{\AAA\times\AAA}$ by
    \begin{equation}\label{eq.Theta}
      \left(\Theta_{\pi,\lambda}\right)_{\zeta,\eta} = \RHScase{ 1, & \zeta=\eta \\ 1, & \zeta=\alpha,\eta = \beta, \\ 0, & \mbox{otherwise}.}
    \end{equation}
  We may now note the following relationship using Defintion \ref{def.RV},
    \begin{equation}\label{eq.Theta_Lambda}
      \lambda = \Theta_{\pi,\lambda} \lambda'.
    \end{equation}
  \begin{rem}
    The definition of $\Theta_{\pi,\lambda}$ only uses $\pi$ and $\lambda$ to determine $\alpha$ and $\beta$. As a result, we may use the following equivalent expressions for this matrix:
      $$ \Theta_{\pi,\lambda} = \Theta_{\pi,\eps} = \Theta_{\alpha,\beta}$$
    where $\eps$ is the type of inductive move on $T=(\pi,\lambda)$.
  \end{rem}
  It follows that, up to a zero measure set, $\RR_+^\AAA = \Theta_{\pi,0}\RR_+^\AAA \sqcup \Theta_{\pi,1}\RR_+^\AAA$.
  \begin{defn}
    Let $\RRR\subseteq\irr_\AAA$ be a Rauzy Class. A \term{finite Rauzy Path} $\gamma$ of length $N$ (or $|\gamma|=N$) is a sequence
      $$(\pi,\eps_1,\eps_2,\dots,\eps_{N}) \in\RRR\times\{0,1\}^{N}$$
    or equivalently (if $d>2$)
      $$ (\pi,\pi',\dots,\pi^{(N-1)},\pi^{(N)}) \in \RRR^{N+1}$$
    where $\pi^{(i)} = \eps_i\pi^{(i-1)}$ for $i\in\{1,\dots,N\}$. An \term{infinite Rauzy Path} $\gamma$ is similarly an element of $\RRR\times\{0,1\}^{\NN}$ (or equivalently an element of $\RRR^\NN$ when $d>2$). Let $T=(\pi,\lambda)$ satisfy the Keane Condition, so the $n^{th}$ step of induction exists for all $n\geq 0$. Then the \term{Rauzy Path} $\gamma$ of $T$  is the infinite Rauzy Path that begins at $\pi$ and $\eps_i$ is the type of move from $T^{(i-1)}$ to $T^{(i)}$ for each $i\in\NN$.
  \end{defn}
  Let $\gamma$ be a finite path in Rauzy Class $\RRR$ of length $N$, let
  \begin{equation} 
    \Theta_\gamma := \Theta^{(1)}_{\gamma}\Theta^{(2)}_{\gamma}\cdots\Theta^{(N)}_{\gamma},\mbox{ where }\Theta^{(i)}_{\gamma} := \Theta_{\pi^{(i-1)},\eps_i}, i\in\{1,\dots,N\}.
  \end{equation}
  \begin{rem}\label{rem.Path_And_Subset}
    For finite path $\gamma$ starting at $\pi$ of length $N$, the cone $\Theta_\gamma\RR_+^\AAA$ is precisely the set of $\lambda$'s such that
      \begin{itemize}
	\item $T=(\pi,\lambda)$ is inducible at least $N$ times,
	\item The move from $T^{(i-1)}$ to $T^{(i)}$ is type $\eps_i$ for all $i\in\{1,\dots,N\}$.
      \end{itemize}
    This leads us to conclude that, up to a zero measure set, there exists a partition
	$$ \RR_+^\AAA = \bigsqcup_{\gamma:|\gamma|=N} \Theta_{\gamma}\RR_+^\AAA$$
    for any $N$.
  \end{rem}
  \begin{defn}
		Suppose $A$ is a (component-wise) non-negative matrix in $\RR^\AAA$. The projected action $\hat{A}:\Del_\AAA\to \Del_\AAA$ is defined as
			$$ \hat{A}\lambda = \frac{A\lambda}{|A\lambda|}$$
		for each $\lambda\in\Del_\AAA$.
  \end{defn}
  \begin{prop}\label{prop.theta_and_measures}
      Let $\gamma$ be an infinite Rauzy Path beginning at $\pi\in\irr_\AAA$ and $\gamma_N$ finite path representing the first $N$ steps in $\gamma$ for $N\in\NN$. Let
      $$ \Lambda(\gamma) : = \bigcap_{N>0} \Theta_{\gamma_N}\RR_+^\AAA\mbox{, and }\Del(\gamma) : = \bigcap_{N>0} \hat\Theta_{\gamma_N}\Del_\AAA.$$
      Then the following are equivalent:
      \begin{enumerate}
	  \item $\Del(\gamma) \neq \emptyset$.
	  \item $\Lambda(\gamma) \neq \emptyset$.
	  \item There exists $T=(\pi,\lambda)$ satisfying the Keane Condition such that $\gamma$ is its Rauzy Path.
	  \item Let $\{\alpha_i\}_{i\in\NN}$ and $\{\beta_i\}_{i\in\NN}$, be the sequence of winners and losers respectively of the path $\gamma$. Then every letter $\eta\in\AAA$ appears each sequence infinitely often.
	  \item For every $k>0$, there exists $j=j(k)>k$ such that
	    $$ \Theta_{\gamma}^{(k)}\cdots\Theta_{\gamma}^{(j)}$$
	  is a positive matrix (a matrix with all positive entries).
      \end{enumerate}
  \end{prop}
  \begin{rem}
	  The proof that statement (3) implies (4) comes from Section 4.3 of \cite{c.Yoc03}. This particular presentation adapts the style found in Section 5 of \cite{c.Vi06}. The proof that statement (4) implies (5) comes from Lemma 3.4 in \cite{c.AvRe09}. The argument is relaxed, as we are proving a weaker result here.
  \end{rem}
  \begin{proof}
      The first two statements are equivalent, as
	$$ \lambda\in\Lambda(\gamma) \iff \frac{\lambda}{|\lambda|}\in\Del(\gamma).$$

      Statements (2) and (3) can be seen to be equivalent by Remark \ref{rem.Path_And_Subset}. Indeed, if $\lambda\in\Lambda(\gamma)$ then $\lambda\in\Theta_{\gamma_N}\RR_+^\AAA$ for any $N$. As a result $T=(\pi,\lambda)$ may be induced $N$ times and these moves follow $\gamma_N$. Let $N$ go to $\infty$. Conversely, suppose $T=(\pi,\lambda)$ satisfies the Keane Condition and has path $\gamma$. Then for all $N>0$, $\lambda\in\Theta_{\gamma_N}\RR_+^\AAA$. Therefore, $\lambda\in\Lambda(\gamma)$.

      Suppose statement (3) holds. We will then verify statement (4). Let $\eps_N,\alpha_N,\beta_N$ be the type, winner and loser respectively of the move from $T^{(N-1)}$ to $T^{(N)}$. We first note that the sequence $\{\eps_N\}_{N\in\NN}$ takes on values $0$ and $1$ infinitely often. If not, then for some $N_0$ and all $N>N_0$,
	$$ \alpha_N = \alpha_{N_0}.$$
      From Definition \ref{def.RV}, we see that
	$$ \lambda^{(N+1)}_{\alpha_{N_0}} = \lambda^{(N)}_{\alpha_{N_0}} - \lambda^{(N)}_{\beta_{N}},$$
      and as $\lambda^{(N)}_{\eta} = \lambda^{(N_0)}_{\eta}$ for all $\eta\neq\alpha_{N_0}$, this would imply that for some $N>N_0$ $\lambda^{(N)}_{\alpha_{N_0}}<0$, a contradiction.
      
      So as the type of induction move changes infinitely often, $\alpha_N=\beta_{N'}$ for some $N'>N$ ($N'$ in this case can be the minimum value greater than $N$ such that $\eps_{N'} \neq \eps_{N}$). So it suffices to prove that each $\eta\in\AAA$ appears in $\{\alpha_N\}$ infinitely often. Let $\BBB\subset \AAA$ be the letters that only appear as winners finitely often and $\CCC=\AAA\setminus\BBB$. We may first assume that by starting far enough into our sequence, each $\eta\in\BBB$ never wins. It follows now that each $\eta\in\BBB$ must only lose finitely many times. If $\eta\in\BBB$ lost infinitely often, it would lose to a particular letter $\zeta\in\CCC$ infinitely often. Because $\lambda^{(N)}_\eta>0$ is fixed for all $N>0$, $\lambda^{(N)}_{\zeta}<0$ for some $N$, a contradiction. Begin our sequence again far enough to assume that no letter in $\BBB$ ever wins or loses. So again by Definition \ref{def.RV},
	$$ \pi_{\eps}^{(N)}(\eta) \leq \pi_\eps^{(N+1)}(\eta),\mbox{ for }\eps\in\{0,1\}$$
      as the only time a letter would move ``backwards'' is when it loses. We may start our sequence even farther out and assume that
	$$\pi^{(N)}_\eps(\eta) = \pi_\eps(\eta)$$
      for all $\eta\in\BBB,\eps\in\{0,1\},N>0$. We now claim that
      \begin{equation}\label{eq.B_and_C_reducible}
	  \forall \eta\in\BBB,\zeta\in\CCC, \eps\in\{0,1\} ~\pi_\eps(\eta)<\pi_\eps(\zeta).
      \end{equation}
      Indeed suppose that $\pi_{\eps'}(\zeta) < \pi_{\eps'}(\eta)$. As $\zeta\in\CCC$, $\zeta = \alpha_N$ for some $N>0$. This would imply that $\eps_N=1-\eps'$ and therefore $\pi^{(N)}_{\eps'}(\eta) = \pi^{(N-1)}_{\eps'}(\eta) +1 > \pi_{\eps'}(\eta)$, a contradiction. But if Equation \eqref{eq.B_and_C_reducible} holds then
	$$ \pi_0(\BBB) = \pi_1(\BBB) = \{1,\dots, \#\BBB\},$$
      and as $\pi$ is irreducible, $\BBB=\emptyset$ and $\CCC=\AAA$, which proves (4).
     
      Now we will show that (4) implies (5). Fix $k>0$ and for any $j>k$, let
	$$ \Theta^{(k,j)}_\gamma = \Theta^{(k)}_\gamma\cdots\Theta^{(j)}_\gamma.$$
      Because each matrix on the right hand side is (component-wise) greater than or equal to the identity matrix on $\RR^{\AAA\times\AAA}$,
	$$ \left(\Theta^{(k,j)}_\gamma\right)_{\zeta,\eta} \leq \left(\Theta^{(k,j+1)}_\gamma\right)_{\zeta,\eta}$$
      for any $\zeta,\eta\in\AAA$ and $j\geq k$. In other words the $(\zeta,\eta)$-entry of the sequence of matrices $\{\Theta_\gamma^{(k,j)}\}_{j\geq k}$ is a non-decreasing. Fix $\zeta\in\AAA$ and define two sets $\PPP_\zeta\sqcup\ZZZ_\zeta=\AAA$ by
	\begin{align*}
	  \PPP_\zeta = \left\{\omega\in\AAA: \exists j_0>k~\left(\Theta_\gamma^{(k,j_0)}\right)_{\zeta,\omega}>0\right\},\\
	  \ZZZ_\zeta = \left\{\omega\in\AAA: \forall j_0>k~\left(\Theta_\gamma^{(k,j_0)}\right)_{\zeta,\omega}=0\right\}.
	\end{align*}
      We will assume that $\ZZZ_\zeta$ is not empty and arrive at a contradiction. Observe that $\zeta\in\PPP_\zeta$, so $\PPP_\zeta\neq \emptyset$. Because $\PPP_\zeta$ is a finite set, we may fix $k'$ such that the $(\zeta,\omega)$-entry of $\Theta^{(k,j)}_\gamma$ is greater than zero for all $j\geq k'$ and $\omega\in\PPP_\zeta$. As every letter wins infintely often, we must have a $j_0>k'$ such that $\alpha_{j_0}\in \ZZZ_\zeta$ and $\alpha_{j_0+1}\in\PPP_\zeta$ (i.e. the winners at steps $j_0$ and $j_0+1$ belong to the sets $\ZZZ_\zeta$ and $\PPP_\zeta$ respectively). For such a $j_0$, let $\beta:=\beta_{p_0+1} = \alpha_{p_0}$ and $\alpha:=\alpha_{j_0+1}$. So
	  $$ \left(\Theta^{(k,j_0+1)}_\gamma\right)_{\zeta,\beta} = \sum_{\omega\in\AAA}\left(\Theta^{(k,j_0)}_\gamma\right)_{\zeta,\omega}\left(\Theta^{(j_0+1)}_\gamma\right)_{\omega,\beta} \geq \left(\Theta^{(k,j_0)}_\gamma\right)_{\zeta,\alpha}\left(\Theta^{(j_0+1)}_\gamma\right)_{\alpha,\beta} >0$$
      as $\alpha\in\PPP_\zeta$, $j_0>k'$ and $\left(\Theta^{(j_0+1)}_\gamma\right)_{\alpha,\beta}=1$ by Equation \eqref{eq.Theta}. But then $\beta\in\ZZZ_\zeta\cap\PPP_\zeta = \emptyset$, a contradiction. As we may repeat this argument to see that $\PPP_\zeta=\AAA$ for all $\zeta$, statement (5) holds.

      We finally assume (5) and show that (1) is also true. For each $N$,
	$$ \hat\Theta_{\gamma_{N}}\Del_\AAA\subset\hat\Theta_{\gamma_{N+1}}\Del_\AAA.$$
      Choose a subsequence $0<N_1<N_2<N_3<\dots$ of $\NN$ such that $\Theta_\gamma^{(N_i)}\cdots\Theta_\gamma^{(N_{i+1}-1)}$ is a positive matrix for each $i\geq 1$. We then may say that
	$$\overline{\hat\Theta_\gamma^{(N_i)}\cdots\hat\Theta_\gamma^{(N_i-1)}\Del_\AAA} \subsetneq \Del_\AAA\mbox{, and } \hat\Theta_{\gamma_{N_i}}\Del_\AAA \supsetneq \overline{\hat\Theta_{\gamma_{N_{i+1}}}\Del_\AAA}\supsetneq\hat\Theta_{\gamma_{N_{i+1}}}\Del_\AAA$$
      All of the above allow us to conclude that
	$$ \Del(\gamma) = \bigcap_{i\geq 1} \hat\Theta_{\gamma_{N_i}}\Del_\AAA = \bigcap_{i\geq 1} \overline{\hat\Theta_{\gamma_{N_{i+1}}}\Del_\AAA}.$$
      The set on the right is nonempty, as it is an intersection of nested compact subsets of $\RR_+^\AAA$.
  \end{proof}

  \begin{defn}\label{def.complete_path}
      We will call an infinite Rauzy Path $\gamma$ \term{complete} if it satisies the equivalent conditions of Proposition \ref{prop.theta_and_measures}.
  \end{defn}

\subsection{The Cone of Invariant Measures}\label{sec.cone_of_measures}

  \begin{defn}
      For an \IET\ $T:I\to I$, denote by $\MMM(T)$ and $\MMM_1(T)$ the finite $T$-invariant and probability $T$-invariant measures on $I$ respectively. Likewise, let $\EEE(T)\subset \MMM(T)$ and $\EEE_1(T)\subset\MMM_1(T)$ denote the finite ergodic and probability ergodic measures of $T$.
  \end{defn}
  \begin{rem}\label{rem.Keane_minimal_cone} Assume $T$ satisfies the Keane condtion. As a result, $T$ is minimal (see \cite{c.Kea75}). The set $\MMM(T)$ has a structure of a positive cone. In other words $\mu,\nu\in\MMM(T)$ and $c>0$ imply that $\mu+\nu\in\MMM(T)$ and $c\cdot\mu\in\MMM(T)$. The set of extremal rays in $\MMM(T)$ is precisely $\EEE(T)$. Also, $\MMM_1(T)$ is convex ($t\cdot \mu + (1-t)\cdot\nu \in\MMM_1(T)$ for any $t\in[0,1]$ and $\mu,\nu\in\MMM_1(T)$) with $\EEE_1(T)$ its set of extremal points.
  \end{rem}
  In the following proposition, $\Omega_\pi$ is the anti-symmetric matrix given in Equation \ref{eq.omega_pi2}, and $g(\pi)$ is the genus of $\pi$, or the genus of any suspension of $\pi$ as indicated in the discussion before Proposition \ref{prop.genus_is_inv}.
  \begin{thm}[Veech 1978, 1982, 1984]\label{thm.Bound_on_Measures} Let $T=(\pi,\lambda)$ be an \IET\ that satisfies the Keane Condition. Then
      $$\#\EEE_1(T) \leq \frac{1}{2}\mathrm{rank}(\Omega_\pi) = g(\pi).$$
  \end{thm}
  The inequality is Theorem 0.5 in \cite{c.Ve78}. That $\frac{1}{2}\mathrm{rank}(\Omega_\pi) = g(\pi)$ may be deduced from a combination of Proposition 6.4 in \cite{c.Ve82} and Lemma 5.3 \cite{c.Ve84_I}.The following allows us to relate our cone of invariant measures $\MMM(T)$ with the cone $\Lambda(\gamma)\subset\RR_+^\AAA$ given in Proposition \ref{prop.theta_and_measures}. This is Lemma 1.5 in \cite{c.Ve78}.
  \begin{thm}\label{thm.LVector_is_Measure}(Veech 1978) Let $T$ satisfy the Keane Condition with Rauzy Path $\gamma$, then the map $\phi:\MMM(T)\to\RR_+^\AAA$ given by
      $$\phi(\mu)_\alpha = \mu(I_\alpha)$$
      is a bijection on its image $\Lambda(\gamma) = \cap_{N>0}\Theta_{\gamma_N}\RR_+^\AAA$. Namely
      $$ \MMM(T) \cong \Lambda(\gamma)$$
      by the isomorphism $\phi:\MMM(T) \to \Lambda(\gamma)$.
  \end{thm}
  \begin{rem}\label{rem.UE_is_generic}
      Consider the space $\irr_\AAA\times\RR_+^\AAA$ with its natural measure $\nu$, which is counting measure times Lebesque. Then for $\nu$-almost every $T=(\pi,\lambda)$
      \begin{itemize}
	  \item $T$ satisfies the Keane Condition, and
	  \item if $\gamma$ is the Rauzy Path of $T$, then there exists a positive matrix $B$ such that
	      $$\Theta_\gamma^{(j_i)}\cdots\Theta_\gamma^{(j_i+n)} = B$$
	      for an infinite increasing sequence $\{j_i\}_{i\in\NN}$ and fixed $n$.
      \end{itemize}
      It may be shown that in these cases, $\Lambda(\gamma)$ is a ray (or $\Del(\gamma)$ is a point). However this means that $T$ is uniquely ergodic by Theorem \ref{thm.LVector_is_Measure}. The result that unique ergodicity is generic was proved independently in \cite{c.Mas82} and \cite{c.Ve82}.
  \end{rem}

\subsection{Classification of Rauzy Classes}\label{sec.classify}

    Each stratum, $\HHH(\ell_1,\dots, \ell_m)$, can generally be divided further into connected components, which correspond to \term{Extended Rauzy Classes} (see \cite{c.Ve90}). The following theorems completely categorize every connected component for all strata. A stratum is \term{hyperelliptic} if a Riemann surface with differential in the stratum is hyperelliptic (see Section \ref{sec.hyperelliptic}). A stratum with all singularities of even degree has a flow invariant $\ZZ_2$-valued property called the \term{parity} of its \term{spin structure}. Details on this and calculations will be presented in Section \ref{sec.spin}.

    If the genus of $\RClass\subseteq\irr_d$ is $1$, we conclude from Sections \ref{sec.hyperelliptic} and \ref{sec.removable} that
		$$\pi=(d,2,\dots,d-1,1)$$
    belongs to $\RClass$. The following theorem categorizes all strata of genus $2$ and $3$.
	
	\begin{nonum}(M. Kontsevich and A. Zorich \cite{c.KZ})
		The moduli space of Abelian differentials on a complex curve of genus $g=2$ contains two strata: $\HHH(1,1)$ and
		$\HHH(2)$. Each of them is connected and hyperelliptic.
		
		Each stratum $\HHH(2,2)$, $\HHH(4)$ of the moduli space of Abelian differentials on a complex curve of genus $g=3$
		has two connected components: the hyperelliptic one, and one having odd spin structure. The other strata are
		connected for genus $g=3$.
	\end{nonum}
	The following theorem categorizes the connected components for each stratum of genera $4$ or greater.
	\begin{nonum}(M. Kontsevich and A. Zorich \cite{c.KZ})
		All connected components of any stratum of Abelian differentials on a complex curve of genus $g\geq 4$ are
		described by the following list:
		\begin{itemize}
            \item The stratum $\HHH(2g-2)$ has three connected components: the hyperelliptic one, $\Hhyp(2g-2)$, and components $\Heven(2g-2)$ and $\Hodd(2g-2)$ corresponding to even and odd spin structures.
            \item The stratum $\HHH(2\ell,2\ell)$, $\ell\geq 2$ has three connected components: $\Hhyp(2\ell,2\ell)$, $\Heven(2\ell,2\ell)$ and $\Hodd(2\ell,2\ell)$.
            \item All the other strata of the form $\HHH(2\ell_1,\dots,2\ell_m)$, where all $\ell_i\geq 1$, have two connected components: $\Heven(2\ell_1,\dots,2\ell_m)$ and $\Hodd(2\ell_1,\dots,2\ell_m)$.
			\item The strata $\HHH(2\ell-1,2\ell-1)$, $\ell\geq 2$, have two connected components; one of them,
				$\Hhyp(2\ell-1,2\ell-1)$, is hyperelliptic; the other one, $\Hnonhyp(2\ell-1,2\ell-1)$, is not.
			\item All other strata of Abelian differentials on complex curves of genera $g\geq 4$ are nonempty and
				connected.
		\end{itemize}
	\end{nonum}
    We are given a full classification of each connected component by the above results. To each connected component, we denote the \term{type} by the information other than the singularities. The type takes one of the following values \{-, even, odd, hyperelliptic, nonhyperelliptic\} as applicable. This however is not enough to calculate what Rauzy Class a permutation $\pi\in\irr$ belongs to, only the Extended Rauzy Class.
    \begin{exam}
        Consider
        $$ \pi= \cmtrx{\LL{a}{z}~\LL{b}{c}~\LL{c}{b}~\LL{d}{d}~\LL{e}{w}~\LL{f}{f}~\LL{w}{e}~\LL{x}{y}~\LL{y}{x}~\LL{z}{a}}$$
        and
        $$ \pi'=\cmtrx{\LL{a}{z}~\LL{b}{c}~\LL{c}{b}~\LL{d}{d}~\LL{e}{f}~\LL{f}{e}~\LL{w}{y}~\LL{x}{x}~\LL{y}{w}~\LL{z}{a}}.$$
        Both $\pi$ and $\pi'$ have three singularities one each of degrees $1$,$2$ and $3$. Therefore both $\pi$ and $\pi'$ belong to the stratum $\HHH(1,2,3)$. However the marked singularity of $\pi$ is of degree $3$, while the marked singularity of $\pi'$ is of degree $1$. Because the degree of the marked singularity is fixed throughout a Rauzy Class, $\RClass(\pi)\neq\RClass(\pi')$.
    \end{exam}
    It becomes clear that in order to distinguish Rauzy Classes, the degree of the marked singularity must be considered. Indeed, the following theorem shows the addition of this final invariant completes the classification of all Rauzy Classes:
	\begin{nonum}(C. Boissy \cite{c.B09})
		$\pi_1,\pi_2\in\irr_d$ belong to the same Rauzy class if and only if they belong to the same connected component
		and their marked singularities are the same degree.
	\end{nonum}
	We restate the above information in a different form and make an observation that, while clear from
	everything above, is crucial to our main result.
	\begin{corollary}\label{cor.classify}
		Every Rauzy class is uniquely determined by signature and type. So given Rauzy class $\RClass$,
		if $\pi\in\irr$
		has the same signature and type as $\RClass$, then necessarily $\pi\in\RClass$.
	\end{corollary}
	
\subsection{Hyperelliptic Surfaces}\label{sec.hyperelliptic}

    \begin{defn}\label{def.hyp_perm}
        A surface with quadratic differential $(M,q)$ of genus $g$ is \term{hyperelliptic} if there exists a map $h:M\to M$ such that
        \begin{itemize}
            \item $h=\inv{h}$,
            \item $h_*q = -q$,
            \item $h$ fixes $2g+2$ points,
        \end{itemize}
         and such an $h$ is called a \term{hyperelliptic involution}. A permutation $\pi$ is \term{hyperelliptic} if every suspension of $\pi$ is hyperelliptic.
    \end{defn}
    \begin{rem}
        Because $h$ is a well defined map on the differential, $h$ must take singularities to singularities. Also, $h$ must take geodesics to geodesics. Therefore $h$ maps saddle connections (geodesics with endpoints that are singularities) to saddle connections. In this case, removable singularities are not considered.
    \end{rem}
    \begin{rem}\label{rem.how_to_check_hyp}
        Consider $\pi\in\irr_\AAA$. Any given suspension $S$ is represented by a polygon in $\CC$ whose differential is represented by the standard $dz$ in its interior. In this case, the only possible candidates for a hyperelliptic involution on $S$ are of the local form $z\mapsto -z + c$ for some constant $c\in\CC$. These maps automatically satisfy the first two conditions in Definition \ref{def.hyp_perm}.
    \end{rem}
    \begin{defn}
        For $d\geq 2$, let $\pi_{(d)}$ be the permutation such that $\pi_{(d)}(i)=d-i+1$ for all $i\in\dset{d}$.
    \end{defn}
    \begin{lem}\label{lem.pi_d_is_hyp}
        $\pi_{(d)}$ is hyperelliptic.
    \end{lem}
    \begin{proof}
        Consider any suspension $S=S(\pi_{(d)},\lambda,\tau)$. We will construct $h$ and show that it satisfies Definition \ref{def.hyp_perm}. Let $h(z) = -z + |\lambda|+\imath|\tau|$. The first two conditions are satisfied. In order to show that $h$ is the appropriate map, we will define the vertices $p^\eps_k$ for $k\in\{0,\dots,d\}$ and $\eps\in\{0,1\}$ by
        $$ \begin{array}{rcl}
            p^0_k &=&\dusum{j=1}{k}\lambda_j+\imath\dusum{j=1}{k}\tau_j, \mbox{ and}\\
            p^1_k &=&\dusum{j=1}{k}\lambda_{d-j+1}+\imath\dusum{j=1}{k}\tau_{d-j+1}.
            \end{array}$$
        We note that the top segment labeled $k$ in $S$ has endpoints $p^0_{k-1}$ and $p^0_k$ while the bottom segment labeled $k$ in $S$ has endpoints $p^1_{d-k}$ and $p^1_{d-k+1}$. Because $h$ is an isometry, it maps segments to segments. We examine mapping the endpoints $\pi_k^\eps$ under $h$. For $k\in\{0,\dots,d\}$,
        $$ \begin{array}{rcl}
                h(p^0_k) & = & -p^0_k + |\lambda|+\imath|\tau| \\
                    & = & \dusum{j=1}{d}\lambda_j - \dusum{\ell=1}{k}\lambda_\ell + \imath\left(\dusum{j=1}{d}\tau_j - \dusum{\ell=1}{k}\tau_\ell\right)\\
                    & = & \dusum{j=k+1}{d}\lambda_j = \imath\dusum{j=k+1}{d}\tau_j \\
                    & = & \dusum{j'=1}{d-k}\lambda_{d-j'+1} + \imath\dusum{j'=1}{d-k}\tau_{d-j'+1}\\
                    & = & p^1_{d-k}.
            \end{array}$$
         Because $h=\inv h$, we conclude that any segment labeled $k$ is mapped to the other segment labeled $k$. As these segments are identified, $h$ fixes these segments. Now it remains to count the fixed points.

        If $d=2m$ is even, there is one singularity of degree $2m-2$, the genus is $m$ and there should be $2m+2 = d+2$ fixed points. There are $d$ segments each with a fixed midpoint. The point $\frac{1}{2}(|\lambda|+\imath|\tau|)$ is fixed, and the singularity represented by the class of all $p^\eps_k$'s is fixed. Therefore $h$ fixes $2g+2$ points.

        If $d=2m+1$ is odd, there are two singularities each of degree $m-1$, the genus is $m$ and there should be $2m+2=d+1$ fixed points. We note that this time, the two singularities, one represented by all $p_k^\eps$'s with even $k$'s and the other by all odd $k$'s, are interchanged by $h$. However $h$ fixes the $d$ midpoints of the labeled segments and the point $\frac{1}{2}(|\lambda|+\imath|\tau|)$. Therefore $h$ fixes $2g+2$ points.

        So we see that in either case, $h$ is the hyperelliptic involution for $S$.
    \end{proof}
    \begin{prop}\label{prop.hyp}
        Let $\pi\in\irr_d$ be standard with no removable singularities. If $\pi$ is hyperelliptic, then $\pi=\pi_{(d)}$.
    \end{prop}
    \begin{figure}[t]
        \begin{center}
           \setlength{\unitlength}{200pt}
            \begin{picture}(1,.42)
                \put(0,0){\includegraphics[width=\unitlength]{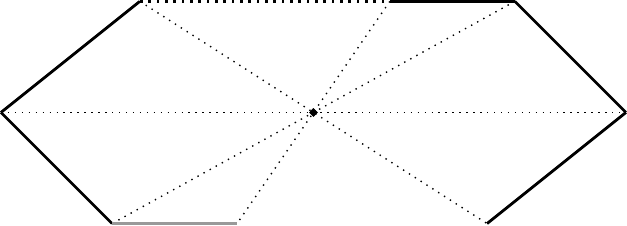}}

                \put(.09,.29){$1$}
                \put(.68,.38){$d-1$}
                \put(.92,.29){$d$}

                \put(.50,.13){$p$}

                \put(.06,.04){$d$}
                \put(.89,.04){$1$}
            \end{picture}
       \end{center}
       \caption{If the hyperelliptic involution must fix the point $p$, then the grey segment must be identified with $d-1$.}\label{fig.std_hyp}
    \end{figure}
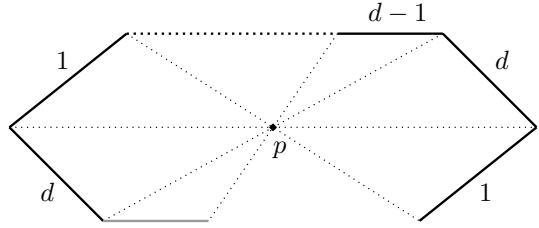
    \begin{proof}
        Let $\pi$ be a hyperelliptic standard permutation. We will explicitly construct a suspension for $\pi$ that excludes any possibility but $\pi=\pi_{(d)}.$ Assume $\AAA=\dset{d}$ and $\pi=(\pi_0,\pi_1)$ where $\pi_0(i)=i$. Fix any $m\in(0,\frac{1}{4})$ and let $\lambda\in\RR_+^d$ be defined by $\lambda_i = 1+m^i$ for $i\in\dset{d}$. Also, let $\tau$ be defined by $\tau = (1,0,\dots,0,-1)$. From Equation \eqref{eq.Tpi}, we see that $S=S(\pi,\lambda,\tau)$ is a valid suspension for $\pi$. Now consider the hyperelliptic involution for $S$, denoted as $h$. By construction, the segments labeled by $1$ and $d$ must be interchanged under $h$, as no other saddle connections exist of the appropriate length. We conclude that $h(z) = -z + |\lambda|+\imath|\tau| = |\lambda|-z$ and see that the top segment labeled $j$ must be mapped to the bottom segment labeled $j$, as no other saddle connection would have the appropriate length. We then show iteratively that $\pi(j) = d - j + 1$ as desired (see Figure \ref{fig.std_hyp}). By Lemma \ref{lem.pi_d_is_hyp}, this is hyperelliptic.
    \end{proof}

\subsection{Calculation of Spin Parity}\label{sec.spin}

	The results in this section follow from Appendix C in \cite{c.Z08}. We refer the reader to that paper for
	details.
	
    To each $\pi\in\irr_d$ with all singularities of even degree, we can define the parity of the spin structure of the corresponding suspension surface $S$. To do so, we must find a \term{symplectic basis} of $H_1(S)$. This is a choice of closed cycles $\alpha_1,\beta_1,\dots,\alpha_g,\beta_g\in H_1(S)$, $g=g(\pi)$, with the following conditions: $\alpha_i\inx\alpha_j=\beta_i\inx\beta_j=0$ and $\alpha_i\inx\beta_j=\delta_{ij}$ where $\alpha\inx\beta$ is the algebraic intersection number. For a loop $\gamma$, the Gauss map is the lift of $\gamma$ to the unit tangent bundle, a map from $H_1(S)\rightarrow \sS^1$ (where $\sS^1\subset\RR^2$ is the unit circle), and let $\mathrm{ind}(\gamma)$ be the degree of the Gauss map. The spin parity of the surface can be calculated by:
	\begin{equation}\label{eq.prespin}
	{
		\Phi(S) := \sum_{i=1}^{g}(\mathrm{ind}(\alpha_i)+1)(\mathrm{ind}(\beta_i)+1)~~(mod~2).
	}
	\end{equation}
    Let $\phi(\gamma):=\mathrm{ind}(\gamma)+1$. This value is independent of choice of suspension surface $S$. Therefore we may instead speak of the parity of $\pi$ itself. Using these conventions, the previous equation becomes
	\begin{equation}\label{eq.spin}
	{
		\Phi(\pi) :=\sum_{i=1}^{g}\phi(\alpha_i)\phi(\beta_i)~~(mod~2).
	}
	\end{equation}
    For a surface $S=S(\pi)$, we will define for each $i\in\dset{d}$ a loop $\gamma_i$. Start with any point on the embedded subinterval $I_i$ in $I_S$. The loop will move in the positive vertical direction until it returns to $I_S$. Then close the loop by a horizontal line. Now deform the loop continuously so it becomes smooth and everywhere transverse to the horizontal direction. Call this loop $\gamma_i$. Let $c_i=[\gamma_i]$ be the cycle representative of $\gamma_i$ in $H_1(S)$. See Figure \ref{fig.gamma}.
	\begin{figure}
        \begin{center}
           \setlength{\unitlength}{200pt}
            \begin{picture}(1,.46)
                \put(0,0){\includegraphics[width=\unitlength]{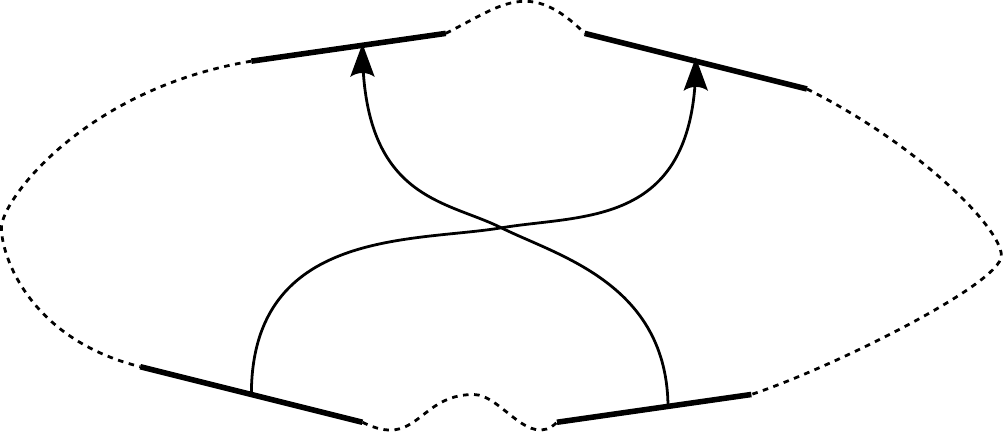}}

                \put(.34,.42){$i$}
                \put(.69,.42){$j$}

                \put(.22,-.03){$j$}
                \put(.66,-.04){$i$}

                \put(.30,.28){$\gamma_i$}
                \put(.70,.26){$\gamma_j$}

            \end{picture}
       \end{center}
        \caption{Loops $\gamma_i$ and $\gamma_j$ intersect in suspension $S$.}\label{fig.gamma}
    \end{figure}
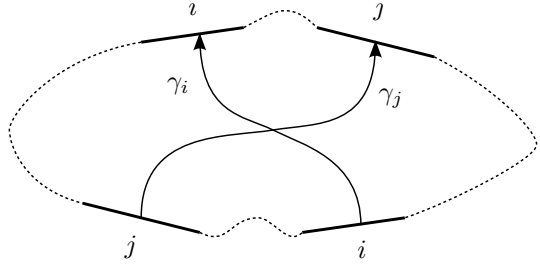
				
    It is clear that $\mathrm{ind}(\gamma_i)=0$ as $\gamma_i$ is always transverse to the horizontal direction. Therefore $\phi(c_i)=1$. From the definition of $\gamma_i$ and $\Omega=\Omega_\pi$ (from Section \ref{sec.IET}),
	\begin{equation}\label{eq.inx}
	{
		c_i\inx c_j = \Omega_{i,j}
	}
	\end{equation}
    and that the span of the $c_i$'s is $H_1(S)$. Because the above calculations \eqref{eq.prespin} and \eqref{eq.spin} are $(mod~2)$, the following calculations are over $\ZZ_2$. Note that now $\Omega$ is a symmetric matrix of zeros and ones. We may still keep the definition $\phi(c_i):H_1(S)\rightarrow\ZZ_2$. It is a well defined quadratic form on the intersection and has the following relationship as a direct result from \cite{c.J80}: for $c,c'\in H_1(S)$,
	\begin{equation}\label{eq.quadform}
		\phi(c+c')=\phi(c)+\phi(c')+c\inx c'.
	\end{equation}
	We recall the following relationship for $a,b,c\in H_1(S)$ on the intersection number:
	\begin{equation}\label{eq.addinx}
		(a+b)\inx c = a\inx c + b\inx c.
	\end{equation}
    We now describe the iterative process to choose our symplectic basis from the $c_i$'s. First let $\alpha_1:=c_1$. Let $\beta_1:=c_j$ for some $j$ such that $\Omega_{1j}=1$. We adjust each $c_i$, $i=2,3,\dots,j-1,j+1,\dots,d$ by the following rule: the remaining vectors must be adjusted so that they have trivial intersection number with $c_1$ and $c_j$. So we consider $c_i':=c_i+\eps_1c_1+\eps_jc_j$. Then $c_i'\inx c_1=0\Rightarrow\eps_j=c_i\inx c_1=\Omega_{1,i}$, and $c_i'\inx c_j=0\Rightarrow\eps_1=c_i\inx c_j=\Omega_{i,j}$. Now we use \eqref{eq.quadform} and \eqref{eq.inx} to calculate
	$$ \arry{rcl}{
			\phi(c_i') &=& \phi(c_i+\Omega_{i,j}c_1+\Omega_{1,i}c_j)\\
				&=& \phi(c_i)+\phi(\Omega_{i,j}c_1+\Omega_{1,i}c_j)\\
				&=& \phi(c_i)+\Omega_{i,j}\phi(c_1) + \Omega_{1,i}\phi(c_j) + \Omega_{i,j}\Omega_{1,i}.}$$
	And using \eqref{eq.addinx}, for $i,k\in\{2,3,\dots,j-1,j+1,\dots,d\}$,
	$$ \arry{rcl}{
			c_i'\inx c_k' &=& (c_i+\Omega_{i,j}c_1+\Omega_{1,i}c_j)\inx(c_k+\Omega_{k,j}c_1+\Omega_{1,k}c_j)\\
				&=& c_i\inx(c_k+\Omega_{k,j}c_1+\Omega_{1,k}c_j)+
					\Omega_{i,j}c_1\inx(c_k+\Omega_{k,j}c_1+\Omega_{1,k}c_j)\\
				& +& \Omega_{1,i}c_j\inx(c_k+\Omega_{k,j}c_1+\Omega_{1,k}c_j)\\
				&=& \Omega_{i,k}+\Omega_{k,j}\Omega_{1,i}+\Omega_{1,k}\Omega_{i,j}
				+ \Omega_{i,j}\Omega_{1,k}\\
				&+& \Omega_{i,j}\Omega_{1,k}
					+\Omega_{1,i}\Omega_{j,k}+\Omega_{1,i}\Omega_{k,j}\\
				&=& \Omega_{i,k}+\Omega_{k,j}\Omega_{1,i}+\Omega_{1,k}\Omega_{i,j}.}$$
	So we restate these results together for reference in later calculations,
    \begin{subequations}\label{eq.iterate}
	   \begin{align}
		  c_i'&:=c_i+\Omega_{i,j}c_1+\Omega_{i,1}c_j,\\
    		\phi(c_i')&:=\phi(c_i)+\Omega_{i,j}\phi(c_1)+\Omega_{i,1}\phi(c_j)+\Omega_{i,1}\Omega_{i,j},\\
    		c_i'\inx c_k'&:=\Omega_{i,k}+\Omega_{i,1}\Omega_{k,j}+\Omega_{i,j}\Omega_{k,1}.
	   \end{align}
	\end{subequations}
    We now have a new set of remaining cycles $c_i'$ with intersection matrix defined by $\Omega_{i,k}'=c_i'\inx c_k'$. We then pick a pair of intersecting cycles and name them $\alpha_2$ and $\beta_2$. We then alter the remaining cycles again by Equations \eqref{eq.iterate}. This process terminates when all pairs $\alpha_1,\beta_1,\dots,\alpha_g,\beta_g$ are chosen. Now we can calculate the parity by \eqref{eq.spin}.
    \begin{exam}
        We will calculate the spin parity of $\pi=(4,3,6,1,5,2)$. This is not hyperelliptic and has one singularity of degree $4$. So we first consider the initial conditions,
        $$ \Omega = \tbl{|r|cccccc|}{
            \hline
            & 1&2&3&4&5&6\\
            \hline
            1& 0&0&1&1&0&1\\
            2& 0&0&1&1&1&1\\
            3& 1&1&0&1&0&0\\
            4& 1&1&1&0&0&0\\
            5& 0&1&0&0&0&1\\
            6& 1&1&0&0&1&0\\
            \hline } \mbox{ and } \phi(c_i)=1.$$
        We may choose initial basis pair $(\alpha_1,\beta_1)=(c_1,c_3)$ and $\phi(\alpha_1)=\phi(\beta_1)=1$. We use the Equations \ref{eq.iterate} to derive
        $$ \Omega' = \tbl{|r|cccc|}{
            \hline
            & 2&4&5&6\\
            \hline
            2& 0&0&1&0\\
            4& 0&0&0&1\\
            5& 1&0&0&1\\
            6& 0&1&1&0\\
            \hline }
            \mbox{ for  }
            \arry{|rcl|rcl|}{
                \hline
                c_2' & = & c_1+c_2      &\phi(c_2') & = & 0\\
                c_4' & = & c_1+c_3+c_4  &\phi(c_4') & = & 0\\
                c_5' & = & c_5          &\phi(c_5') & = & 1\\
                c_6' & = & c_3+c_6      &\phi(c_6') & = & 0\\
                \hline } $$
        From these remaining vectors, we choose $(\alpha_2,\beta_2) =(c_2',c_5')$ where $\phi(\alpha_2)=0$ and $\phi(\beta_2)=1$. We then modify the remaining vectors to derive
        $$  \Omega'' = \tbl{|r|cc|}{
                \hline
                    & 4&6 \\
                \hline
                    4& 0&1 \\
                    6& 1&0 \\
                \hline }
            \mbox{ for }
            \arry{|rcl|rcl|}{
                \hline
                c_4'' & = & c_4'        & \phi (c_4'') &=& 0 \\
                c_6'' & = & c_2'+c_6'   & \phi (c_6'') &=& 0 \\
                \hline }$$
        Our only remaining choice is $(\alpha_3,\beta_3)=(c_4'',c_6'')$ with $\phi(\alpha_3)=\phi(\beta_3)=0$. So by Equation \eqref{eq.spin},
        $$ \Phi(\pi) = \dusum{i=1}{3}\phi(\alpha_i)\phi(\beta_i) = 1.$$
    \end{exam}

\section{Self-Inverses of Rauzy Class}\label{chapTrue}

	\begin{thm}\label{thm.main}
	(Main Result) Every (true) Rauzy Class contains a permutation $\pi$ such that $\pi=\inv{\pi}$.
	\end{thm}

    We prove this by using Corollary~\ref{cor.classify} to identify each special case of Rauzy Class. The hyperelliptic case is covered as the only standard element of each hyperelliptic class, $(d,d-1,\dots,2,1)$ as shown in Proposition \ref{prop.hyp}, is its own inverse. The remainder of the connected components shall be covered by Theorems~\ref{thm.g3},~\ref{thm.odds}, \ref{thm.even.odd}, \ref{thm.twos.even}, and \ref{thm.even.even}. The special case of component $\Hnonhyp(2\ell-1,2\ell-1)$ mentioned in Section \ref{sec.classify} is covered in Theorem \ref{thm.odds}, as the permutation constructed is not hyperelliptic so can apply to these components. Finally, singularities of degree zero are considered in Theorem \ref{thm.zero}.
	
    While the following fact can be deduced from the works \cite{c.B09} and \cite{c.KZ}, we can now state an alternate proof as a direct result.
	\begin{corollary}\label{cor.main}
		Every Rauzy Class is closed under taking inverses.
	\end{corollary}
	\begin{proof}
    For any Rauzy Class $\RClass$, we may pick $\pi'\in\RClass$ that is self inverse by Theorem \ref{thm.main}. Now choose any $\pi\in\RClass$. By Claim \ref{cor.1}, we may choose a series $\rsub{\eps_1}\dots \rsub{\eps_k}$, $\eps_i\in\{0,1\}$, such that $\pi=\rsub{\eps_k}\rsub{\eps_{k-1}}\dots \rsub{\eps_2}\rsub{\eps_1}\pi'$. By Claim \ref{cor.2},
		$$ \begin{array}{rcl}
			\inv{\pi}&=&\inv{(\rsub{\eps_k}\rsub{\eps_{k-1}}\dots \rsub{\eps_2}\rsub{\eps_1}\pi')}\\
				&=&\rsub{1-\eps_k}\rsub{1-\eps_{k-1}}\dots \rsub{1-\eps_2}\rsub{1-\eps_1}\inv{\pi'}\\
				&=&\rsub{1-\eps_k}\rsub{1-\eps_{k-1}}\dots \rsub{1-\eps_2}\rsub{1-\eps_1}\pi'.
		\end{array}$$
		 So $\pi^{-1}\in\RClass$ as $\pi'\in\RClass$.
	\end{proof}
    \begin{corollary}\label{cor.main2}
        In every connected component of every stratum $\HHH(\ell_1,\dots,\ell_m)$, there exists a differential that allows an order two orientation reversing linear isometry.
    \end{corollary}
    \begin{proof}
        In every connected component $\CCC\subset\HHH(\ell_1,\dots,\ell_m)$, consider a Rauzy Class $\RClass$ contained in $\CCC$. Choose self-inverse $\pi\in\RClass$ by Theorem \ref{thm.main}. Let $\mone=(1,\dots,1)\in\RR_+^\AAA$ and $\tau=(1,0,\dots,0,-1)\in\TTT_\pi$. Let $S=S(\pi,\mone,\tau)$ and note that $h(x,y)=(x,-y)$ satisfies the claim.
    \end{proof}

\subsection{Spin Parity for Standard Permutations}\label{sec.stdspin}

    When $\pi=(\pi_0,\pi_1)\in\irr_d$ is standard, the calculations mentioned in Section \ref{sec.spin} can be further refined. Just as in that section, the following calculations are over $\ZZ_2$. The matrix $\Omega=\Omega_\pi$ has the following form
	\begin{equation}
		\Omega = \pmtrx{	0& 1 & \dots&\dots & 1 & 1\\
					1& \mA_1& \mO& \dots &\mO& 1\\
					1& \mO& \mA_2 & \mO&\vdots & 1\\
					\vdots& \vdots&\ddots&\ddots&\ddots&\vdots\\
					1&\mO&\dots&\mO&\mA_p&1\\
					1&1&\dots&\dots&1&0}
	\end{equation}
    In other words, along the rows $1$ and $d$ and columns $1$ and $d$, the entries are all $1$ except entries $\Omega_{1,1}$ and $\Omega_{d,d}$ which are $0$. The interior $d-2$ by $d-2$ matrix is composed of $p$ square matrices, labeled $\mA_i$, along the diagonal with zeros otherwise. Note that $p=1$ is allowed. Each matrix $\mA_i$ corresponds to sub-alphabet $\AAA_i$ such that, for $\eps\in\{0,1\}$, $\pi_\eps(\AAA_i)=\{n_i,\dots,n_i+m_i-1\}$ where $n_i>1$ and $m_i=\#\AAA_i<d-2$.
    \begin{exam} Let $\pi=(7,3,2,6,5,4,1)$. We have that, for $p=2$
	$$ \Omega_\pi=\pmtrx{	0&1&1&1&1&1&1\\
				1&0&1&0&0&0&1\\
				1&1&0&0&0&0&1\\
				1&0&0&0&1&1&1\\
				1&0&0&1&0&1&1\\
				1&0&0&1&1&0&1\\
				1&1&1&1&1&1&0},~\mA_1=\pmtrx{0&1\\1&0},~\mA_2=\pmtrx{0&1&1\\1&0&1\\1&1&0}$$
    So if we write $\pi=\cmtrx{1&2&3&4&5&6&7\\7&3&2&6&5&4&1}$, we can assign to our $\mA_i$'s their corresponding blocks in $\pi$ as follows:
	$$ \mA_1\sim\cmtrx{2&3\\3&2},~\mA_2\sim\cmtrx{4&5&6\\6&5&4}.$$
    \end{exam}
    Because our definitions are invariant under renaming, we have a unique correspondence between a matrix of the form $\mA_i$ and its block in the permutation $\pi=(\pi_0,\pi_1)$.
    \begin{defn} \label{def.what_is_block}
        We allow $\mA_i$ to refer to the matrix and the block in $\pi$, and we shall denote $\mA_i$ as a \term{block} in either case.
    \end{defn}
    We now show the significance of this definition by showing how it aids in determining the spin parity for a standard permutation.
    Because $\Omega_{1,d}=1$, we may choose $\alpha_1=c_1$ and $\beta_1=c_d$. Recalling that $\phi(c_1)=\phi(c_d)=1$, the Equations \eqref{eq.iterate}, for $i,k=2,\dots,d-1$, are now
	\begin{subequations}\label{eq.iterate2}
		\begin{align}
			c_i'&:=c_i+c_1+c_d,\\
			\phi(c_i')&:=1+1+1+1=0,\\
			c_i'\inx c_k'&:=\Omega_{i,k}+1+1=\Omega_{i,k}.
		\end{align}
	\end{subequations}
	The new matrix $\Omega'$ over the remaining $c_i$'s becomes
	\begin{equation}
		\Omega'= \pmtrx{
					\mA_1& \mO& \dots &\dots&\mO\\
					\mO& \mA_2 & \mO&\dots&\vdots\\
					\vdots&\mO&\ddots&\ddots&\vdots\\
					\vdots& \vdots&\ddots&\mA_{p-1}&\mO\\
					\mO&\dots&\dots&\mO&\mA_p		},
	\end{equation}
	and \eqref{eq.spin} becomes
	\begin{equation}\label{eq.spin2}
		\Phi(\pi)=1+\sum_{i=2}^g\phi(\alpha_i)\phi(\beta_i).
	\end{equation}

    Next we notice that if $c_i$ belongs to the block associated to $\mA_j$, which we shall denote as $c_i\in\mA_j$, and $c_k$ is associated to $\mA_m$, $j\neq m$, then $c_i\inx c_k=0$. So for any $\alpha_i$ we select in a given $\mA_j$, $\beta_i$ must also belong to $\mA_j$ as well. So once a pair $\alpha_i,\beta_i\in\mA_j$ has been chosen, for any $c_k\in\mA_m$, $j\neq m$, then by \eqref{eq.iterate},
	\begin{subequations}\label{eq.iterate3}
		\begin{align}
			c_k'&:=c_k,\\
			\phi(c_k')&:=\phi(c_k)=0.
		\end{align}
	\end{subequations}
	Beginning with the initial data in \eqref{eq.iterate2}, we can calculate the value of
	\begin{equation}\label{eq.spinblock}
		\phi(\mA_i):=\sum_{j=2,~\alpha_j,\beta_j\in\mA_i}^{g}\phi(\alpha_j)\phi(\beta_j)
	\end{equation}
	for each $\mA_i$ independently over each other $\mA_j$. So we are lead to our final equation
	\begin{equation}\label{eq.spinf}
		\Phi(\pi)=1+\sum_{i=1}^p \phi(\mA_i).
	\end{equation}
    This final equation is a crucial part to Theorems \ref{thm.even.odd}, \ref{thm.twos.even} and \ref{thm.even.even}.

\subsection{Blocks}\label{sec.blocks}

	In this section, we define the necessary blocks to construct the permutations in the following sections.
	These blocks allow us to control the degrees of the singularities as well as the parity of spin.
	
	\begin{defn}\label{def.blocks}
		Let
		$$\SPACE:=\cmtrx{\LL{\alpha}{\alpha}}.$$
		For $n\geq 1$, let
            $$\EVEN_{2n}:=\cmtrx{ \LL{\alpha_1}{\beta_1} ~\LL{\beta_1}{\alpha_1} ~\LL{\alpha_2}{\beta_2} ~\LL{\beta_2}{\alpha_2} ~\LL{\dots}{\dots} ~\LL{\alpha_n}{\beta_n} ~\LL{\beta_n}{\alpha_n} }.$$
		For $n>1$, let
            $$ \ODD_{2n}:=\cmtrx{	\LL{\alpha_1}{\beta_1} ~\LL{\beta_1}{\alpha_1} ~\LL{\alpha_2}{\beta_2} ~\LL{\beta_2}{\alpha_2} ~\LL{\dots}{\dots} ~\LL{\alpha_{n-1}}{\beta_n} ~\LL{\beta_{n-1}}{\alpha_n} ~\LL{\alpha_n}{\beta_{n-1}} ~\LL{\beta_n}{\alpha_{n-1}}}.$$
		Let
            $$ \ODD_{2,2}:=\cmtrx{	\LL{\alpha_1}{\alpha_5} ~\LL{\alpha_2}{\alpha_4} ~\LL{\alpha_3}{\alpha_3} ~\LL{\alpha_4}{\alpha_2} ~\LL{\alpha_5}{\alpha_1}}.$$
		For $m,n\geq 0$, let
            $$ \PAIR_{2m+1,2n+1}:= \cmtrx{\LL{\alpha_1}{\beta_1} ~\LL{\beta_1}{\alpha_1} ~\LL{\dots}{\dots} ~\LL{\alpha_m}{\beta_m} ~\LL{\beta_m}{\alpha_m} ~ \LL{\eps}{\eta} ~\LL{\zeta}{\zeta} ~\LL{\gamma_1}{\delta_1} ~\LL{\delta_1}{\gamma_1} ~\LL{\dots}{\dots} ~\LL{\gamma_n}{\delta_n} ~\LL{\delta_n}{\gamma_n} }.$$
	\end{defn}
	
    For the remainder of the paper, when we speak of concatenating the blocks above, we assume that each block is defined over its own unique subalphabet. For example, $\PAIR_{2m+1, 2n+1}=\EVEN_{2m}\PAIR_{1,1}\EVEN_{2n}$ and $\EVEN_{2(m+n)}=\EVEN_{2m}\EVEN_{2n}$. Also note that all of these blocks contribute a self-inverse portion of a permutation. When we say a block appears \term{inside} a standard permutation $\pi$, we mean that it is a block in $\pi$ and does not include the letters on the outside, i.e. the letters $\pi_0^{-1}(1)=\pi_1^{-1}(d)$ and $\pi_1^{-1}(1)=\pi_0^{-1}(d)$.
    \begin{defn}\label{def.blockConstructed}
        A permutation $\pi\in\irr_d$ is \term{block-constructed} if it is standard and every block that appears inside comes from Definition \ref{def.blocks}. In other words, if $\pi$ is block constructed then
        $$ \pi=\stdperm{\mB_1\cdots\mB_k}$$
        where each $\mB_i$ is from Definition \ref{def.blocks}.
    \end{defn}
    We now show the desired properties of the defined blocks.
	\begin{lem}\label{lem.blocks}
        For the blocks in Definition \ref{def.blocks}, assuming they appear inside a standard permutation, the following are true:
		\begin{itemize}
			\item Assuming its leftmost (or rightmost) top and bottom singularities are identified,
				$\EVEN_{2n}$ contributes a singularity of degree $2n$ and $\phi(\EVEN_{2n})=0$.
			\item Assuming its leftmost (or rightmost) top and bottom singularities are identified,
				$\ODD_{2n}$, $n>1$, contributes a singularity of degree $2n$ and $\phi(\ODD_{2n})=1$.
			\item $\ODD_{2,2}$ contributes two singularities of degree $2$ and $\phi(\ODD_{2,2})=1$.
			\item $\PAIR_{2m+1,2n+1}$ contributes two singularities, one of degree $2m+1$ and one of degree $2n+1$.
		\end{itemize}
		Also, any block-constructed permutation is its own inverse.
	\end{lem}
	\begin{proof}
		We make note of a few relationships between our defined blocks:
		\begin{equation*}
			\begin{split}
				\EVEN_{2(n+1)}&=\EVEN_2\EVEN_{2n},\\
				\ODD_{2(n+1)}&=\EVEN_2\ODD_{2n},\\
				\PAIR_{2m+1,2n+1}&=\EVEN_{2m}\PAIR_{1,1}\EVEN_{2n}.
			\end{split}
		\end{equation*}
        We first prove the statement for $\EVEN_2$. This block has the structure in the suspended surface for $\pi$ as seen in Figure \ref{fig.blocks}. Because the two singularities are identified by assumption, this is one singularity of degree $2$. To calculate $\phi(\EVEN_2)$, we observe that the matrix associated to $\EVEN_2$ is just $\pmtrx{0&1\\1&0}$. So we choose $\alpha_1=c_1,\beta_1=c_2$ as our \term{canonical basis}. Using Equations \eqref{eq.iterate2} and \eqref{eq.spinblock},
		$$\phi(\EVEN_2)=\phi(\alpha_1)\phi(\beta_1)=0.$$
        \begin{figure}[h]
            \begin{center}
               \setlength{\unitlength}{350pt}
                \begin{picture}(1,.49)
                    \put(0,0){\includegraphics[width=\unitlength]{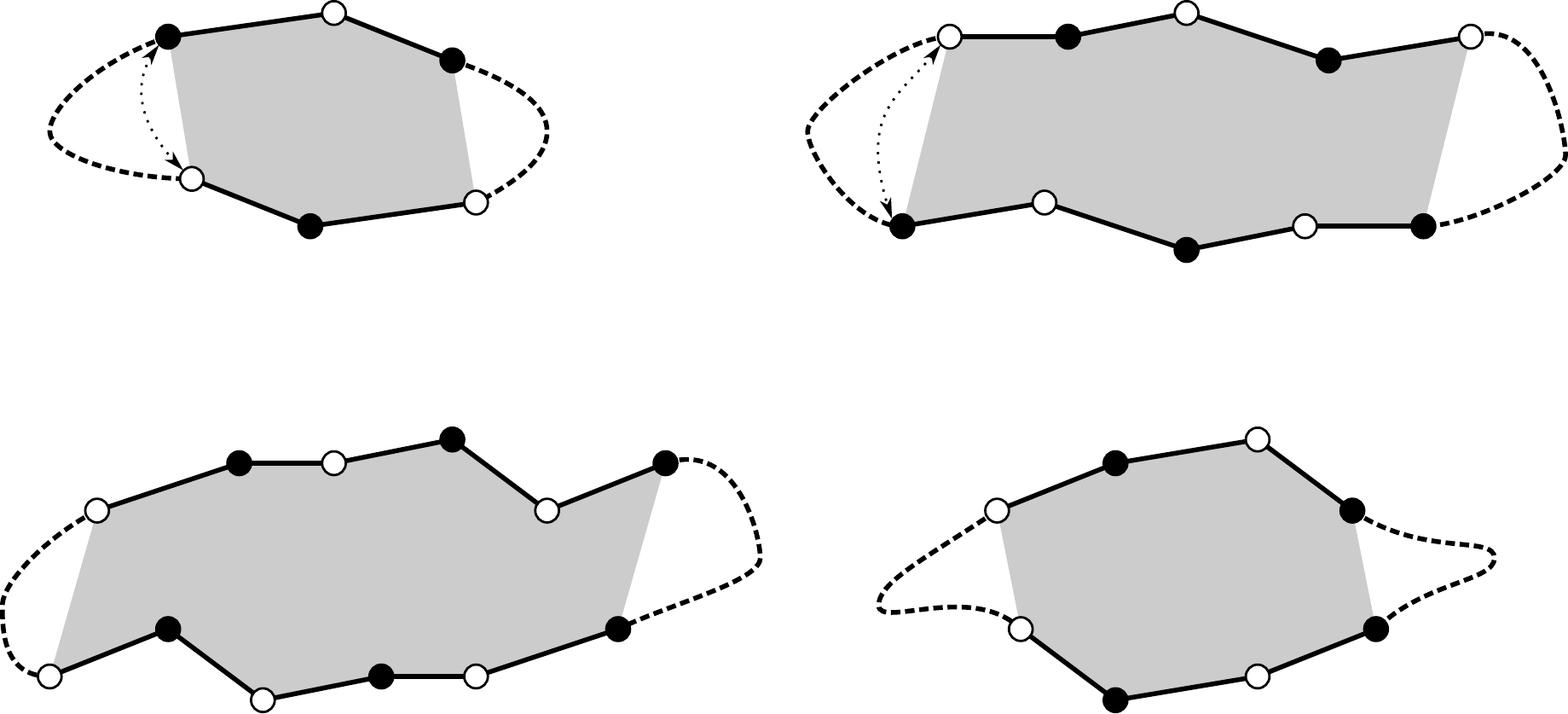}}

                    \put(.14,.46){$\alpha_1$} \put(.24,.46){$\beta_1$}
                    \put(.14,.29){$\beta_1$} \put(.24,.29){$\alpha_1$}
                    \put(.18,.37){$\EVEN_2$}

                    \put(.63,.46){$\alpha_1$} \put(.71,.47){$\beta_1$} \put(.79,.46){$\alpha_2$} \put(.87,.46){$\beta_2$}
                    \put(.61,.28){$\beta_2$} \put(.69,.28){$\alpha_2$} \put(.79,.26){$\beta_1$} \put(.86,.28){$\alpha_1$}
                    \put(.73,.36){$\ODD_{4}$}

                    \put(.08,.16){$\alpha_1$} \put(.16,.18){$\alpha_2$} \put(.235,.19){$\alpha_3$} \put(.32,.17){$\alpha_4$} \put(.38,.17){$\alpha_5$}
                    \put(.05,0){$\alpha_5$} \put(.11,0){$\alpha_4$} \put(.195,-.02){$\alpha_3$} \put(.26,-.01){$\alpha_2$} \put(.34,.01){$\alpha_1$}
                    \put(.19,.08){$\ODD_{2,2}$}

                    \put(.66,.15){$\eps$} \put(.75,.19){$\zeta$} \put(.83,.17){$\eta$}
                    \put(.66,.00){$\eta$} \put(.75,-.02){$\zeta$} \put(.84,.01){$\eps$}

                    \put(.73,.08){$\PAIR_{1,1}$}
                \end{picture}
            \end{center}
            \caption{The blocks $\EVEN_2$, $\ODD_4$, $\ODD_{2,2}$ and $\PAIR_{1,1}$ in suspensions. The two singularities in the top surfaces are identified by assumption in Lemma \ref{lem.blocks}.}\label{fig.blocks}
        \end{figure}
		
        For $\EVEN_{2n}$, see that this is nothing more than $n$-$\EVEN_2$ blocks concatenated, each block contributing a singularity of degree $2$. Because of concatenation, these singularities are all identified to form one of degree $2n$. To calculate $\phi(\EVEN_{2n})$, notice that its matrix is just $n$ $\EVEN_2$ blocks along the diagonal. So by our reasoning in Section \ref{sec.stdspin},
    		$$ \phi(\EVEN_{2n})=\sum_{i=1}^n \phi(\EVEN_2)=0.$$
		
        We now prove the claim for $\ODD_4$. This block has the structure in the surface as indicated by Figure \ref{fig.blocks}. As in the case for $\EVEN_2$, it is clear that this contributes a singularity of degree $4$. The matrix for $\ODD_4$ is
    		$$ \ODD_4=\pmtrx{0&1&1&1\\1&0&1&1\\1&1&0&1\\1&1&1&0}$$
        We now calculate $\phi(\ODD_4)$. Using again Equations \eqref{eq.iterate2} and \eqref{eq.spinblock} with choices $\alpha_1=c_1$, $\beta_1=c_4$, $\alpha_2=c_2'$, and $\beta_2=c_3'$,	
    		$$ \phi(\ODD_4)=\phi(\alpha_1)\phi(\beta_1)+\phi(\alpha_2)\phi(\beta_2)=0+1=1.$$

        For $\ODD_{2n}$, we note that it is $(n-2)$-copies of $\EVEN_2$ followed by $\ODD_4$. As above the degree of the singularity is then $(n-2)\cdot 2+4=2n$. The matrix of $\ODD_{2n}$ is $(n-2)$-$\EVEN_2$ blocks and one $\ODD_4$ block along the diagonal. So
    		$$ \phi(\ODD_{2n})=\sum_{i=1}^{n-2}\phi(\EVEN_2)+\phi(\ODD_4)=0+1=1$$
		
        We now prove the theorem for $\ODD_{2,2}$. This follows from the block's portion in the surface (see Figure \ref{fig.blocks}). The matrix for $\ODD_{2,2}$ is
    		$$ \ODD_{2,2}=\pmtrx{0&1&1&1&1\\1&0&1&1&1\\1&1&0&1&1\\1&1&1&0&1\\1&1&1&1&0}$$
		all of the statements in the theorem for $\ODD_{2,2}$ follow immediately as they have for the previous blocks.
		
        For $\PAIR_{2m+1,2n+1}$, it suffices to prove the statement for $\PAIR_{1,1}$, as $\PAIR_{2m+1,2n+1} = \EVEN_{2m}\PAIR_{1,1}\EVEN_{2n}$, making the singularities degrees $2m+1$ and $2n+1$ as desired. The portion of the surface determined by $\PAIR_{1,1}$ is shown in Figure \ref{fig.blocks}. So again, counting verifies that there are two singularities with degree $1$.
		
		The final statement of the theorem is clear as each block places all of its letters in self-inverse positions and the outside $\lA$ and $\lB$ letters (making the permutation standard) are in self-inverse position as well.
	\end{proof}
	Before proving the main theorem, we remark that the block $\SPACE$ is designed to keep singularities of neighboring blocks separate. To illustrate this point, notice that $\EVEN_{2n}\EVEN_{2m}$ contributes \term{one} singularity of degree $2(m+n)$, while $\EVEN_{2m}\SPACE\EVEN_{2n}$ contributes the desired \term{two} singularities. The block $\SPACE$ also causes any neighboring block's leftmost (or rightmost) top and bottom singularities to be identified, as required in Lemma \ref{lem.blocks}.

\subsection{Self-Inverses for \texorpdfstring{$g\leq 3$}{g<=3}}\label{sec.gleq3}

    \begin{thm}\label{thm.g3}
            Given $\tilde{\pi}\in\irr$ such that $g(\tilde{\pi})\leq 3$, There exists $\pi\in\RClass(\tilde{\pi})$ such that $\pi=\pi^{-1}$.
	\end{thm}

    We shall prove this result simply by stating such an element for each class, as listed in Figure \ref{fig.genus_three_below}. When possible, we use the construction methods in the theorems for higher genera to make our example. To consider additional removable singularities, refer to Theorem \ref{thm.zero}.

	\begin{figure}[t]
	\begin{center}
        \begin{tabular}{c}
		Genus 1 \\
		\tbl {|c |c| c|}{
			\hline
			Signature & Type & Self-inverse \\
			\hline
			$(0)$ 	& hyperelliptic	& $(2,1)$\\
			\hline
			$(0,0)$	& hyperelliptic	& $(3,2,1)$\\
			\hline }
        \end{tabular}
	\end{center}

	\begin{center}
        \begin{tabular}{c}
		Genus 2 \\
		\tbl {|c |c| c|}{
			\hline
			Signature & Type & Self-inverse \\
			\hline	
			$(2)$ 	& hyperelliptic	& $(4,3,2,1)$\\
			\hline
			$(1,1)$ & hyperelliptic	& $(5,4,3,2,1)$\\
			\hline }
        \end{tabular}
	\end{center}

	\begin{center}
        \begin{tabular}{c}
		Genus 3 \\
		\tbl{|c |c| c|}{
			\hline
			Signature & Type & Self-inverse \\
			\hline
			$(4)$ & hyperelliptic	& $(6,5,4,3,2,1)$\\
			\hline
			$(4)$ & odd	& $(6,3,2,5,4,1)$\\
			\hline
			$(3,1)$ & -	& $(7,4,3,2,6,5,1)$\\
			\hline
			$(1,3)$ & -	& $(7,3,2,6,5,4,1)$\\
			\hline
			$(2,2)$ & hyperelliptic	& $(7,6,5,4,3,2,1)$\\
			\hline
			$(2,2)$ & odd	& $(7,3,2,4,6,5,1)$\\
			\hline
			$(1,1,2)$ & -	& $(8,3,2,4,7,6,5,1)$\\
			\hline
			$(2,1,1)$ & -	& $(8,4,3,2,5,7,6,1)$\\
			\hline
			$(1,1,1,1)$ & -	& $(9,4,3,2,5,8,7,6,1)$\\
			\hline }
        \end{tabular}
	\end{center}
	\caption{Self-Inverse representatives for genus at most $3$.}\label{fig.genus_three_below}
	\end{figure}

\subsection{Self-Inverses for \texorpdfstring{$g\geq 4$}{g>=4}}\label{sec.ggeq4}

	\begin{thm}\label{thm.odds} Let Rauzy Class $\RClass$ have signature $(\ell_1,\dots,\ell_m)$ such that $\ell_i$ is odd for
		some $i$. Then there exists $\pi\in\RClass$ such that $\pi=\inv{\pi}$.
	\end{thm}
	
	\begin{proof}
        We shall give an explicit construction of such a $\pi$. As there are singularities of odd degree, we only need to verify that our constructed permutation has the appropriate signature. We do this considering two cases: $\ell_1$ is even or $\ell_1$ is odd.
		
        First assume that $\ell_1$ is odd. Then we can rearrange our $\ell_i$'s such that $\ell_i$ is odd for $1\leq i\leq k$ and $\ell_i$ is even for $k<i\leq m$. Notice that $k$ must be even as the sum over all $\ell_i$'s is even. So we define $\pi$ by
            $$ \pi=\stdperm{ \EVEN_{\ell_m} \SPACE \EVEN_{\ell_{m-1}} \SPACE \cdots \SPACE \EVEN_{\ell_{k+1}} \SPACE \PAIR_{\ell_k,\ell_{k-1}} \SPACE \cdots \SPACE \PAIR_{\ell_2,\ell_1}}.$$
        By Lemma \ref{lem.blocks}, this permutation has the appropriate singularities. Since the singularity of degree $\ell_1$ is the one immediately to the left of $\lA$ and $\lB$, it is the marked singularity.
		
        The second case is to assume that $\ell_1$ is even. We then make a division such that $\ell_i$ is even for all $1\leq i\leq k$ and is odd for $k< i \leq m$, noticing that this time $m-k$ must be even. Then our desired $\pi$ is
            $$ \pi=\stdperm{\PAIR_{\ell_m,\ell_{m-1}} \SPACE \PAIR_{\ell_{m-2},\ell_{m-3}} \SPACE \cdots \SPACE \PAIR_{\ell_{k+2},\ell_{k+1}} \SPACE \EVEN_{\ell_k} \SPACE \cdots \SPACE \EVEN_{\ell_1}}.$$
		Just as in the previous case, this has the desired signature.
	\end{proof}
	\begin{thm}\label{thm.even.odd}
        Let Rauzy Class $\RClass$ have signature $(2\ell_1,\dots,2\ell_m)$ and odd spin. Then there exists $\pi\in\RClass$ such that $\pi=\inv{\pi}$.
	\end{thm}
	\begin{proof}
        As opposed to the proof of Theorem \ref{thm.odds}, we must construct a permutation that not only has the appropriate signature $(2\ell_1,\dots,2\ell_m)$ but also satisfies $\Phi(\pi)=1$. Let $\pi$ be defined as
    		$$ \pi = \stdperm{\EVEN_{2\ell_m} \SPACE \EVEN_{2\ell_{m-1}} \SPACE \cdots \SPACE \EVEN_{2\ell_2} \SPACE \EVEN_{2\ell_1}}.$$
		Again, as in the proof of Theorem \ref{thm.odds}, this has the desired signature. By Equation \eqref{eq.spinf},
		$$ \Phi(\pi) = 1 + \sum_{i=1}^{m}\phi(\EVEN_{2\ell_i})=1$$
		as $\phi(\EVEN_{2\ell_i})=0$ by Lemma \ref{lem.blocks}.
	\end{proof}
	\begin{thm}\label{thm.twos.even}
        Let Rauzy Class $\RClass$ have signature $(2,\dots,2)$ and even spin. Then there exists $\pi\in\RClass$ such that $\pi=\inv{\pi}$.
	\end{thm}
	\begin{proof}
        We will construct our desired $\pi$, show that it has the appropriate signature and verify that $\Phi(\pi)=0$. Let $m>1$ be the number singularities of degree $2$. Then we may define $\pi$ as
            $$ \pi = \stdperm{\mB_{m-1} \SPACE \mB_{m-2} \SPACE \cdots \SPACE \mB_1},~\mB_i =\RHScase{\ODD_{2,2},& i=1 \\ \EVEN_2,& \mathrm{otherwise.}}$$
        By Lemma \ref{lem.blocks}, this has the appropriate signature. We also know that $\phi(\EVEN_2)=0$ and $\phi(\ODD_{2,2})=1$. So by \eqref{eq.spinf}
    		$$ \Phi(\pi)=1+\sum_{i=1}^{m-1}\phi(\mB_i)=1+1=0.$$
	\end{proof}
	\begin{thm}\label{thm.even.even}
         Let Rauzy Class $\RClass$ have even spin and signature $(2\ell_1,\dots,2\ell_m)$ such that $\ell_i>1$ for some $i$. Then there exists $\pi\in\RClass$ such that $\pi=\inv{\pi}$.
	\end{thm}
	\begin{proof}
        We must again construct a $\pi$ with signature $(2\ell_1,2\ell_2,\dots,2\ell_m)$ and such that $\Phi(\pi)=0$. Let $j$ be chosen such that $\ell_j>1$. Then we define $\pi$ as
            $$ \pi = \stdperm{\mB_m \SPACE \mB_{m-1} \SPACE \cdots \SPACE \mB_1},~\mB_i = \RHScase{\ODD_{2\ell_i}& \mbox{if }i=j, \\ \EVEN_{2\ell_i}& \mathrm{otherwise.}}$$
        By Lemma \ref{lem.blocks}, this has the appropriate signature. We know that $\phi(\EVEN_{2\ell_i})=0$ and $\phi(\ODD_{2\ell_j})=1$. So by Equation \eqref{eq.spinf},
    		$$ \Phi(\pi)=1+\sum_{i=1}^m\phi(\mB_i)=1+1=0.$$	
	\end{proof}

\subsection{Self-inverses with Removable Singularities}\label{sec.removable}

    Singularities are called removable if they have degree $0$. A suspension of a permutation with removable singularities will move by induction to consider such singularities, but they are not actually zeroes of the corresponding Abelian differential. The results for Corollary \ref{cor.classify} extend to Rauzy Classes with removable singularities. Given a permutation $\pi$ with removable singularities, we first must consider which connected component of which stratum $\RClass(\pi)$ belongs to, called $\CCC$. We then choose a Rauzy Class $\RClass'\subseteq \CCC$ in such that:
    \begin{itemize}
        \item If the genus of $\RClass$ is $1$, then we define $\RClass'$ to be the irreducible permutation on $2$ letters.
        \item If the genus of $\RClass$ is greater than $1$, no singularity of $\RClass'$ is removable.
        \item If the marked singularity of $\RClass$ is not removable and of degree $n$, then the marked singularity of $\RClass'$ is of degree $n$ as well. If the marked singularity of $\RClass$ is removable, $\RClass'$ has no restriction on which singularity is marked.
    \end{itemize}
    In another context, consider a suspension with differential, $(S,q)$, on any representative $\pi\in\RClass$. We complete the differential $q$ at any removable singularity and call this new differential $q'$ on the same surface. As long as the marked singularity of $\pi$ wasn't removable, then the same singularity is marked by $\pi'$, the permutation resulting from $(S,q')$. If the marked singularity of $\pi$ was removable, then we may choose a new marked singularity of what remains in $(S,q')$ and define $\pi'$ by this choice.
    \begin{defn}
        We call such $\RClass'$ an \term{underlying Rauzy Class}, or underlying class, of $\RClass$.
    \end{defn}
    \begin{exam}
        For $\pi=(7,4,5,2,6,3,1)$, $\RClass=\RClass(\pi)$ has signature $(4,0)$ and odd spin (non hyperelliptic). So we need to find $\RClass'\subseteq\Hodd(4)$ with no removable singularities. Our only choice is, by theorem \ref{thm.even.odd}, $\RClass'=\RClass(\pi')$ for $\pi'=(6,3,2,5,4,1)$.
    \end{exam}
    \begin{exam}
        For $\pi=(7,6,1,4,3,2,8,5)$, $\RClass=\RClass(\pi)$ has signature $(0,3,1)$ so it belongs to the connected stratum $\HHH(3,1)$. There are two choices of underlying Rauzy classes, $\RClass'=\RClass(7,4,3,2,6,5,1)$ and $\RClass''=\RClass(7,3,2,6,5,4,1)$, based on the choice of the new marked singularity.
    \end{exam}
    \begin{rem}
        If $\RClass$ has a marked singularity that is not removable, the choice of $\RClass'$ is unique. One can also check that if $\RClass$ is a Rauzy class on $d$ letters with $k$ removable singularities, an underlying class $\RClass'$ is a Rauzy class on $d-k$ letters.
    \end{rem}
    \begin{thm}\label{thm.zero}
        Given a Rauzy Class $\RClass$ with at least one removable singularity, there exists $\pi\in\RClass$ such that $\pi=\inv{\pi}$.
    \end{thm}
    \begin{proof}
        By the above discussion and Theorems \ref{thm.g3}-\ref{thm.even.even}, we have a $\tilde{\pi}=\inv{\tilde{\pi}}$ that belongs to the underlying class $\RClass'$ on $d-k$ letters. Denote this by
            $$ \tilde{\pi}=\stdperm{\BLAH}$$
        We now only need to confirm two cases, either $\ell_1=0$ or $\ell_1\neq 0$ for $\sig(\RClass)=(\ell_1,\dots,\ell_m)$. In the first case, we have the permutation
            $$ \pi=\stdperm{ \BLAH \underbrace{\SPACE\cdots\SPACE}_k}.$$
        In the second case, the permutation is of the form
            $$ \pi=\stdperm{\underbrace{\SPACE\cdots\SPACE}_k \BLAH}.$$
    \end{proof}
    \begin{rem}\label{rem.reduce_removable}
        Consider any permutation $\tilde \pi$. Move to standard $\pi\in\RClass(\tilde \pi)$ (see Claim \ref{cor.std_in_RC}). If $\pi$ has removable singularities then one of the following must be satisfied:
        \begin{enumerate}
            \item There exists $\gamma,\delta\in\AAA$ such that $\pi_\eps(\gamma)+1=\pi_\eps(\delta)$ for each $\eps\in\{0,1\}$.
                Consider the map
                $$ \pi = \cmtrx{\LL{\alpha}{\beta}~\LL{\dots\gamma\delta}{\dots}~\LL{\dots}{\gamma\delta\dots}~\LL{\beta}{\alpha}}
                    \to
                \cmtrx{\LL{\alpha}{\beta}~\LL{\dots\gamma}{\dots}~\LL{\dots}{\gamma\dots}~\LL{\beta}{\alpha}} = \pi'$$
                that ``forgets" $\delta$. Figure \ref{fig.remove_ex} shows how this map eliminates the removable singularity between $\gamma$ and $\delta$.
            \item There exists $\gamma\in\AAA$ such that $\pi_0(\gamma)=\pi_1(\gamma)=2$. The map
                $$ \pi = \cmtrx{\LL{\alpha}{\beta}~\LL{\gamma}{\gamma}~\LL{\delta\dots}{\eta\dots}~\LL{\beta}{\alpha}}
                    \to
                \cmtrx{\LL{\alpha}{\beta}~\LL{\delta\dots}{\eta\dots}~\LL{\beta}{\alpha}}=\pi'$$
                ``forgets" $\gamma$. This map eliminates the removable singularity to the left of the segments labeled $\gamma$.
            \item There exists $\gamma\in\AAA$ such that $\pi_0(\gamma)=\pi_1(\gamma) = d-1$. Consider the map
                $$ \pi = \cmtrx{\LL{\alpha}{\beta}~\LL{\dots\delta}{\dots\eta}~\LL{\gamma}{\gamma}~\LL{\beta}{\alpha}}
                    \to
                \cmtrx{\LL{\alpha}{\beta}~\LL{\dots\delta}{\dots\eta}~\LL{\beta}{\alpha}}=\pi'$$
                that ``forgets" $\gamma$. This new permutation no longer has a removable marked singularity. The new marked singularity was originally on the left of the segments labeled $\gamma$.
        \end{enumerate}
        By performing these maps, we may explicitly derive a standard representative of $\RClass'$.
    \end{rem}
    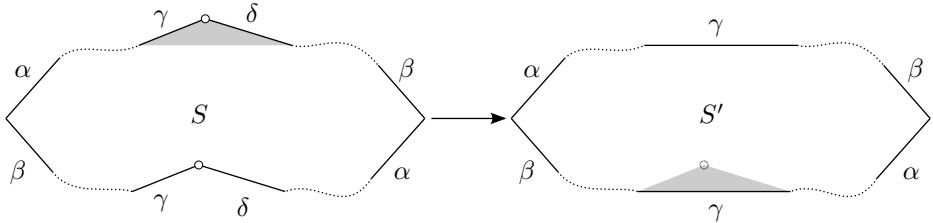
\begin{figure}[h]
        \begin{center}
           \setlength{\unitlength}{350pt}
            \begin{picture}(1,.23)
                \put(0,0){\includegraphics[width=\unitlength]{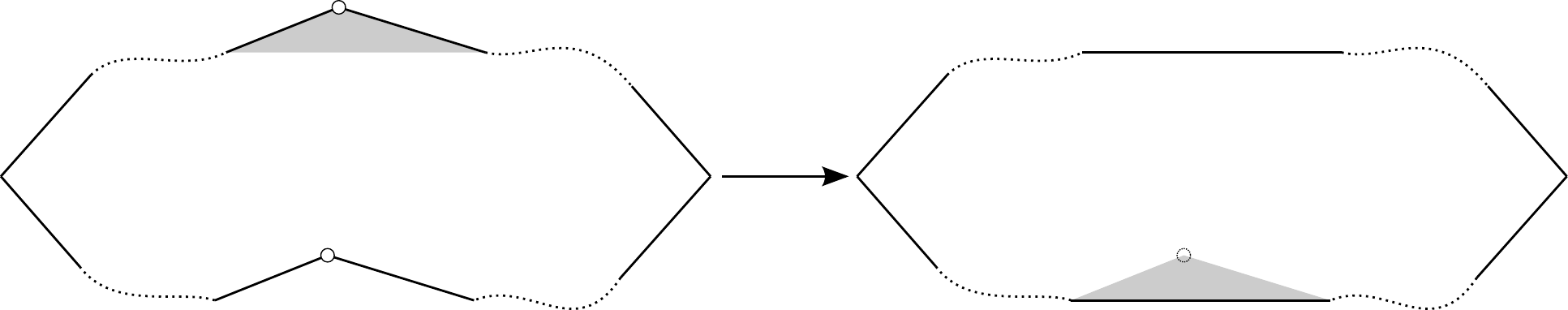}}

                \put(.01,.13){$\alpha$} \put(.16,.19){$\gamma$} \put(.26,.19){$\delta$} \put(.425,.13){$\beta$} \put(.005,.02){$\beta$} \put(.16,-.01){$\gamma$} \put(.25,-.02){$\delta$} \put(.42,.02){$\alpha$}

                \put(.20,.08){$S$} \put(.75,.08){$S'$}

                \put(.56,.13){$\alpha$} \put(.76,.18){$\gamma$} \put(.555,.02){$\beta$} \put(.975,.13){$\beta$} \put(.76,-.02){$\gamma$} \put(.97,.02){$\alpha$}
           \end{picture}
       \end{center}
       \caption{The differential for $S'$ completes the removable singularity in the differential for $S$. All other singularities remain unchanged.}\label{fig.remove_ex}
    \end{figure}
    \begin{exam}
        Begin with $\pi = (7,4,5,1,6,2,3)$, with signature $\sig(\pi)=(2,0,0,0)$. We then consider standard $\pi'=\rb^3\pi = (7,6,2,3,4,5,1)$. By performing the first reduction listed in Remark \ref{rem.reduce_removable}, we get
        $$ \cmtrx{\LL{1}{7}\LL{2}{6}\LL{3}{2}\LL{4}{3}\LL{5}{4}\LL{6}{5}\LL{7}{1}}
            \to \cmtrx{\LL{1}{7}\LL{2}{6}\LL{4}{2}\LL{5}{4}\LL{6}{5}\LL{7}{1}}
            \to \cmtrx{\LL{1}{7}\LL{2}{6}\LL{5}{2}\LL{6}{5}\LL{7}{1}}
            \to \cmtrx{\LL{1}{7}\LL{2}{6}\LL{6}{2}\LL{7}{1}} = (4,3,2,1) = \tilde\pi.$$
        In this case, $\tilde\pi$ is in the hyperelliptic class and is self-inverse with one singularity of degree $2$. So by the proof of Theorem \ref{thm.zero}, we derive,
        $$ \pi'' = \cmtrx{\LL{1}{4}~\LL{a}{a}~\LL{b}{b}~\LL{c}{c}~\LL{2}{3}~\LL{3}{2}~\LL{4}{1}}=(7,2,3,4,6,5,1).$$
        One may verify that $\pi'' = \rt^3\rb^2\rt^2\rb\pi$. Therefore $\pi''\in\RClass(\pi)$ and is self-inverse.
    \end{exam}
\section{Explicit Lagrangian Subspaces in Rauzy Classes}\label{chapLag}

    Let us now consider $\pi\in\irr_d$ and unit suspension $S_1=S(\pi,\mone,\tau)$, where $\mone=(1,\dots,1)$, of genus $g$. We will consider a natural question, when do the closed vertical loops in $S_1$ span a $g$-dimensional subspace in homology $H=H_1(S_1)$? Consider the symplectic space $(H,\omega)$ where $\omega$ is the (algebraic) intersection number. In Theorem \ref{t.Isotropic}, we verify algebraically that the vertical loops do not intersect. We then prove in Theorem \ref{t.Lag.Perm} that if $\pi=\inv{\pi}$, then there are $g$ vertical loops independent in homology. Self-inverses have transpositions, i.e. letters that are interchanged, and fixed letters. Theorem \ref{t.SLag.Perm} shows that the transposition pairs in block-constructed $\pi$ from Section \ref{chapTrue} form the basis of this $g$-dimensional space.

\subsection{Symplectic Space}
    In this section, we give the definition of a symplectic space and list some basic properties of such spaces. Using the well known result of Proposition \ref{prop.iso_dim}, we prove Lemma \ref{l.splitting} which we will use to prove Theorem \ref{t.Lag.Perm}.

    \begin{defn}
        A vector space and bilinear form $(H,\omega)$ is \term{symplectic} if for all $v\in H$
        \begin{itemize}
            \item $\omega(v,v)=0$, or $(H,\omega)$ is \term{isotropic}
            \item$\omega(u,v)=0$ for all $u\in H$ implies $v=0$, or $(H,\omega)$ is \term{non-degenerate}
        \end{itemize}
    \end{defn}

    \begin{defn}
        Given a symplectic space $(H,\omega)$ and subspace $V$, define $V^\omega$ as
        $$ V^\omega := \{u\in H: \omega(u,v)=0\mbox{ for every }v\in V\}.$$
    \end{defn}

    \begin{prop}\label{prop.iso_dim}
        Let $(H,\omega)$ be a symplectic space with subspace $V$. Then
        $$\dim V + \dim V^\omega = \dim H.$$
    \end{prop}

    \begin{defn}
        Given symplectic space $(H,\omega)$, $V\subset H$ is \term{isotropic} if $V\subset V^\omega$. $V$ is \term{Lagrangian} if $V=V^\omega$.
    \end{defn}

    \begin{lem} \label{l.splitting}
        Let $W,V$ be subspaces of symplectic space $(H,\omega)$ such that:
        \begin{itemize}
            \item $V$ is isotropic,
	    \item $W$ is isotropic, and
            \item $H = V + W$.
        \end{itemize}
        Then $V$ is Lagrangian.
    \end{lem}

    \begin{proof}
	The first two conditions and Proposition \ref{prop.iso_dim} imply that 
	    $$\dim V\leq g,\dim W \leq g$$
	and the third imples 
	    $$\dim V + \dim W \geq 2g = \dim H.$$
	It follows that $\dim V = \dim W = g$ and therefore $V$ (and $W$) is a Largangian subspace.
    \end{proof}

\subsection{Lagrangian Permutations}

    Let $\pi=(\pi_0,\pi_1)\in\irr_d$ over alphabet $\AAA$, recall the definition of closed loops $\gamma_\alpha$ and their corresponding cycles $c_\alpha = [\gamma_\alpha]$ in homology (see Section \ref{sec.spin}). We recall that $(H^\pi,\omega)$, where $H^\pi = \Omega_\pi\RR^\AAA$ and $\omega(\Omega_\pi u, \Omega_\pi v)= u^t\Omega_\pi v$, form a symplectic space with dimension $2g(\pi)$. Let $\Omega=\Omega_\pi$, and recall that
    \begin{equation}\label{eq.omega_equation}
        \Omega_{\alpha,\beta} = \chi_{\pi_1(\alpha)\leq \pi_1(\beta)}-\chi_{\pi_0(\alpha)\leq \pi_0(\beta)}.
    \end{equation}
    \begin{defn}
        For alphabet $\AAA$ and $\pi=(\pi_0,\pi_1)\in\irr_d$, consider the natural action of $\pi$ on $\AAA$, $\pi_\AAA$, by
        $$ \pi_\AAA:= \inv{\pi_0}\circ \pi \circ \pi_0 = \inv{\pi_0}\circ (\pi_1\circ\inv{\pi_0}) \circ \pi_0 = \inv{\pi_0}\circ \pi_1.$$
        Denote the set of orbits of $\AAA$ of $\pi_\AAA$ by
        $$ \AA := \{\BBB\subseteq \AAA: \BBB= \OOO_{\pi_\AAA}(\alpha)\mbox{ for }\alpha\in\AAA\}.$$
        For each $k\in\dset{d}$, let
        $$ \AA_k:= \{\BBB\in\AA: \#\BBB=k\}\subseteq \AA.$$
    \end{defn}
    \begin{exam}\label{ex.perm_AAA}
        Consider
        $$ \pi=\cmtrx{\LL{a}{f}~\LL{b}{e}~\LL{c}{b}~\LL{d}{d}~\LL{e}{c}~\LL{f}{a}}\in\irr_6.$$
        In this case $\pi_\AAA(a) = f$, $\pi_\AAA(b) = c$, $\pi_\AAA(c)=e$, $\pi_\AAA(d) = d$, $\pi_\AAA(e) = b$ and $\pi_\AAA(f) = a$. Also, $\AA = \{\{a,f\},\{b,c,e\},\{d\}\}$, $\AA_1 = \{\{d\}\}$, $\AA_2 = \{\{a,f\}\}$ and $\AA_3 = \{\{b,c,e\}\}$.
    \end{exam}
    \begin{defn}
        Let $\{\ee_\alpha\}_{\alpha\in\AAA}$ form the standard orthonormal basis of $\RR^\AAA$. For $\BBB\subseteq\AAA$, let
        $$ \ee_\BBB := \sum_{\alpha\in\BBB}\ee_\alpha.$$
        Let the vector space of vertical cycles under $\pi$ be
        $$ V^\pi := \spaN\{\ee_\BBB: \BBB\in\AA\}\subseteq \RR^\AAA$$
        and for $k\in\dset{d}$, let $V^\pi_k := \spaN\{\ee_\BBB: \BBB\in\AA_k\}$.
        Let $W^\pi$ be naturally defined by $\RR^\AAA = V^\pi\oplus W^\pi$.
    \end{defn}
    \begin{defn}
        Let the image of the vertical cycles under homology be $\vv_\BBB:=\Omega\ee_\BBB$, their span be
        $$ H^\pi_V := \Omega V^\pi = \spaN\{\vv_\BBB: \BBB\in\AA\}$$
        and $H^\pi_{V_k} :=\Omega V^\pi_k = \spaN\{\vv_\BBB: \BBB\in\AA_k\}$ for $k\in\dset{d}$. Also, let $H^\pi_W = \Omega W^\pi$.
    \end{defn}
    \begin{exam}
        Consider again $\pi$ from Example \ref{ex.perm_AAA} with
			$$\AA=\{\{a,f\},\{b,c,e\},\{d\}\}.$$
		We have that
			$$ V^\pi = \spaN\{\ee_{a,f}, \ee_{b,c,e}, \ee_{d}\}\mbox{ and } W^\pi = \spaN\{\ee_{a}-\ee_f, \ee_b-\ee_c, \ee_c-\ee_e\}.$$
        Consider the definition of $\Omega_\pi$ from Equation \ref{eq.omega_pi2}. Then
        $$ \begin{array}{l}
                H_V^\pi = \spaN\{\underbrace{(1,0,0,0,0,-1)^t}_{\vv_{a,f}}, \underbrace{(3,1,1,0,-2,-3)^t}_{\vv_{b,c,e}}, \underbrace{(1,0,1,0,-1,-1)^t}_{\vv_d}\}\mbox{ and }\\
                H_W^\pi = \spaN\{\underbrace{(-1,-2,-2,-2,-2,-1)^t}_{\ww_1}, \underbrace{(0,0,0,1,0,0)^t}_{\ww_2}, \underbrace{(0,-1,-1,-2,-1,0)^t}_{\ww_3}\}.
            \end{array}$$
        One may then calculate that $\omega(\vv_\BBB,\vv_\CCC)=\omega(\ww_i,\ww_j)=0$ for each $\BBB,\CCC\in\AA$ and $i,j\in\{1,2,3\}$. The other values are given by the following table
        $$ \omega(\vv_\BBB,\ww_i)=\begin{array}{|r|ccc|}
            \hline
                \BBB ~ \setminus ~ i & 1 & 2 & 3\\
            \hline
                \{a,f\} & -2&0&0\\
                \{b,c,e\}& -6&0&-3\\
                \{d\}& -2&1&-2\\
            \hline
            \end{array}$$
        and that $\omega$ is antisymmetric.
    \end{exam}
    \begin{rem}
        For any $\BBB,\CCC\subseteq\AAA$, we note the following formula
        \begin{equation}\label{eq.cycle_form}
            \omega(\vv_\BBB,\vv_\CCC) = \sum_{\alpha\in\BBB, \beta\in\CCC} \Omega_{\alpha,\beta}.
        \end{equation}
    \end{rem}
    \begin{thm}\label{t.Isotropic}
        For any $\pi\in\irr_d$, $H_V^\pi$ is isotropic.
    \end{thm}
    \begin{proof}
        It suffices to show that for every $\BBB,\CCC\in\AA$, $\omega(\vv_\BBB,\vv_\CCC)=0$. Because each $\BBB\in\AA$ is an orbit of $\pi_\AAA$, for $\eps\in\{0,1\}$,
        \begin{equation}\label{eq.orbit_good}
            \forall \alpha\in\BBB, ~\exists \beta\in\BBB \mbox{ s.t. }\pi_\eps(\alpha) = \pi_{1-\eps}(\beta).
        \end{equation}
        So for each $k\in\dset{d}$,
        \begin{equation}\label{eq.orbit_result}
            \#\{\alpha\in\BBB:\pi_0(\alpha)\leq k\} = \#\{\alpha\in\BBB:\pi_1(\alpha)\leq k\}.
        \end{equation}
        The calculation follows:
        $$ \begin{array}{rcll}
                \omega(\vv_\BBB,\vv_\CCC) & = & \underset{\alpha\in\BBB, \beta\in\CCC}\sum \Omega_{\alpha,\beta} & \mbox{by \eqref{eq.cycle_form}}\\
                    & = & \underset{\alpha\in\BBB}\sum\underset{\beta\in\CCC}\sum \chi_{\pi_1(\alpha) \leq \pi_1(\beta)} - \chi_{\pi_0(\alpha) \leq \pi_0(\beta)} & \mbox{by \eqref{eq.omega_equation}}\\
                    & = & \underset{\alpha\in\BBB}\sum \#\{\beta\in\CCC:\pi_1(\alpha) \leq \pi_1(\beta)\}\\
					& & -\underset{\alpha\in\BBB}\sum \#\{\beta\in\CCC:\pi_0(\alpha) \leq \pi_0(\beta)\} & \\
                    & = & \underset{\alpha\in\BBB}\sum \#\{\beta\in\CCC:\pi_1(\alpha) \leq \pi_1(\beta)\}\\
					& & -\underset{\alpha\in\BBB}\sum \#\{\beta\in\CCC:\pi_0(\alpha) \leq \pi_1(\beta)\} & \mbox{by \eqref{eq.orbit_good}}\\
                    & = &  0 &\mbox{by \eqref{eq.orbit_result}}
            \end{array}$$
        Therefore $H^\pi_V$ is isotropic.
    \end{proof}
    \begin{defn}
        $\pi\in\irr_d$ is \term{Lagrangian} if $H_V^\pi$ is Lagrangian.
    \end{defn}
    \begin{exam}
        Let $\pi = (4,1,3,2)$. In this case $H^\pi_V$ is spanned by two vectors, $(1,1,0,-2)^t$ and $(0,1,0,-1)^t$. So $\pi$ is Lagrangian. On the other hand, if $\pi' = (3,1,4,2)$ then $H_V^{\pi'}$ is spanned by only the vector $(1,2,-2,-1)^t$. Therefore $\pi'$ is not Lagrangian.
    \end{exam}
    \begin{rem}
        Naturally
        $$V^\pi = \bigoplus_{k=1}^d V^\pi_k.$$
        When $\pi$ is self-inverse, $V^\pi_k=\{0\}$ for all $k>2$. So
        $$ V^\pi = V^\pi_2 \oplus V^\pi_1,$$
        where $V^\pi_1$ corresponds to fixed $\alpha$ under $\pi_\AAA$, and $V^\pi_2$ corresponds to \term{transpositions}, pairs of letters $\{\alpha,\beta\}$ that are switched under $\pi_\AAA$.
        In this case, let
        $$W^\pi= \spaN\{\ee_{\alpha,-\beta}: \{\alpha,\beta\}\in\AA_2\}$$
        where $\ee_{\alpha,-\beta}=\ee_\alpha-\ee_\beta$ and $\vv_{\alpha,-\beta} = \Omega\ee_{\alpha,-\beta}$. It follows that $\RR^\AAA=V^\pi\oplus W^\pi$ and for
	$$ H_W^\pi = \Omega W^\pi = \spaN\{\vv_{\alpha,-\beta}: \{\alpha,\beta\}\in\AA_2\},$$
	$H^\pi = H^\pi_V+H^\pi_W$.
    \end{rem}
    \begin{thm}\label{t.Lag.Perm}
        Suppose $\pi\in\irr_d$ is self-inverse. Then $\pi$ is Lagrangian.
    \end{thm}
    \begin{proof}
        From the previous remark, $H^\pi = H^\pi_V+H^\pi_W$. Also by Theorem \ref{t.Isotropic}, $H^\pi_V$ is isotropic. It suffices to show that $H^\pi_W$ is also isotropic. Consider two vectors $\vv_{\alpha,-\beta},\vv_{\zeta,-\eta}$ such that $\{\alpha,\beta\},\{\zeta,\eta\}\in\AA_2$. Because $\{\alpha,\beta\}$ and $\{\zeta,\eta\}$ belong to $\AA_2$,
	\begin{equation}\label{eq.tLagParmA}
	  \begin{array}{rl}
	      \pi_0(\alpha) = \pi_1(\beta), & \pi_1(\alpha) = \pi_0(\beta), \mbox{ and} \\
	      \pi_0(\zeta) = \pi_1(\eta), & \pi_1(\zeta) = \pi_0(\eta). 	   
	  \end{array}
	\end{equation}
	Then
	$$ \begin{array}{rcll}
	      \omega(\vv_{\alpha,-\beta},\vv_{\zeta,-\eta}) & =&  \Omega_{\alpha,\zeta}+\Omega_{\beta,\eta}-\Omega_{\alpha,\eta} - \Omega_{\beta,\zeta}& \mbox{by }\eqref{eq.cycle_form}\\
	      & = & \chi_{\pi_1(\alpha)\leq\pi_1(\zeta)}- \chi_{\pi_0(\alpha)\leq\pi_0(\zeta)}& \\
	      &  & +\chi_{\pi_1(\beta)\leq\pi_1(\eta)}- \chi_{\pi_0(\beta)\leq\pi_0(\eta)}& \\
	      &  & +\chi_{\pi_1(\alpha)\leq\pi_1(\zeta)}- \chi_{\pi_0(\alpha)\leq\pi_0(\zeta)}\\
	      &  & +\chi_{\pi_1(\beta)\leq\pi_1(\eta)}- \chi_{\pi_0(\beta)\leq\pi_0(\eta)}& \mbox{by } \eqref{eq.omega_equation}\\
	      & = & 0. & \mbox{by \eqref{eq.tLagParmA}}
	   \end{array}$$
	Therefore as $H_W^\pi$ is generated by $\vv_{\alpha,-\beta}$, $\{\alpha,\beta\}\in\AA_2$, $H_W^\pi$ is isotropic. We conclude that $H_V^\pi$ is Lagrangian by Lemma \ref{l.splitting}.
    \end{proof}
    \begin{corollary}
        If $\pi$ is self-inverse, then $\pi$ has at least $g(\pi)$ transpositions.
    \end{corollary}
    \begin{proof}
        By construction $\#\AA_2 = \dim V^\pi_2 = \dim W^\pi \geq \dim H_W^\pi = g(\pi)$.
    \end{proof}
    The following provides an alternative proof of Lemma 4.4 in \cite{c.For02}.
    \begin{corollary}\label{cor.Lag.For}
        In every connected component $\CCC$ of every stratum of Abelian differentials, let $\LLL$ be the set of $q\in\CCC$ such that:
        \begin{itemize}
            \item the vertical trajectories defined by $q$ that avoid singularities are periodic,
            \item the span of these vertical trajectories span a Lagrangian subspace in homology.
        \end{itemize}
        Then the set $\LLL$ is dense in $\CCC$.
    \end{corollary}
    \begin{proof}
        By Theorem \ref{thm.main}, every Rauzy Class $\RClass$ in $\CCC$ contains a self-inverse permutation $\pi$. Theorem \ref{t.Lag.Perm} shows that any unit suspension $S_1=S(\pi,\mone,\tau)$ satisfies the conditions of the claim. It is known that the Teichm\"uller geodesic flow (see Section \ref{sec.surface}) is ergodic, and therefore, we may choose $\tau$ such that the inverse flow is dense. It follows from an argument similar to Proposition 2.11 in \cite{c.Ve84_II}, for example, that every differential in the inverse flow also satisfies the conditions of the claim.
    \end{proof}

\subsection{Transposition Lagrangian Permutations}

    Theorem \ref{t.Lag.Perm} shows that the vertical cycles of any self-inverse permutation span a Lagrangian subspace in homology. In general, choosing a basis from these cycles still requires calculation. However the block-constructed permutations in Definition \ref{def.blockConstructed} enjoy an additional property: the transpositions cycles form a basis for the Lagrangian subspace. We make this definition explicit, and then prove this result in Theorem \ref{t.SLag.Perm}.

    \begin{defn}
        A self-inverse permutation $\pi$ is \term{transposition Lagrangian} if $\dim H_{V_2}^\pi = \dim V^\pi_2 = g(\pi)$.
    \end{defn}
    \begin{exam}
        The permutation $\pi = (7,5,3,6,2,4,1)$ is self-inverse with $g(\pi)=3$. There are $3$ transposition pairs, $\{1,7\}$,$\{2,5\}$ and $\{4,6\}$, and one fixed letter $\{3\}$. However, we see that $\vv_{2,5}=\vv_{4,6} = (2,1,0,1,-1,-1,-2)^t$, $\vv_{1,7}=(1,0,0,0,0,0,-1)^t$ and $\vv_3=(1,1,0,0,-1,0,-1)^t$. So $\dim H_{V_2}^\pi = 2 < 3 = \dim H_V^\pi$, and therefore $\pi$ is Lagrangian, but not transposition Lagrangian.
    \end{exam}
    \begin{thm}
        If $\pi=(d,d-1,\dots,2,1)\in\irr_d$ then $\pi$ is transposition Lagrangian.
    \end{thm}
    \begin{proof}
        Suppose
        $$ \pi = \cmtrx{a_1 & a_2 & \dots & a_{d-1} & a_d \\ a_d & a_{d-1} & \dots & a_2 & a_1}.$$
        We recall that
        $$\Omega_{a_i,a_j} = \RHScase{1, & \mbox{if }i<j, \\ 0, & \mbox{if }i=j, \\ -1, & \mbox{if }i>j.}$$
        There are exactly $g=g(\pi)$ transpositions, $\{a_i,a_{d+1-i}\}\in\AA_2$ for $i\in\dset{g}$. Because $\#\AA_2 = \dim V^\pi_2 = g$, we must now show that the vectors $\vv_{a_i,a_{d+1-i}}$ are linearly independent. We see that for $j,k\in\dset{g}$,
        $$ \omega(\vv_{a_k}, \vv_{a_j,a_{d+1-j}}) = \Omega_{a_k,a_j} + \Omega_{a_k,a_{d+1-j}}= \RHScase{0, & \mbox{if }k>j,\\ 1, & \mbox{if }k=j, \\ 2, & \mbox{if }k<j.}$$
        So consider any $c_1,\dots,c_g\in\RR$ such that
        $$ \ww = c_1\vv_{a_1,a_{d}}+ \dots + c_g\vv_{a_g,a_{d+1-g}}= 0.$$
        It is clear now that $\omega(\vv_{a_g},\ww) = c_g = 0$. Inductively, if $c_g=\dots = c_k=0$, then $\omega(\vv_{a_{k-1}},\ww) = 2c_g+\dots+2c_k+c_{k-1}=c_{k-1}=0$. So $c_1=\dots=c_g=0$, implying the vectors $\vv_{a_i,a_{d+1-i}}$ are linearly independent. So $\dim H^\pi_{V_2} = \dim V^\pi_2=g$.
    \end{proof}
    \begin{thm} \label{t.SLag.Perm}
        Let $\pi$ be a block-constructed permutation. Then $\pi$ is transposition Lagrangian.
    \end{thm}
    \begin{proof}
        We begin by showing that $\#\AA_2 = \dim V^\pi_2 = g(\pi)$. Recall the formula
        $$ g=g(\pi) = \frac{1}{2}\sum_i \ell_i + 1$$
        where the $\ell_i$'s are the degrees of the singularities of $\pi$. Every block constructed self-inverse is of the form
        $$ \pi = \stdperm{\mB_1\SPACE\cdots\SPACE\mB_k}.$$
        So $\{\lA,\lB\}\in\AA_2$. For each $i\in\dset{k}$, $$\mB_i\in\{\EVEN_{2m}=\EVEN_2^m,\ODD_{2,2},\ODD_{2n}=\ODD_2^n,\PAIR_{2m+1,2n+1}=\EVEN_2^m\PAIR_{1,1}\EVEN_2^n\}.$$
        The desired result then follows as $\EVEN_2$ and $\PAIR_{1,1}$ have exactly one transposition pair, and $\ODD_4$ and $\ODD_{2,2}$ have exactly two transposition pairs respectively.

        So there are $g$ pairs $\{\alpha_1,\beta_1\}\dots\{\alpha_g,\beta_g\}\in\AA_2$. Now we show that the set of $\vv_{\alpha_i,\beta_i}$'s are linearly independent. Suppose that there exists $c_1,\dots,c_g\in\RR$ such that
        $$ \ww = c_1\vv_{\alpha_1,\beta_1} + \dots + c_g\vv_{\alpha_g,\beta_g} = 0.$$
        Suppose $i$ is such that $\{\alpha_i,\beta_i\}=\{a,b\}\subseteq\EVEN_2$ or $\PAIR_{1,1}$ (as in the notation mentioned in \ref{sec.stdspin}), where
			$$\EVEN_2=\cmtrx{\LL{a}{b}~\LL{b}{a}} \mbox{ or } \PAIR_{1,1}=\cmtrx{\LL{a}{b}~\LL{c}{c}~\LL{b}{a}}.$$
        We see that
			$$ \omega(\vv_{a},\ww) = c_i\Omega_{a,b} = \pm c_i = 0,$$
        implying that $c_i=0$. Suppose $i$ is such that $\{\alpha_i,\beta_i\}=\{a,b\}\subseteq\ODD_{2,2}$ or $\ODD_4$ for
			$$\ODD_{2,2}=\cmtrx{\LL{a}{b} ~\LL{c}{d} ~\LL{e}{e} ~\LL{d}{c} ~\LL{b}{a}}\mbox{ or }\ODD_4=\cmtrx{\LL{a}{b} ~\LL{c}{d} ~\LL{d}{c} ~\LL{b}{a}}$$
        with $j$ such that pair $\{\alpha_j,\beta_j\}=\{c,d\}\in\AA_2$. Then
        $$ \begin{array}{rcl}
            \omega(\vv_{c},\ww) & = & c_j\Omega_{c,d}\\
                & = & \pm c_j=0 \Rightarrow c_j=0\\
            \omega(\vv_{a},\ww) & = & c_j(\Omega_{a,c}+\Omega_{a,d})+ c_i\Omega_{a,b}\\
                & = & \pm c_i=0 \Rightarrow c_i=0.
        \end{array}$$
        We now see that for $i$ such that $\{\alpha_i,\beta_i\}=\{\lA,\lB\}\in\AA_2$ (the outside letters of the permutation),
        $$ \ww = c_i\vv_{\lA,\lB} = 0 \Rightarrow c_i = 0.$$
        So $c_1=\dots=c_g=0$, implying that the $\vv_{\alpha_i,\beta_i}$'s are linearly independent.
    \end{proof}
    
\section{Invariant Measures}\label{chap.measures}

    In this section, we prove the bound in Theorem \ref{thm.Bound_on_Measures} is sharp. Influenced by Keane's exmaple (\cite{c.Kea77}) and the examples given by Yoccoz (\cite{c.Yoc07}) for hyperelliptic classes, we will construct complete Rauzy Paths $\gamma$ such that their associated simplex of unit length vectors, $\Del(\gamma)$ has $g(\pi)$ vertices. This construction will be suitable for every non-hyperelliptic Rauzy Class, as examples abound for the hyperelliptic classes.

\subsection{A First Example}\label{sec.meas.exam}

	Let $\pi$ be the permutation
		$$ \pi = \cmtrx{1&2&3&4&5&6 \\ 6&3&2&5&4&1}.$$
	This is in the non-hyperelliptic Rauzy Class with one singularity of degree $4$, with genus $g(\pi)=3$. Given two integers $a,c>0$, we define four Rauzy Paths
		\begin{equation}
			\begin{array}{rcl}
				\gamma_0 &=& 0, \\
				\gamma_{1,a} &=& 10^a10^2,\\
				\gamma_{2,a} &=& 1^30^a10^2,\\
				\gamma_{3,c} &=& 1^{5c}, \mbox{ and}\\
				\gamma^{a,c} &=& \gamma_0\gamma_{1,a}\gamma_{2,a}\gamma_{3,c}.
			\end{array}
		\end{equation}
	\begin{figure}[t]
		\begin{center}
			$$ \xymatrix{ {\cmtrx{1~2~3~4~5~6 \\ 6~3~2~5~4~1}} \ar[r]^{0} \ar@(dr,dl)^{1^{5c}}& {\cmtrx{1~2~3~4~5~6 \\ 6~1~3~2~5~4}}\ar[r]^{0^2} \ar@<1ex>[d]^{1}& {\cmtrx{1~2~3~4~5~6 \\ 6~5~4~1~3~2}} \ar@(u,u)[ll]_{0^2} \ar@<1ex>[d]^{1^3}\\
					& {\cmtrx{1~2~3~4~6~5 \\ 6~1~3~2~5~4}} \ar@<1ex>[u]^{1} \ar@(dr,dl)^{0^{a}}& {\cmtrx{1~2~4~5~6~3 \\ 6~5~4~1~3~2}}\ar@<1ex>[u]^{1}\ar@(dr,dl)^{0^{a}} }$$
		\end{center}
		\caption{The path $\gamma^{a,c}$ acting on $\pi$.}\label{fig.gamma_a_b_c}
	\end{figure}
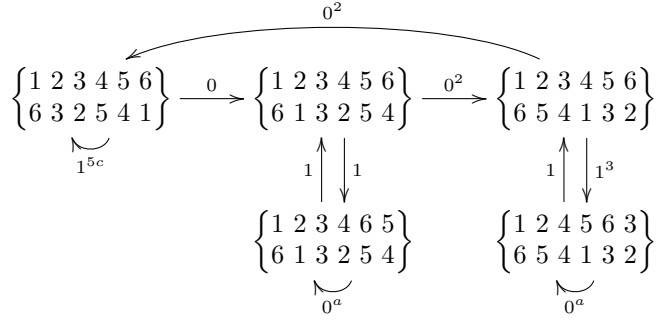
	Every letter in $\AAA=\{1,2,3,4,5,6\}$ wins in the path $\gamma^{a,c}$: the symbol `$1$' wins in $\gamma_{0}$, `$2$' and `$3$' win in path $\gamma_{1,a}$, `$4$' and `$5$' win in path $\gamma_{2,a}$ and `$6$' wins in path $\gamma_{3,c}$. We have the following matrices
	$$ \Theta_{\gamma_0} = \pmtrx{1&0&0&0&0&0 \\ 0&1&0&0&0&0 \\ 0&0&1&0&0&0 \\ 0&0&0&1&0&0 \\ 0&0&0&0&1&0 \\ 1&0&0&0&0&1}
			\mbox{, }
		\Theta_{\gamma_{1,a}} = \pmtrx{1&0&0&0&0&0 \\ 0&1&0&0&0&0 \\ 0&0&1&0&0&0 \\ 0&0&0&2&2&1 \\ 0&0&0&a&a+1&0 \\ 0&0&0&1&1&1}
			\mbox{,}$$
	$$ \Theta_{\gamma_{2,a}} = \pmtrx{1&0&0&0&0&0 \\ 0&2&2&1&1&1 \\ 0&a&a+1&0&0&0 \\ 0&0&0&1&0&0 \\ 0&0&0&0&1&0 \\ 0&1&1&0&0&1}
			\mbox{ and }
		\Theta_{\gamma_{3,c}} = \pmtrx{1&c&c&c&c&c \\ 0&1&0&0&0&0 \\ 0&0&1&0&0&0 \\ 0&0&0&1&0&0 \\ 0&0&0&0&1&0 \\ 0&0&0&0&0&1}\mbox{.}$$
	Therefore
	$$ \Theta_{\gamma^{a,c}} = \Theta_{\gamma_0}\Theta_{\gamma_{1,a}}\Theta_{\gamma_{2,a}}\Theta_{\gamma_{3,c}}
		= \pmtrx{1&c&c&c&c&c \\ 0&2&2&1&1&1 \\ 0&a&a+1&0&0&0 \\ 0&1&1&2&2&1 \\ 0&0&0&a&a+1&0 \\ 1&c+1&c+1&c+1&c+1&c+1}.$$
	Now let $\ma=\{a_i\}_{i>0}$ and $\mc=\{c_i\}_{i>0}$ be sequences of positive integers. We then define an infinite Rauzy Path
		$$ \gamma=\gamma^{\ma,\mc} = \gamma^{a_1,c_1}\gamma^{a_2,c_2}\dots$$
	Because every letter wins in each $\gamma^{a_i,c_i}$, $\gamma$ is a complete path (see Definition \ref{def.complete_path}). Therefore $\Lambda(\gamma)$, the set of length vectors $\lambda\in\RR_+^\AAA$ such that $T=(\pi,\lambda)$ has Rauzy path $\gamma$, is non-empty.

	Now suppose that for each $i$, $a_i\geq 3c_i$ and $c_i\geq 3a_{i-1}$ and let $\lambda^{(j)}$, $j\in\AAA$, be defined as
	$$ \lambda^{(j)} = \lim_{i\to\infty}\hat\Theta_{\gamma^{a_1,c_1}}\cdots \hat\Theta_{\gamma^{a_i,c_i}} \me_j$$
	where $\me_1,\dots,\me_6$ are the standard basis vectors of $\RR^\AAA$. It follows that $\lambda^{(1)}=\lambda^{(6)}$, $\lambda^{(2)}=\lambda^{(3)}$ and $\lambda^{(4)}=\lambda^{(5)}$ and these three points are the vertices of simplex $\Del(\gamma)$. These correspond to three discint invariant ergodic probability measures for any $T=(\pi,\lambda)$, $\lambda\in\Lambda(\gamma)$ (see Theorem \ref{thm.LVector_is_Measure}).
	
\subsection{Main Result}

	In this section, we consider a specific type of permutation. These satisfy Equations \eqref{eq.j_and_n} and \eqref{eq.pi_j_n} and exist in every Rauzy Class. For such a $\pi$, we define a closed Rauzy Path on $\pi$ dependent on (up to) three parameters. We then define an infinite complete Ruazy Path that begins at $\pi$ and is defined by three positive integer sequences. We then state a claim about such paths. We shall verify this claim in Section \ref{sec.geq_exp} for a specific case of such sequences.
	
	Consider $2=j_1<\dots j_m<d$ and $n_1,\dots,n_m$ such that
	\begin{equation}\label{eq.j_and_n}
	\begin{array}{l}
		j_{\alpha+1}=j_\alpha+n_\alpha \mbox{ for } 1\leq \alpha < m,\\
		n_\alpha\in\{1,\dots, 5\} \mbox{ for }1\leq \alpha \leq m \mbox{, and}\\
		j_m+n_m=d.
	\end{array}
	\end{equation}
	We add (for convenience of notation) $j_0=1$ and $j_{m+1}=d$. We will consider only permutations $\pi$ such that
	\begin{equation}\label{eq.pi_j_n}
		\pi(k) = \RHScase{d, & k=1\\ 2j_\alpha + n_\alpha -1-k, & j_\alpha\leq k<j_\alpha+ n_\alpha \\ 1, & k=d.}
	\end{equation}
	for some $\{j_\alpha,n_\alpha\}_{\alpha=1}^m$ satisying Equation \eqref{eq.j_and_n}. In other words, $\pi$ is a standard permutaion such that places blocks of letters in reverse order, and each block is of size at most $5$. For example, the permutation from the previous section,
	  $$ \pi = \cmtrx{1&2&3&4&5&6 \\ 6&3&2&5&4&1},$$
	satisfies \eqref{eq.pi_j_n} for $j_1=2$, $j_2=4$ and $n_1=n_2=2$.
	While this may initally seem a very special case to consider, we have the following as a result from the constructions in Section \ref{chapTrue}.
	\begin{corollary}\label{cor.n_j_perms_happen}
	  Every non-hyperelliptic Rauzy Class $\RRR$ contains an element $\pi$ that satisfies Equation \eqref{eq.pi_j_n} for some $\{j_\alpha,n_\alpha\}_{\alpha=1}^m$ satisying Equation \eqref{eq.j_and_n}.
	\end{corollary}
	\begin{proof}
	  The block constructed permutations from Section \ref{chapTrue} satisfy \eqref{eq.pi_j_n} and appear in every non-hyperelliptic Rauzy Class.
	\end{proof}
	We remark that the genus of $\pi$ is given by
	\begin{equation}\label{eq.n_and_genus}
		g(\pi) = 1+\sum_{\alpha=1}^{m} \left\lfloor \frac{n_\alpha}{2} \right\rfloor,
	\end{equation}
	and that $\{j_\alpha,n_\alpha\}_{\alpha=1}^m$ from Equation \eqref{eq.j_and_n} uniquely determines a $\pi$ satifsyting \eqref{eq.pi_j_n} and vice-versa.
	
	Consider a permutation from \eqref{eq.pi_j_n} and choose $\alpha\in\{1,\dots,m\}$. We will construct a path dependent on each $j_\alpha,n_\alpha$. We will consider the permutation $\pi' := 0^*\pi$ such that
	\begin{equation}\label{eq.pi_prime_alpha}
	  \pi'(j_\alpha)=d\mbox{  and }\pi'(j_{\alpha-1}) = d-n_\alpha.
	\end{equation}
	We construct a Rauzy Path based on $n_\alpha$ and up to $2$ positive integer paramters $a$ and $b$.

	\begin{itemize}
		\item Assume $n_\alpha=1$. We consider the path $\gamma_{j_\alpha,1} :=1^{d-j_\alpha}0$.
		The matrix associated to $\gamma_{j_\alpha,1}$ is
		$$ \Theta_{\gamma_{j_\alpha,1}} =
			\pmtrx{\mI_{j_\alpha-1} & \mtrx{0 \\ \vdots \\ 0}& \mtrx{0 & \dots & 0 \\ \vdots & \ddots & \vdots \\ 0 & \dots & 0}\\
					\mtrx{0 & \dots & 0} & \mA_1 & \mtrx{1 &\dots &1}\\
					\mtrx{0 & \dots & 0 \\ \vdots & \ddots & \vdots \\ 0 & \dots & 0\\ 0 & \dots & 0} & \mtrx{0\\ \vdots \\ 0 \\ 1} & \mI_{d-j_\alpha}},\mbox{ for } \mA_1=\pmtrx{2}.$$
					
		\item Assume $n_\alpha=2$. Define $\gamma_{j_\alpha,2,a} := 1^{d-j_\alpha-1}0^a10^2$. Compare this with the paths in Section \ref{sec.meas.exam}. The matrix for $\gamma_{j_\alpha,2,a}$ is
		$$ \Theta_{\gamma_{j_\alpha,2,a}} =
			\pmtrx{ \mI_{j_\alpha-1} & \mtrx{0&0 \\ \vdots&\vdots \\ 0&0}& \mtrx{\zeros}\\
				\mtrx{0&\dots&0 \\ 0&\dots&0} & \mA_{2,a} & \mtrx{1&\dots&1 \\ 0&\dots&0}\\
				\mtrx{\zeros \\ 0 & \dots & 0} & \mtrx{0&0\\ \vdots&\vdots \\ 0&0 \\ 1&1} & \mI_{d-j_\alpha-1}}, \mbox{ for }\mA_{2,a} = \pmtrx{2 & 2 \\ a & a+1}.$$
					
		\item Assume $n_\alpha=3$. Define $\gamma_{j_\alpha,3,a}:=1^{d-j_\alpha-2}01^a01^20^3$. The matrix in this case is
		$$ \begin{array}{rcl}
			\Theta_{\gamma_{j_\alpha,3,a}} &=&
				\pmtrx{ \mI_{j_\alpha-1} & \mtrx{0&0&0\\ \vdots&\vdots&\vdots \\ 0&0&0} & \mtrx{\zeros}\\
				\mtrx{0&\dots&0 \\ 0&\dots&0 \\ 0&\dots&0} & \mA_{3,a} & \mtrx{1&\dots&1 \\ 0&\dots&0 \\ 0&\dots&0}\\
				\mtrx{\zeros \\ 0&\dots&0} & \mtrx{0&0&0 \\ \vdots&\vdots&\vdots \\ 0&0&0 \\ 1&1&1} & \mI_{d-j_\alpha-2}},\\
			\mbox{where} &\mA_{3,a}& =~\pmtrx{2 & 2 & 2 \\ 0 & a+1 & a \\ 1 & 2& 2}.
			\end{array}$$
					
		\item Assume $n_\alpha=4$. Then let $\gamma_{j_\alpha,4,a,b}:=1^{d-j_\alpha-3}0^21^b01^20^a10^4$ with matrix
		$$ \begin{array}{rcl}
			\Theta_{\gamma_{j_\alpha,4,a,b}} &=&
			\pmtrx{ \mI_{j_\alpha-1} & \mtrx{\zeros}& \mtrx{\zeros}\\
					\mtrx{0 & \dots & 0 \\ \zeros} & \mA_{4,a,b} & \mtrx{1 &\dots &1\\ \zeros}\\
					\mtrx{\zeros \\ 0 & \dots & 0} & \mtrx{\zeros \\ 1&\dots &1} & \mI_{d-j_\alpha-3}	},\\
				\mbox{where} &\mA_{4,a,b}& =~\pmtrx{2 & 2 & 2 & 2 \\
				a & a+1 & 0 & 0 \\
				0 & 0 & b+1 & b \\
				a+1 & a+2 & 2 & 2}.
			\end{array}$$

		\item Assume $n_\alpha=5$. Let $\gamma_{j_\alpha,5,a,b}:=1^{d-j_i-4}0^{2}1^{2b}0101^30^{a}10^5$. Note that its associated matrix is
		$$ \Theta_{\gamma_{j_i,5,a,b}} =
			\pmtrx{ \mI_{j_\alpha-1} & \mtrx{\zeros}& \mtrx{\zeros}\\
					\mtrx{0 & \dots & 0 \\ \zeros} & \mA_{5,a,b} & \mtrx{1 &\dots &1\\ \zeros}\\
					\mtrx{\zeros \\ 0 & \dots & 0} & \mtrx{\zeros \\ 1&\dots &1} & \mI_{d-j_\alpha-4}	},$$
		$$ \mbox{where }\mA_{5,a,b} = \pmtrx{2 & 2 & 2 & 2 & 2 \\
				a & a+1 & 0 & 0 & 0 \\
				0 & 0 & b+1 & 3b & 2b \\
				0 & 0 & 0 & 2 & 1 \\
				a+1 & a+2 & 2 & 2 & 2}.$$
	\end{itemize}
	In any case, the resulting permutation $\pi''=\gamma_{j_\alpha,n_\alpha,(a,b)}\pi'$ is such that $\pi'' = \pi$ if $\alpha = 1$ or $\pi''$ satisfies \eqref{eq.pi_prime_alpha} for $\alpha-1$. So if we initally consider $\pi'=0\pi$, $\alpha=m$ and each defined path will decrease $\alpha$ by $1$ until we return to $\pi$. This therefore forms a closed Rauzy Path on $\pi$.
	
	So consider $\pi\in\irr_d$ and $\{j_\alpha,n_\alpha\}_{\alpha=1}^{k}$ satisfying \eqref{eq.j_and_n} and \eqref{eq.pi_j_n} and integers $a,b,c>0$. Define path
	\begin{equation}\label{eq.gamma}
		\gamma^{\pi,a,b,c} := 0 \gamma_{j_m,n_m,a,b} \gamma_{j_{m-1},n_{m-1},a,b} \cdots \gamma_{j_2,n_2,a,b} \gamma_{j_1,n_1,a,b} 1^{c(d-1)}.
	\end{equation}
	It then follows that
	\begin{equation}
	\Theta_{\gamma^{\pi,a,b,c}} =
		\pmtrx{1 & \mtrx{c & \dots & \msp} & \msp & \msp & \mtrx{\msp & \msp & \dots} & c\\
		\mtrx{0\\ \vdots \\ \msp \\ \msp} & \mA_{n_1,a,b} & \mtrx{1 &\dots & \msp\\ 0 & \dots &\msp \\ \vdots & \msp & \msp\\ 0 & \dots& \msp} & \msp & \mtrx{\msp & \msp & \dots \\ \msp & \msp & \dots \\ \msp & \msp & \msp \\ \msp & \msp & \dots} & \mtrx{1 \\ 0 \\ \vdots \\ 0} \\ 
		\msp &  \mtrx{1 &\dots & 1\\ 0 & \dots &0 \\ \vdots & \msp & \vdots\\ 0 & \dots& 0} & \mA_{n_2,a,b} & \mtrx{1 &\dots & \msp\\ 0 & \dots &\msp \\ \vdots & \msp & \msp\\ 0 & \dots& \msp}  & \mtrx{\msp & \msp & \dots \\ \msp & \msp & \dots \\ \msp & \msp & \msp \\ \msp & \msp & \dots} & \mtrx{1 \\ 0 \\ \vdots \\ 0} \\ 
		\msp & \mtrx{\msp \\ \vdots \\ \msp} & \msp & \ddots & \msp & \vdots \\
		\mtrx{\msp \\ \msp \\ \vdots \\ 0}& \mtrx{1 &\dots & \msp\\ 0 & \dots &\msp \\ \vdots & \msp & \msp\\ 0 & \dots& \msp}& \msp & \mtrx{\msp &\dots & 1\\ \msp & \dots &0 \\ \msp & \msp & \vdots\\ \msp & \dots& 0}& \mA_{n_k,a,b}& \mtrx{1 \\ 0 \\ \vdots \\ 0}\\
		1 & \mtrx{ c+1 & \dots & \msp} & \msp& \msp& \mtrx{\msp & \msp & \dots}& c+1}.
	\end{equation}

	\begin{defn}
	    If we consider sequences of positive inetegers $\ma = \{a_i\}_{i>0}$, $\mb = \{b_i\}_{i>0}$ and $\mc = \{c_i\}_{i>0}$ and permutation $\pi$ satisfying Equations \eqref{eq.j_and_n} and \eqref{eq.pi_j_n}, let the infinite path $\gamma^{\pi,\ma,\mb,\mc}$ be the concatenation of the paths $\gamma^{\pi,a_i,b_i,c_i}$ or
	      $$ \gamma^{\pi,\ma,\mb,\mc} = \gamma^{\pi,a_1,b_1,c_1}\gamma^{\pi,a_2,b_2,c_2}\gamma^{\pi,a_3,b_3,c_3}\dots$$
	\end{defn}

	\begin{rem}
	    Such sequences $\gamma=\gamma^{\pi,\ma,\mb,\mc}$ are complete (Definition \ref{def.complete_path}). This follows as for each $i>0$ and $\beta\in\AAA$, $\beta$ wins at least once in the subpath $\gamma^{\pi,a_i,b_i,c_i}$. Therefore, there exists at least one $\lambda\in\Del_\AAA$ such that $\gamma$ is the Rauzy Path associated to $T=(\pi,\lambda)$.
	\end{rem}

	\begin{rem}
	    It is an immediate consequence that if $\ma$, $\mb$ and $\mc$ are universally bounded sequences, then $T=(\pi,\lambda)$ is uniquely ergodic.
	\end{rem}

	We make the following claim without general proof. However, in the following section, we shall prove the claim for a specific type of sequences.

	\begin{claim}
	    Suppose $\pi$ satisfies Equations \eqref{eq.j_and_n} and \eqref{eq.pi_j_n} and the sequences $\ma$, $\mb$, $\mc$ satisfy
		$$ a_i \gg b_i \gg c_i \gg a_{i-1}$$
	    for all large $i$ and a suitable definition of ``$\gg$.'' Then for $\gamma = \gamma^{\pi,\ma,\mb,\mc}$ and every $\lambda\in\Del(\gamma)$, the \IET\ $T=(\pi,\lambda)$ is minimal and admits $g(\pi)$ distinct ergodic probabilty measures.
	\end{claim}

	Before we move to the next section, note that the statement of minimality trivially follows from the completeness of $\gamma$ (see Proposition \ref{prop.theta_and_measures} and Remark \ref{rem.Keane_minimal_cone}). Also, note that for any integers $a \gg b \gg c \gg 1$, the column vectors of $\Theta = \Theta_{\gamma^{\pi,a,b,c}}$ point in specific directions. More precisely, the first and last column will both point in one direction (concentrated in the first and last coordinates), while the columns associated with the block $\alpha$ will point in $\lfloor \frac{n_\alpha}{2}\rfloor$ other directions. So if $a,b,c$ are large, the columns of $\Theta$ will point in precisely $g(\pi)$ direction by Equation \eqref{eq.n_and_genus}. It is not clear however that the products of such matrices will exhibit the same phenomena (two of these directions may collapse into one in the limit). Proving this will be the aim of Corollary \ref{cor.meas.limits} that follows.

\subsection{\texorpdfstring{``At Least"}{"At Least"} Exponential Sequences}\label{sec.geq_exp}

	Let $\pi$ and $\{j_\alpha,n_\alpha\}_{\alpha=1}^{m}$ satisfy Equations \eqref{eq.j_and_n} and \eqref{eq.pi_j_n}. We will show that if ``$\gg$" is taken to mean
		$$ a \gg b \iff a \geq \rho b$$
	for a fixed real $\rho > 2$, then the path $\gamma$ defined by sequences $\ma, \mb, \mc$ (under this working definition of $\gg$) and $\pi$ admits \IET 's with $g(\pi)$ distinct ergodic invariant probability measures (see Corollary \ref{cor.meas.meas}).

	\begin{defn}
		For $r>0$, let $\bigO(r)$ denote a vector in $\left[\RR_{\geq 0}\right]^d$ such that $|\bigO(r)| \leq r$ and $\bigo(r)$ a real number such that $\bigo(r)\in[0,r]$.
	\end{defn}
	
	\noindent Assuming such sequences $\ma,\mb,\mc$, we define subpaths as $\gamma_{j_\alpha,n_\alpha}^{(i)}$ by the following:
			$$ \gamma_{j_\alpha,n_\alpha}^{(i)} = \RHScase{\gamma_{j_\alpha,1}, & n_\alpha = 1 \\
				\gamma_{j_\alpha,n_\alpha,a_i}, & n_\alpha=2,3 \\
				\gamma_{j_\alpha,n_\alpha,\lceil \frac{a_i}{2}\rceil,b_i}, & n_\alpha=4,5}$$
	\noindent Now we consider the path $\gamma^{(i)} := 0\gamma_{j_m,n_m}^{(i)}\dots \gamma_{j_1,n_1}^{(i)} 1^{c_i(d-1)}$. For convenience of notation, let
			$$ \Theta_i := \Theta_{\gamma^{(i)}}.$$
	\noindent For purposes of normalization, we will perform calculations on  $\Theta_i' := \frac{1}{a_i}\Theta_i$. This will be equivalent to examining $\Theta_i$ as for any $\lambda\in \Delta_{d-1}$, $ \hat{\Theta}'_i\lambda = \hat{\Theta}_i\lambda$ and therefore
	    \begin{equation}\label{eq.Psi}
		\Psi_i \lambda := \lim_{k\to\infty}\hat{\Theta}_i\cdots \hat{\Theta}_{i+k}\lambda = \lim_{k\to\infty}\hat{\Theta}'_i\cdots \hat{\Theta}'_{i+k}\lambda
	    \end{equation}
	assuming such a limit exists. The only $\lambda$ we will consider will be the endpoints, and such limits are well defined in this case. We finally define
			$$ M_i := \sup_{j\geq i} \sup_{\lambda \in \Delta_{d-1}} |\Theta_j' \lambda| = \sup_{j\geq i}\max_{1\leq j \leq d} |\Theta_j' \me_j|$$
	where we recall that $|\lambda| = \lambda_1+\dots+\lambda_d$ for any $\lambda\in\RR_+^\AAA$.
			
	We consider the limiting vectors from the blocks related to $\{j_\alpha,n_\alpha\}_{\alpha=1}^{m}$ and the outside columns.	
	Let $\vv_0 = \ee_1+\ee_d$. We note that
        \begin{equation}
			\begin{array}{rcl}
                \Theta_i' \vv_0 &= &\frac{c_i}{a_i}\vv_0 + \bigO\left(\frac{m+3}{a_i}\right).
			\end{array}
        \end{equation}
	
	Now for any $1\leq \alpha \leq m$, assume $n_\alpha = 2$. In this case, let $\vv_\alpha := \ee_{j_\alpha + 1}$. We conclude that
		\begin{equation}
			\begin{array}{rcl}
				\Theta_i' \vv_\alpha & = & \vv_\alpha + \frac{c_i}{a_i} \vv_0 + \bigO\left(\frac{m+3}{a_i}\right).
			\end{array}
		\end{equation}
		
	\noindent Now assume $n_\alpha =3$. Then let $\vv_\alpha = \ee_{j_\alpha+1}$. We see that
		\begin{equation}
			\begin{array}{rcl}
				\Theta_i' \vv_\alpha & = & \vv_\alpha + \frac{c_i}{a_i} \vv_0 + \bigO\left( \frac{m+5}{a_i}\right).
			\end{array}
		\end{equation}
		
	\noindent If $n_\alpha = 4$, let $\vv_\alpha = \ee_{j_\alpha+1} + \ee_{j_\alpha+3}$ and $\ww_\alpha=\ee_{j_\alpha+2}$. It follows that
		\begin{equation}
			\begin{array}{rcl}
				\Theta_i' \vv_\alpha & = & \frac{1}{2} \vv_\alpha + \frac{b_i}{a_i} \ww_\alpha + 2\frac{c_i}{a_i} \vv_0 + \underset{\msp}{\bigO\left(2\frac{m+6}{a_i}\right)} \\
				\Theta_i' \ww_\alpha & = & \frac{b_i}{a_i} \ww_\alpha + \frac{c_i}{a_i} \vv_0 + \bigO\left(\frac{m+5}{a_i}\right).
			\end{array}
		\end{equation}
		
	\noindent If $n_\alpha = 5$, let $\vv_\alpha = \ee_{j_\alpha+1} + \ee_{j_\alpha+4}$ and $\ww_\alpha=\ee_{j_\alpha+2}$. We verify the following,
		\begin{equation}\label{eq.n_5_v_and_w}
			\begin{array}{rcl}
				\Theta_i' \vv_\alpha & = & \frac{1}{2} \vv_\alpha + 2\frac{b_i}{a_i} \ww_\alpha + 2\frac{c_i}{a_i} \vv_0 + \underset{\msp}{\bigO\left(2\frac{m+6}{a_i}\right)} \\
				\Theta_i' \ww_\alpha & = & \frac{b_i}{a_i} \ww_\alpha + \frac{c_i}{a_i} \vv_0 + \bigO\left(\frac{m+5}{a_i}\right).
			\end{array}
		\end{equation}
		
		\begin{defn}
			For real sequences $\ma,\mb,\mc$ and integers $i,j$, let
				$$ \duS{i}{j}(\ma,\mb,\mc) := \dusum{\ell=i}{j}\left[\left(\duprod{r=i}{\ell-1} a_r\right) b_\ell \left( \duprod{r=\ell+1}{j} c_r\right)\right].$$
			If $\md,\me,\mf,\mg$ are also real sequences, then let
				$$ \begin{array}{rcl}
					\duS{i}{j}(\ma,\mb,\mc,\md,\me) &:=& \underset{\msp}{\dusum{\ell=i}{j-1}} \left[\left(\duprod{r=i}{\ell-1}a_r\right) b_\ell \duS{\ell+1}{j}(\mc,\md,\me) \right]\mbox{ and}\\
					\duS{i}{j}(\ma,\mb,\mc,\md,\me,\mf,\mg) &:=& {\dusum{\ell=i}{j-2}} \left[\left(\duprod{r=i}{\ell-1}a_r\right) b_\ell \duS{\ell+1}{j}(\mc,\md,\me,\mf,\mg) \right].
					\end{array}$$
		\end{defn}
		
		\begin{lem}\label{l.linear_ops}
			Let $\VVV$ be a vector space, $\mA,\mB,\mC,\mD\in\VVV$ and $\ma,\mb,\mc,\md,\me,\mf,\mmu,\mv,\mx,\my$ be real sequences. Suppose $\{F_i\}_{i>0}$ is a sequence of linear operators such that
			$$ \begin{array}{rcl}
				F_i \mA &=& a_i\mA,\\ 
				F_i \mB &=& b_i\mB + c_i\mA, \\
				F_i \mC &=& d_i\mC + e_i\mB + f_i\mA,\\
				F_i \mD &=& u_i\mD + v_i\mC + x_i\mB + y_i\mA.
				\end{array}$$
			
			Then for each $k\geq 0$
			$$ \begin{array}{rl}
				F_i\circ\cdots\circ F_{i+k}\mA & = \underset{\msp}{\left(\duprod{\ell=i}{i+k}a_\ell\right)}\mA, \\
				
				F_i\circ\cdots\circ F_{i+k}\mB & = \underset{\msp}{\left(\duprod{\ell=i}{i+k}b_\ell\right)}\mB + \duS{i}{i+k}(\ma,\mc,\mb)\mA, \\
				
				F_i\circ\cdots\circ F_{i+k}\mC &= \underset{\msp}{\left(\duprod{\ell=i}{i+k}d_\ell\right)}\mC + \duS{i}{i+k}(\mb,\me,\md)\mB \\
				
					&+  \underset{\msp}{\left[ \duS{i}{i+k}(\ma,\mf,\md) + \duS{i}{i+k}(\ma,\mc,\mb,\me,\md)\right]}\mA, \\
					
				F_i\circ\cdots\circ F_{i+k}\mD &= \underset{\msp}{\left(\duprod{\ell=i}{i+k}u_\ell\right)}\mD + \duS{i}{i+k}(\md,\mv,\mmu)\mC  \\
				
					&+ \underset{\msp}{\left[ \duS{i}{i+k}(\mb,\mx,\mmu) + \duS{i}{i+k}(\mb,\me,\md,\mv,\mmu)\right]}\mB \\
					
					& + \underset{\msp}{\left[ \duS{i}{i+k}(\ma,\my,\mmu) + \duS{i}{i+k}(\ma,\mf,\md,\mv,\mmu)\right.} \\
					
					& + \underset{\msp}{\left.\duS{i}{i+k}(\ma,\mc,\mb,\mx,\mmu) + \duS{i}{i+k}(\ma,\mc,\mb,\me,\md,\mv,\mmu)\right]}\mA . \end{array}$$
		\end{lem}
		
		\begin{corollary}\label{cor.meas.limits}
			Let $\pi$ and $\{j_\alpha,n_\alpha\}_{\alpha=1}^{m}$ satisfy \eqref{eq.j_and_n} and \eqref{eq.pi_j_n}. Suppose positive integer sequences $\ma,\mb,\mc$ satisfy
			$$ \frac{c_i}{a_{i-1}},\frac{a_i}{b_i},\frac{b_i}{c_i} \geq \rho$$
			for real $\rho > 2$ and all $i \geq i_0$. Also, assume that
			$$ M' := \max\left(1+\frac{2}{\rho^2}+\frac{m+7}{a_{i_0}},\frac{3}{\rho}+\frac{2}{\rho^2} +\frac{m+6}{a_{i_0}}\right)<2.$$
			If $\{\Theta_i\}_{i>0}$ are defined by $\gamma_{\pi,\ma,\mb,\mc}$, then for $i \geq i_0$,
			$$ \begin{array} {rclr}
				M_i & \leq & \underset{\msp}{M'},& \\
				\Psi_i \vv_0 &\sim & \vv_0 + \underset{\msp}{\bigO\left(\frac{K_1}{a_{i-1}}\right)}, & \\
				\Psi_i \vv_\alpha & \sim  & \vv_\alpha + \underset{\msp}{\bigo\left(\frac{1}{\rho^2-1}\right)}\vv_0 + \bigO\left(\frac{K_2}{a_i}\right), & \mbox{for } n_\alpha = 2 \\
				\Psi_i \vv_\alpha & \sim & \vv_\alpha + \underset{\msp}{\bigo\left(\frac{1}{\rho^2-1}\right)}\vv_0 + \bigO\left(\frac{K_3}{a_i}\right), & \mbox{for } n_\alpha=3\\
				\Psi_i \vv_\alpha & \sim & \vv_\alpha + \underset{\msp}{\bigo\left(\frac{2}{\rho-2}\right)}\ww_\alpha + \bigo\left(\frac{4(\rho-1)}{(\rho-2)(\rho^2-2)}\right)\vv_0+ \bigO\left(\frac{K_4}{a_i}\right), & \mbox{for } n_\alpha=4 \\
				\Psi_i \vv_\alpha & \sim & \vv_\alpha + \underset{\msp}{\bigo\left(\frac{4}{\rho-2}\right)}\ww_\alpha + \bigo\left(\frac{4\rho}{(\rho-2)(\rho^2-2)}\right)\vv_0+ \bigO\left(\frac{K_5}{a_i}\right), & \mbox{for } n_\alpha=5 \\
				\Psi_i \ww_\alpha & \sim & \ww_\alpha + \underset{\msp}{\bigo\left(\frac{1}{\rho-1}\right)}\vv_0 + \bigO\left(\frac{K_6}{a_{i-1}}\right), & \mbox{for } n_\alpha\geq 4
				\end{array}$$
		where $K_1,\dots,K_6$ are constants dependent only on $\rho$, $M_{i_0}'$, and $m$. Here $\ww\sim\vv$ is defined as $\vv = c\ww$ for some $c>0$.
		\end{corollary}

		\begin{proof}
		    By direct consideration of each column of our matrices $\Theta_i'$, it follows that $M_i<M'$ for each $i>i_0$. As a result,
			$$ \Theta'_i \bigO(1) = M'\bigO(1).$$
		    We then use Lemma \ref{l.linear_ops} on Equations \eqref{eq.Psi}-\eqref{eq.n_5_v_and_w}.
		\end{proof}
		
		\begin{corollary}\label{cor.meas.meas}
			If $\pi,\{j_\alpha,n_\alpha\}_{\alpha=1}^{m}$ and sequences $\ma, \mb, \mc$ are as in Corollary \ref{cor.meas.limits}, then
			for any $\lambda\in\Del(\gamma)$, the \IET\ $T\sim(\pi,\lambda)$ admits $g(\pi)$ ergodic invariant probability measures.
		\end{corollary}

		\begin{proof}
		    Let $\gamma$ be the path defined by $\pi$ and $\ma,\mb,\mc$. Because $\rho>2$, we may assume $M'<2$ if we choose sufficiently large $i_0$. By the previous corollary, $\Del(\gamma')$ has $g(\pi)$ vertices where $\gamma'$ is the infinite Rauzy Path beginning at step $i_0$ rather than $1$. Because
		      $$ \Del(\gamma) = \hat\Theta_1\cdots\hat\Theta_{i_0-1} \Del(\gamma'),$$
		    the simplex $\Del(\gamma)$ must also have $g(\pi)$ vertices. By Theorem \ref{thm.LVector_is_Measure}, the vertices relate to $g(\pi)$ distinct ergodic probability measures for any $T=(\pi,\lambda)$, $\lambda\in\Del(\gamma)$ (or $\lambda\in\Lambda(\gamma)$ even).
		\end{proof}

\bibliographystyle{abbrv}
\bibliography{../../bibfile}

\end{document}